\documentclass[tikz,border=10pt]{article}
\UseRawInputEncoding
\usepackage[english]{babel}
\usepackage[utf8x]{inputenc}
\usepackage[T1]{fontenc}

\usepackage{tikz}
\usepackage{pgflibraryarrows}
\usepackage{pgflibrarysnakes}
\usetikzlibrary{decorations.markings}

\usepackage[a4paper,top=3cm,bottom=2cm,left=3cm,right=3cm,marginparwidth=1.75cm]{geometry}

\usepackage{amsmath}
\usepackage{graphicx}
\usepackage[colorinlistoftodos]{todonotes}
\usepackage[colorlinks=true, allcolors=blue]{hyperref} \usepackage{mathrsfs}
\usepackage{amssymb}
\usepackage{lineno}
\usepackage{bbding}
\usepackage{fancyhdr}\usepackage{graphics}
\usepackage{titlesec}\usepackage{titletoc}
\usepackage{indentfirst}
\usepackage{amsmath}\usepackage{amsfonts}\usepackage{latexsym}
\usepackage{eucal}\usepackage{epsfig}\usepackage{mathrsfs,amsmath,amsthm}
\usepackage{geometry}

\newtheorem{theorem}{Theorem}[section]
 
 \newtheorem{proposition}[theorem]{Proposition}
 
 \newtheorem{definition}[theorem]{Definition}

 \newcount\refno
\def\CC{{\mathbb C}}

 \def\RR{{\mathbb R}}
 \def\NN{{\mathbb N}}

 \def\ZZ{{\mathbb Z}}

 \def\SF{{\mathscr F}}
 
 \title{\bf Hardy spaces  associated with  One-dimensional Dunkl transform for  $\frac{2\lambda}{2\lambda+1}<p\leq1$
 \thanks{1} \footnote{E-mail:
huzhuoran010@163.com[ZhuoRan Hu].}}

\author{{ZhuoRan Hu}\\
{\small }\\
{\small Beijing 100048, China}}

\begin{document}
\maketitle \setcounter{page}{1} \pagestyle{myheadings}
 \markboth{Hu}{Hardy spaces  associated with One-dimensional Dunkl transform for  $\frac{2\lambda}{2\lambda+1}<p\leq1$}
\begin{abstract}
This paper mainly contains two parts.  In the first part, we will characterize the Homogeneous Hardy spaces on the real line  by  a  kernel with a compact support  for $\frac{1}{1+\gamma}<p\leq1$ where $0<\gamma\leq1$.

In the second part of this paper,  we will study the Hardy spaces associated with  One-Dimensional Dunkl transform. The usual analytic function is replaced  by the $\lambda$-analytic function which is based upon the $\lambda$-Cauchy-Riemann equations: $D_xu-\partial_y v=0, \partial_y u +D_xv=0$, where $D_x$ is the Dunkl operator:$D_xf(x)=f'(x)+\frac{\lambda}{x}[f(x)-f(-x)].$  The real characterization of  the Complex-Hardy Spaces $H_{\lambda}^p(\RR_+^2)$   will be obtained for $p>\frac{2\lambda}{2\lambda+1}$. We will also prove that the Real Hardy spaces $H_{\lambda}^p(\RR)$ is  Homogeneous Hardy spaces  for $\frac{1}{1+\gamma_\lambda}< p\leq1$ where  $\gamma_\lambda=\frac{1}{2(2\lambda+1)}$ ($\lambda>0$) from which we could obtain  the real-variable method of $H_{\lambda}^p(\RR)$. These results  extend  the results about the Hankel transform of Muckenhoupt and Stein in \cite{MS} to a general case and contain a number of further results.

\vskip .2in
 \noindent
 {\bf 2000 MS Classification:}
 \end{abstract}

\setcounter{page}{1}

\subsection{Introduction}
In 1965,   Muckenhoupt and Stein studied the  Hardy spaces associated with the Hankel transform in \cite{MS}.
 Their starting point is the generalized Cauchy-Riemann equations:
\begin{eqnarray}\label{lam-c-r}
u_x-v_y=0,\ \ \ u_y+v_x+\frac{2\lambda}{x}v=0
\end{eqnarray}
for functions $u(x,y)$, $v(x,y)$ on the domain $\{(x, y): x>0, y>0\}$. And  they introduced a notion of conjugacy associated with
the Bessel operators $\triangle_{B\lambda}$, $\lambda>0$, defined by
$$\triangle_{B\lambda} f(x)= -\frac{d^2}{dx^2}f(x)-\frac{2\lambda}{x}f(x),\ \ \ x>0.$$
They developed in this setting a theory parallel to the classical case associated to the Euclidean Laplacian. In \cite{MS}, definitions
of Poisson kernels, harmonic functions, conjugate functions and fractional integrals associated with $\triangle_{B\lambda}$ are given. Results
parallel to the classical case about $L^p((0,\infty), x^{2\lambda}dx)$-boundedness, $1\leq p<\infty$, for these operators were obtained. In sight of the whole half-plane
$\RR_{+}^{2} = \{(x,y): x \in \RR, y> 0\}$, the study in \cite{MS} is restricted to the case when u is even
in x and v is odd in x,  and the nonsymmetry of
u and v   lead to some ambiguous treatments in any further study.
And very  little progress has been made on  the real characterization and the  real-variable method  in \cite{MS} on the upper half plane for the case  $p<1$.

To generalize the results in \cite{ZhongKai Li 2} and \cite{MS},  the Harmonic Analysis associated with the Dunkl transform on the line is studied in \cite{ZhongKai Li 3}. The
$\lambda$-subharmonic function,  $\lambda$-Poisson integral,  conjugate $\lambda$-Poisson integral,
and the associated maximal functions are studied in \cite{ZhongKai Li 3}.
 The theory of the associated Complex-Hardy spaces $H^p_{\lambda}(\RR^2_+)$  for $p > p_0 = \frac{2\lambda}{2\lambda+1}$ in \cite{ZhongKai Li 3} extends the results of Muckenhoupt and Stein in \cite{MS}. However, it is difficult to generalize the results of $H_{\lambda}^p(D^+)$ on the disk  in \cite{ZhongKai Li 2} to the results of $H^p_{\lambda}(\RR^2_+)$ on the upper half plane in \cite{ZhongKai Li 3} for   $p > p_0 = \frac{2\lambda}{2\lambda+1}$.  Theory of the real characterization of $H^p_{\lambda}(\RR^2_+)$ and  the real-variable method of $H_{\lambda}^p(\RR)$ are still unknown in \cite{ZhongKai Li 3}. By the theory of Uchiyama's result in \cite{U}, $H_{\lambda}^p(\RR)$
is  Homogeneous Hardy spaces for $p_1<p\leq1$ (for some  $p_1$ close to 1) in \cite{JianQuan Liao 1}. In this paper, we will  give a real characterization of the $H^p_{\lambda}(\RR^2_+)$ for the range of $ p>\frac{2\lambda}{2\lambda+1}$, and we  also prove that $H_{\lambda}^p(\RR)$ is  Homogeneous Hardy spaces for the range of $1\geq p>\frac{1}{1+\gamma_\lambda}$, where $\gamma_\lambda=\frac{1}{2(2\lambda+1)}$. Thus   the real-variable method of $H_{\lambda}^p(\RR)$ could be obtained by the properties of Homogeneous Hardy spaces. These results extend the results in \cite{ZhongKai Li 2} and \cite{ZhongKai Li 3}.

For $0<p<\infty$, $L_{\lambda}^p(\RR)$ is the set of measurable functions satisfying
$ \|f\|_{L_{\lambda}^p}=\Big(c_{\lambda}\int_{\RR}|f(x)|^p|x|^{2\lambda}dx\Big)^{1/p}$ $<\infty$,
$c_{\lambda}^{-1}=2^{\lambda+1/2}\Gamma(\lambda+1/2)$,
and $p=\infty$ is the usual $L^\infty(\RR)$ space.
For $\lambda\geq0$, The Dunkl operator on the line is:
$$D_xf(x)=f'(x)+\frac{\lambda}{x}[f(x)-f(-x)]$$
involving a reflection part. The associated Fourier transfrom for the Dunkl setting for $f\in L_{\lambda}^1(\RR)$ is given by:
\begin{eqnarray}\label{fourier}
(\SF_{\lambda}f)(\xi)=c_{\lambda}\int_{\RR}f(x)E_\lambda(-ix\xi)|x|^{2\lambda}dx,\quad
\xi\in\RR , \,\,f\in L_{\lambda}^1(\RR).
\end{eqnarray}
$E_{\lambda}(-ix\xi)$ is the Dunkl kernel
$$E_{\lambda}(iz)=j_{\lambda-1/2}(z)+\frac{iz}{2\lambda+1}j_{\lambda+1/2}(z),\ \  z\in\CC$$
where $j_{\alpha}(z)$ is the normalized Bessel function
$$j_{\alpha}(z)=2^{\alpha}\Gamma(\alpha+1)\frac{J_{\alpha}(z)}{z^{\alpha}}=\Gamma(\alpha+1)\sum_{n=0}^{\infty}\frac{(-1)^n(z/2)^{2n}}{n!\Gamma(n+\alpha+1)} .$$
Since $j_{\lambda-1/2}(z)=\cos z$, $j_{\lambda+1/2}(z)=z^{-1}\sin z$, it follows that $E_0(iz)=e^{iz}$, and $\SF_{0}$ agrees with the usual Fourier transform. We assume $\lambda>0$ in
what follows.
And the associated $\lambda$-translation in Dunkl setting is
\begin{eqnarray}\label{tau}
 \tau_y f(x)=c_\lambda
     \int_{{\RR}}(\SF_{\lambda}f)(\xi)E(ix\xi)E(iy\xi)|\xi|^{2\lambda}d\xi,
     \ \ x,y\in{\RR} .
\end{eqnarray}
The $\lambda$-convolution$(f\ast_{\lambda}g)(x)$ of two appropriate functions $f$ and $g$ on $\RR$ associated to the $\lambda$-translation $\tau_t$ is defined by
$$(f\ast_{\lambda}g)(x)=c_{\lambda}\int_{\RR}f(t)\tau_{x}g(-t)|t|^{2\lambda}dt.$$
The "Laplace Equation" associated with the Dunkl setting is given by:
$$(\triangle_{\lambda}u)(x, y)=\left(D_x^2+ \partial_y^2\right) u(x, y)=\left(\partial_x^2+ \partial_y^2\right)u+ \frac{\lambda}{x}\partial_xu-\frac{\lambda}{x^2}\left(u(x, y)-u(-x, y)\right).$$
A $C^2$ function $u(x, y)$ satisfying $\triangle_{\lambda}u=0$ is  $\lambda$-harmonic.
When u and v are $\lambda$-harmonic functions satisfying $\lambda$-Cauchy-Riemann equations:
\begin{eqnarray}\label{a c r0}
\left\{\begin{array}{ll}
                                    D_xu-\partial_y v=0,&  \\
                                    \partial_y u +D_xv=0&
                                 \end{array}\right.
\end{eqnarray}
the function F(z)=F(x,y)=u(x,y)+iv(x,y)\,(z=x+iy)\, is a $\lambda$-analytic function.
 We define the Complex-Hardy spaces $H^p_\lambda(\RR^2_+)$ to be the set of
$\lambda$-analytic functions F=u+iv on $\RR^2_+$ satisfying
$$\|F\|_{H^p_\lambda(\RR^2_+)}=\sup\limits_{y>0}\left\{c_{\lambda}\int_{\RR}|F(x+iy)|^p|x|^{2\lambda}dx \right\}^{1/p}<\infty.$$

We use the symbol $D^+$ and $C^+$ to denote the Disk $D^+= \{(x,y)\in\RR^2,
x^2+y^2<1, y>0\}$ and half plane $C^+=\{(x,y)\in R^2:
x>0,y>0\}$. In \cite{JB}, Hardy spaces associated with Bessel operator is introduced for the case $p=1$. In \cite{BDT}  the characterization  of $H_{\lambda}^1(C^+)$ of maximal functions and atomic decomposition could be obtained by the theory in \cite{U}. In \cite{ZhongKai Li 2}, the Complex-Hardy spaces associated with the Dunkl setting on the Disk  $H_{\lambda}^p(D^+)$ have been studied for the range of  $\frac{2\lambda}{2\lambda+1}<p\leq1$. In \cite{MS2} the Homogeneous  Hardy spaces  could be characterized by  atoms for $\frac{1}{1+\gamma}<p\leq1$.  In \cite{Ho}  the real-variable theory of Homogeneous Hardy spaces  is studied by the way of Littlewood–Paley function for $p\in(\omega/(\omega+\eta), 1]$. In \cite{J.P. Anker}, the Real-Hardy spaces $H^1$ in high dimensions  have been studied. In \cite{J.P. Anker2}, the Complex-Hardy spaces  in the rational Dunkl setting $H^1$ in high dimensions  have been studied.
The following is the main structure of this paper:

\textbf{\textsl{b. Summary of Section 1.}}

In Section 1, we will characterize the Homogeneous  Hardy spaces by a kernel.   The theory of  $H_{\mu\gamma}^p(\RR)$ is studied when $\frac{1}{1+\gamma}<p\leq1$ with $0<\gamma<1$ by  \cite{MS2}. However, we will use a way different to \cite{MS2} to characterize the Homogeneous Hardy spaces $H_{\mu\beta}^p(\RR)$ when $0<p\leq1$, with $\beta>p^{-1}-1$  in Theorem\,\ref{H spa1}. For any $f\in A^{n, p}(\RR)$ and $n\geq[p^{-1}-1]$, we could obtain $$A^{n, p}(\RR)= H_{\mu\beta_2}^p(\RR)=H_{\mu\beta_1}^p(\RR),\ \hbox{for}\,\beta_1,\beta_2>p^{-1}-1$$
$$\|f\|_{A^{n, p}(\RR)}\sim\|f\|_{H_{\mu\beta_1}^p(\RR)}\sim\|f\|_{H_{\mu\beta_2}^p(\RR)}.$$

 Kernel is introduced in \cite{U} to characterize the Homogeneous Hardy spaces. Let $X$ be a topological space, $\rho$   a quasi-distance and $\mu$  a Borel doubling measure on $X$, then Hardy spaces $H^p(X)$ associated to this type $(X, \rho, \mu)$ is investigated in a series of studies. $H^p(X)$ becomes trival when $p$ is near to 1.
Let$$F(r, x, f)=\int_{X}K(r, x, y)f(y)d\mu(y)/r,\ \ \ f^\times(x)=\sup_{r>0}|F(r, x, f)| $$
where $K(r, x, y)$ is a kind of nonnegative function on $X\times X$ enjoying several properties. Uchiyama showed that for $1-p>0$ small enough, the maximal function $f^\times(x)$
can be used to characterize the atomic Hardy spaces $H^p(X)$.
\begin{theorem}\cite{U}\label{kk} $\exists p_1$ with $1\geq p_1 $,  such that the following inequality holds:
$$\|f_{\gamma}^*\|_{L^p(X, \mu)}\leq c_1\|f_1^\times\|_{L^p(X, \mu)}\ \ \hbox{for}\ p> p_1 $$
$c_1$ is a constant depending only on $X$ and $p$, $1\geq\gamma>0$.
\end{theorem}

Notice that the  topological space $X$ of Real-Hardy spaces $H_{\lambda}^p(\RR)$ is  $\RR$. Thus we will extend  Uchiyama's result in \cite{U} from $p_1<p\leq1$ (for some  $p_1$ close to 1) to the range $\frac{1}{1+\gamma}<p\leq1$ $(0<\gamma<1)$ when the topological space $X$ is $\RR$ with a quasi-distance $\rho$.
 Then  we will  obtain Theorem\,\ref{important}:   the maximal function $f_1^\times(x)$
can be used to  characterize the atomic Hardy spaces $H^p_{\mu}(\RR)$: for  $f\in S'(\RR, d_\mu x)$, $\frac{1}{1+\gamma}<p\leq1$, $(0<\gamma<1)$
$$\|f_{\gamma}^*(x)\|_{L^p(\RR, \mu)}\sim\| f_1^\times(x)\|_{L^p(\RR, \mu)}\sim\| f_{1\bigtriangledown_{\gamma}}^\times(x)\|_{L^p(\RR, \mu)}.$$
where the kernels satisfy Definition\,\ref{us14}.
(We do not need the kernels $K_1(r, x, y)$ or $K_2(r, x, y)$  to be continuous on $r$ variable.)

\textbf{\textsl{c. Summary of Section 2.}}
Section 2 mainly deals with the real characterization of $H^p_{\lambda}(\RR^2_+)$ and  the real-variable method of $H_{\lambda}^p(\RR)$. One of our  results is that we will prove   Theorem\,\ref{u2}  in \S\ref{area}.  We will use another way different from Burkholder-Gundy-Silverstein theorem in\,\cite{BGS}.

Then we will characterize the Real-Hardy spaces $H_{\lambda}^p(\RR)$  by Definition\,\ref{o1} and Theorem\,\ref{tan12}.  The relation of Complex-Hardy spaces $H_{\lambda}^p(\RR_+^2)$,  Real-Hardy spaces $H_{\lambda}^p(\RR)$ and Homogeneous Hardy spaces is characterized by Definition\,\ref{o1}, Theorem\,\ref{tan12} and Proposition\,\ref{s3}.

In \S\ref{homo}, the $\lambda$-Poisson kernel is introduced.   We will prove   that the Real-Hardy spaces $H^p_{\lambda}(\RR)$ is a kind of Homogeneous Hardy spaces for $ \frac{1}{1+\gamma_\lambda}<p\leq1$ in Theorem\,\ref{tan12}.  Thus  the $H_{\lambda}^p(\RR)$   can be characterized by the maximal functions in Homogeneous Hardy spaces, and the definition of $H_{\lambda}^p(\RR)$  can be evolved from the properties of $\lambda$-analytic functions.

\textbf{\textsl{Main Result}}
The main result of this paper is Theorem\,\ref{u2} and Theorem\,\ref{tan12}. By Theorem\,\ref{u2}, we could know that $H^p_{\lambda}(\RR_+^2)$
can be characterized by $H_{\lambda}^p(\RR)$ for $ \frac{2\lambda}{2\lambda+1}<p\leq1$. By Theorem\,\ref{tan12},  $H_{\lambda}^p(\RR)$ is Homogeneous Hardy spaces for $ \frac{1}{1+\gamma_\lambda}<p\leq1$. The Homogeneous Hardy spaces have many good properties including atomic decomposition.

\textbf{\textsl{e. Notation.}}
 Let $S(\RR, dx)$  the space of  $C^{\infty}$ functions on $\RR$ with the Euclidean distance rapidly decreasing together with their derivatives(Classic Schwartz Class), $L_{\lambda,{\rm loc}}(\RR)$  the set of locally integrable functions on $\RR$ associated with the measure $|x|^{2\lambda}dx$.  $\SF_{\lambda}$ is the Dunkl transform and $\SF$ the Fourier transform.

We use $A\lesssim B$ to denote the estimate $|A|\leq CB$ for some  absolute universal constant $C>0$, which may vary from line to line,
$A\gtrsim B$ to denote the estimate $|A|\geq CB$ for some absolute universal constant $C>0$, $A\sim B$  to denote the estimate $|A|\leq C_1B$, $|A|\geq C_2B$ for some absolute universal constant $C_1, C_2$.

We use  $B(x_0, r_0)$ or $B_{\lambda}(x_0, r_0)$  to denote the ball in the homogenous space in the Dunkl setting: $B(x_0, r_0)=B_{\lambda}(x_0, r_0)=\{y: d_{\lambda}(y, x_0)<r_0\},$
$d_{\lambda}(x, y)$ to denote the distance in the homogeneous space associated with Dunkl setting: $d_{\lambda}(x, y)=\left|(2\lambda+1)\int_y^x|t|^{2\lambda}dt\right|$,  $p_0$  to denote $p_0 = \frac{2\lambda}{2\lambda+1}$, $\Omega$ to denote a domain and $\partial\Omega$ to denote the boundary of $\Omega$, $\gamma_\lambda$ to denote $\gamma_\lambda=\frac{1}{2(2\lambda+1)}$,
$d_{\mu}(x, y)$ to denote the distance in the homogeneous space associated with a positive Radon measure $\mu$ on
the real line satisfying $\mu\left(x, y\right)=\int_y^x d\mu(t)$ and $d_\mu(x, y)= |\mu(x, y)|$,  $B_\mu(x_0, r_0)$ to denote the ball in the homogenous space: $B_\mu(x_0, r_0)=\{y: d_{\mu}(y, x_0)<r_0\}.$ For a measurable set $E\subseteq \RR$, we use $E^c$ to denote  the set $E^c=\{x\in\RR: x\notin E\}.$ For two sets A and B, $A\backslash B$ means that $A\bigcap B^c.$ Thoughout this paper, we assume $\lambda>0$ and $0<\gamma\leq1$.  In section\,\ref{sss5}, $\psi_{t}(x)$ denotes
$$\psi_{t}(x)=\left(\frac{1}{t}\right)^{2\lambda+1}\psi\left(\frac{x}{t}\right).$$

\section{ Homogeneous Hardy spaces on $\RR$ with a  kernel}\label{sss4}
In  this section\,\ref{sss4}, we will characterize the Homogeneous Hardy spaces on the real line by a kernel. We will extend the Uchiyama's result in \cite{U}  when the topological space $X$ is $\RR$ with a quasi-distance $\rho$.
\begin{definition}[$\mathbf{d_\mu(x, y)}$]
 $d_\mu(x, y)$ is a quasi-distance on the real line $\RR$ endowed with a positive Radon measure $\mu$, $\mu\left(x, y\right)=\int_y^x d\mu(t)$, $d_\mu(x, y)= |\mu(x, y)|$ , satisfying the following conditions (for some fixed constant $A>0$):\\
{\rm(i)} \   $d_\mu(x, y)=d_\mu(y, x)$ , for any $x, y\in \RR$; \\
{\rm(ii)}\  $d_\mu(x, y)>0$ , if $x\neq y$; \\
{\rm(iii)}\ $d_\mu(x, z)\leq A\left(d_\mu(x, y)+d_\mu(y, z)\right)$, for any $x, y, z \in \RR$ \\
{\rm(iv)}\  $A^{-1}r\leq\mu \left(B_\mu(x, r)\right)\leq r$, for any $r>0$.  \\
{\rm(v)}\  $B_\mu(x, r)=\left\{y\in \RR: d_\mu(x, y)<r\right\}$ form a basis of open neighbourhoods of the point x.\\
{\rm(vi)}\  $f(u)=\mu(x, u)$ is a continuous bijection on $\RR$ for any fixed $ x\in\RR$.
\end{definition}

\begin{definition}[$\mathbf{S(\RR, d_{\mu}x)}$]
The derivative associated with the  quasi-distance $d_\mu(x, y)$  is defined as follows:
$$\frac{d}{d_\mu x}\phi(x)=\lim_{\varepsilon\rightarrow0, d_\mu(x, y)<\varepsilon}\frac{\phi(y)-\phi(x)}{\mu(y, x)}.$$
Then the  Schwartz Class S associated with the  quasi-distance $d_\mu(x, y)$ could be defined as:
$$\|\phi\|_{(\alpha, \beta)_{\mu}}=\sup_{x\in\RR}\left|\left(d_\mu(x, 0)\right)^{\alpha}\left(\frac{d}{d_\mu x}\right)^{\beta}\phi(x)\right|<\infty$$
for  natural numbers $\alpha$ and $\beta$. This kind of Schwartz Class is denoted as $S(\RR, d_{\mu}x)$.

$\phi(u)\in C(\RR, dx)$ means $\phi(u) \rightarrow \phi(u_0)$ as $u \rightarrow u_0$ in Euclid space, $\phi(u)\in C(\RR, d_\mu x)$ means $\phi(u) \rightarrow \phi(u_0)$ as $d_\mu(u, u_0)\rightarrow 0$.
\end{definition}
\begin{proposition}\label{sk6}
For any $\phi \in S(\RR, d_\mu x)$  with $supp\,\phi(u)\subset B_{\mu}(x_0, r_0)$, there exists $\psi(t)\in S(\RR, dt)$ with $supp\,\psi(t)\subseteq [-1, 1]$  satisfying
$\psi\left(\frac{\mu(x_0, u)}{r_0} \right)=\phi(u)$ for $u\in B_{\mu}(x_0, r_0)$ in $S(\RR, d_\mu x)$ space.
\end{proposition}
\begin{proof}
Let $f(u)=\frac{\mu(x_0, u)}{r_0} $ for fixed $x_0\in\RR$ and $r_0>0$. Thus $f(u)$ is a  bijection and  has an inverse function. Let $g(x)$ to be the  inverse function of $f(x)$: $g\circ f(u)=u.$
Thus for any $ \phi \in S(\RR, d_\mu x)$, we could write $ \phi $ as:
$$\phi(u)=\phi \left(g\circ f(u)\right)=\phi \circ  g\left(\frac{\mu(x_0, u)}{r_0} \right).$$
 We use $\psi$ to denote $\psi=\phi \circ g$ and $\displaystyle{\psi^{(n)}(t)}$  to denote $\displaystyle{\psi^{(n)}(t)=\frac{d^n}{dt^n}\psi(t)}$. Then we could deduce that:
\begin{eqnarray*}
\frac{d}{d_\mu x}\phi(x)&=&\lim_{\varepsilon\rightarrow0, d_\mu(x, y)<\varepsilon}\frac{\phi(y)-\phi(x)}{\mu(y, x)}
\\&=&\lim_{\varepsilon\rightarrow0, d_\mu(x, y)<\varepsilon}-\frac{1}{r_0}\frac{\psi\left(\frac{\mu(x_0, y)}{r_0}\right)-\psi\left(\frac{\mu(x_0, x)}{r_0}\right)}{\frac{\mu(x_0, y)}{r_0}-\frac{\mu(x_0, x)}{r_0}}
\\&=&-\frac{1}{r_0}\psi^{(1)}\left(\frac{\mu(x_0, x)}{r_0} \right).
\end{eqnarray*}
Thus
\begin{eqnarray*}
\left(\frac{d}{d_{\mu}x}\right)^n\phi(x)=\left(-\frac{1}{r_0}\right)^n\psi^{(n)}\left(\frac{\mu(x_0, x)}{r_0} \right).
\end{eqnarray*}
Notice that $\mu$ is a bijection on $\RR$, together with the fact $\phi\in S(\RR, d_\mu  x)$, we could deduce that $\psi \in S(\RR, dx)$. This proves the proposition.
\end{proof}
In the same way as Proposition\,\ref{sk6}, we could obtain:
\begin{proposition}\label{sk66}
For any $ \phi \in C(\RR, d_\mu x)$, there exists $\psi\in C(\RR, dx)$,  satisfying
$\psi\left(\frac{\mu(x_0, u)}{r_0} \right)=\phi(u)$  in  $C(\RR, d_\mu x)$ space.
\end{proposition}
By  Proposition\,\ref{sk66}, together with the fact that $S(\RR, dx)$ is dense in $C_0(\RR, dx)$, we could know that $S(\RR, d_\mu x)$ is dense in $C_0(\RR, d_\mu x)$.

\begin{definition}[ $\mathbf{S'(\RR, d_{\mu}x)}$]\label{01}
A tempered distribution is a linear functional on $S(\RR, d_{\mu}x)$ that is continuous in the topology on $S(\RR, d_{\mu}x)$ induced by this family of seminorms. We shall refer to tempered distributions simply as distributions. Similar to the classical definition, we say a distribution $f$ is bounded if
$$\left|\int_\RR f(y) \phi(y) d\mu(y)\right| \in  L^\infty(\RR, \mu)$$
whenever  $\phi \in S(\RR, d_\mu x)$. We use  $S'(\RR, d_{\mu}x)$ to denote the bounded distributions.
\end{definition}

Then we will define the kernels  $K_1(r, x, y)$ and $K_2(r, x, y)$ as follows:
\begin{definition}[$\mathbf{kernel}\,\mathbf{K_1(r, x, y)}$]\label{us14}
For  constant $A > 0$ and
 constant $1\geq\gamma> 0$, let $K_1(r, x, y)$ be a nonnegative continuous function defined on $\RR^{+}\times \RR\times \RR$ satisfying the following conditions:

 {\rm(i)} \  $K_1(r, x, x) > 1/A, \  \hbox{for}\ \ r>0, x\in\mathbb R$;

{\rm(ii)}  \ $0\le K_1(r, x, t)\le
1, \  \hbox{for} \ \ r>0, x,
t\in\mathbb R$;

{\rm(iii)} \ For $r>0, x, t,z\in\mathbb R$
$$
| K_1(r, x, t)-K_1(r, x, z))|\le
\Big(\frac{d_\mu(t,z)}{r}\Big)^\gamma.$$

{\rm(iv)} $K_1(r, x, y)=0$, if $d_\mu(x, y)>r.$

{\rm(v)} $K_1(r, x, y)=K_1(r, y, x)$.

\end{definition}

\begin{definition}[$\mathbf{kernel}\,\mathbf{K_2(r, x, y)}$]\label{us5} For constants $C_i > 0$, $i=1, 2, 3, 4$ and
 constant $1\geq\gamma> 0$, let $K_2(r, x, y)$ be a nonnegative continuous function defined on $\RR^{+}\times \RR\times \RR$ satisfying the following conditions:

 {\rm(i)} \  $K_2(r, x, x) > C_1, \  \hbox{for}\ \ r>0, x\in\mathbb R$;

{\rm(ii)}  \ $0\le K_2(r, x, t)\le
C_2\Big(1+\frac{d_\mu(x,t)}{r}\Big)^{-\gamma-1}, \  \hbox{for} \ \ r>0, x,
t\in\mathbb R$;

{\rm(iii)} \ For $r>0, x, t,z\in\mathbb R$, if $\frac{d_\mu(t,
z)}{r}\leq C_3\min\{1+\frac{d_\mu(x,t)}{r}, 1+\frac{d_\mu(x,z)}{r}\} $, then
$$
| K_2(r, x, t)-K_2(r, x, z))|\le
C_4\Big(\frac{d_\mu(t,z)}{r}\Big)^\gamma\Big(1+\frac{d_\mu(x,t)}{r}\Big)^{-2\gamma-1}.$$

{\rm(iv)} $K_2(r, x, y)=K_2(r, y, x)$.
\end{definition}

\begin{definition}[$\mathbf{maximal\  functions}$]\label{11}
For  $f\in L^1(\RR, \mu)$, $0<\gamma\leq1$, let
$$F_i(r, x, f)=\int_\RR K_i(r, x, y)f(y)d\mu(y)/r,\ \ \ f_i^\times(x)=\sup_{r>0}|F_i(r, x, f)|,  \ \ \ f_{i\bigtriangledown_{\gamma}}^\times(x)=\sup_{r>0, d_\mu(s, x)< r}|F_i(r, s, f)|$$
for $i=1, 2.$ We use $L(f, 0)$ and $L(f, \alpha)$ to denote as following:
$$L(f, 0)=\sup_{x\in \RR, r>0}\inf_{c\in\RR}\int_{B_\mu(x, r)}|f(y)-c|d\mu(y)/r,$$
$$ L(f, \alpha)=\sup_{x, y\in\RR, x\neq y}|f(x)-f(y)|/d_\mu(x, y)^{\alpha},\ \ \ \hbox{for} \  1\geq\alpha>0.$$
We use $f_{\gamma}^*(x)$ to denote as:
\begin{eqnarray}\label{sr3}
f_{\gamma}^*(x)=\sup_{\phi,r}\left\{\left|\int_\RR f(y)\phi(y)d\mu(y)\right|/r:r>0,{\rm
supp}\,\phi\subset B_\mu(x, r), L(\phi, \gamma)\leq r^{-\gamma}, \|\phi\|_{L^{\infty}}\leq1 \right\}.
\end{eqnarray}
The Hardy-Littlewood maximal operator $M_\mu$  is defined as:
$$ M_\mu f(x)=\sup_{r>0}\frac{1}{r}\int_{B_\mu(x, r)}|f(y)|d\mu(y).$$
Then $M_\mu$ is weak-(1, 1) bounded and (p, p) bounded for $p>1$.
\end{definition}

\begin{definition}[$\mathbf{maximal\  functions}$]\label{1}
For  $f\in S'(\RR, d_{\mu}x)$, $0<\gamma\leq1$,
we use  $f_{S\gamma}^*(x)$ to denote as:
$$f_{S\gamma}^*(x)=\sup_{\phi,r}\left\{\left|\int_\RR f(y)\phi(y)d\mu(y)\right|/r:r>0,{\rm
supp}\,\phi\subset B_\mu(x, r), L(\phi, \gamma)\leq r^{-\gamma}, \phi\in S(\RR, d_\mu x), \|\phi\|_{L^{\infty}}\leq1 \right\} .$$
\end{definition}
From the Definition\,\ref{01}, we could deduce that the above Definition\,\ref{11} and Definition\,\ref{1}  associated with the maximal functions are meaningful.
\begin{definition}[$\mathbf{\phi^{(n)}(x)}$,  $\mathbf{H^{\alpha}(\phi)}$,   $\mathbf{[\phi]_{\beta}}$]
For $\phi\in C(\RR, dx)$, $n\in \NN$, $1\geq\alpha\geq0$ and $\beta>0$, we use $\{\beta\}$, $[\beta]$, $H^{\alpha}(\phi)$ and $\phi^{(n)}(x)$ to denote as:
 $$\{\beta\}=\beta-[\beta];\ \ \ \  [\beta]=\max\{n: n\in\ZZ; n\leq\beta\};$$
$$H^{\alpha}(\phi)=\sup_{x, y\in\RR, x\neq y}|\phi(x)-\phi(y)|/|x-y|^{\alpha};$$
$$\phi^{(n)}(x)=\frac{d^n}{dx^n}\phi(x);\ \ \ \ [\phi]_{\beta}=H^{\{\beta\}}(\phi^{([\beta])}).$$

Thus we could see that if $0<\beta\leq1$  $$[\phi]_{\beta}=H^{\beta}(\phi).$$

\end{definition}
Thus it is clear that the following Propositions\,\ref{holder} and \ref{holder2} hold:
\begin{proposition}\label{holder}
For $\phi\in C(\RR, dx)$ satisfying $H^{\alpha}(\phi)\leq1$, $|\phi|\leq1$ ($1\geq\alpha\geq0$, $\beta>0$), there exists $\phi_{\tau}(x)\in S(\RR, dx)$ satisfying the following:

{\rm(i)}$\lim_{\tau\rightarrow0}\|\phi_{\tau}(x)-\phi(x)\|_{\infty}=0,$

{\rm(ii)}$\|\phi_{\tau}(x)\|_{\infty}\leq1,$ \ \ $H^{\alpha}\phi_{\tau}\leq1,$

{\rm(iii)}$H^{\alpha}(\phi^{(1)}_{\tau})\leq C\frac{1}{\tau^{\alpha+1}}.$
\end{proposition}

\begin{proposition}\label{holder2}
 $\beta\geq\beta_1\geq0$. $n\in\ZZ$, $n\leq\beta$. For any $\phi\in S(\RR, dx)$, if $\|\phi(x)\|_{\infty}\leq1,$  $[\phi]_{\beta}\leq1$, then the following holds:
$$ \|\phi^{(n)}(x)\|_{\infty}\leq C,\ \ \ [\phi]_{\beta_1}\leq C,$$
where C is a constant independent on $\phi$.
\end{proposition}
Then we will prove the following Proposition\,\ref{im1}
\begin{proposition}\label{im1}
For $f\in L^1(\RR, \mu)$, we could have $$f_{i\bigtriangledown_{\gamma}}^\times(x)\lesssim_\lambda  f_{\gamma}^*(x)\ \ \  i=1,2.$$
Then if $f_{\gamma}^*(x)\in L^p(\RR, \mu)$ for $p>0$, we could have $$\|f_{i\bigtriangledown_{\gamma}}^\times\|_{L^p(\RR, \mu)}\lesssim_\lambda\|f_{\gamma}^*\|_{L^p(\RR, \mu)}.$$
\end{proposition}
\begin{proof}
When $i=1$, it is clear to see that for fixed r and s the following hold:
\begin{eqnarray*}
\left\{ \begin{array}{cc}
 \left|K_1(r, s, y)\right|\lesssim1                                \\\\
  L\left(K_1(r, s, y), \gamma\right)\lesssim \left(r\right)^{-\gamma}                            \\\\
   suppK_1(r, s, y)\subseteq B_\mu(x, 2Ar)
 \end{array} \right.
 \end{eqnarray*}
then we could have
$$f_{1\bigtriangledown_{\gamma}}^\times(x)\lesssim f_{\gamma}^*(x).$$

When $i=2$,
fix a positive $\phi(t)\in S(\RR, dt)$ so that  $supp\,\phi(t)\subseteq (-1, 1)$, and $\phi(t)=1$ for $t\in (-1/2, 1/2)$. Let
the functions $\psi_{k,x}(t)$ be defined as follows:
\begin{eqnarray*}
\psi_{0,x}(t)=\phi(\frac{\mu(x, t)}{r}), \psi_{k,x}(t)=\phi(\frac{\mu(x, t)}{2^{k}r})-\phi(\frac{\mu(x, t)}{2^{k-1}r}),\ \hbox{for}\  k\geq1.
\end{eqnarray*}
Thus $supp\,\psi_{0,x}(t)\subseteq B_\mu(x, r)$ and $supp\,\psi_{k,x}(t)\subseteq B_\mu(x, 2^{k+1}r)\setminus B_\mu(x, 2^{k-2}r) \,\hbox{for}\  k\geq1$,
$\psi_{k,x}(t)\in S(\RR, d_\mu t) $ for $k \geq0$.
It is clear that $$\sum_{k=0}^{\infty}\psi_{k,x}(t)=1.$$
Then we could conclude:
\begin{eqnarray*}
f_{2\bigtriangledown_{\gamma}}^\times(x)&=&\sup_{r>0, d_\mu(s, x)\leq r}\left|\int_{\RR}K_2(r, s, y)\sum_{k=0}^{\infty}\psi_{k,x}(y)f(y)d\mu(y)/r\right|
\\&\leq& \sum_{k=0}^{+\infty}\sup_{r>0, d_\mu(s, x)\leq r }\left|\int_{\RR}K_2(r, s, y) \psi_{k,x}(y) f(y)d\mu(y)/r\right|.
\end{eqnarray*}
It is clear that the kernel $K_2(r, x, y)$ satisfies:
\begin{eqnarray*}
\left\{ \begin{array}{cc}
 \left|(1+2^k)^{1+\gamma}K_2(r, s, y)\psi_{k,x}(y)\right|\lesssim1                                \\\\
  L\left((1+2^k)^{1+\gamma}K_2(r, s, y)\psi_{k,x}(y),  \gamma\right)\lesssim \left(2^kr\right)^{-\gamma}                            \\\\
   supp(1+2^k)^{1+\gamma}K_2(r, s, y)\psi_{k,x}(y)\subseteq B_\mu(x, 2^{k+1}r)\setminus B_\mu(x, 2^{k-2}r)\,\hbox{for}\  k\geq1.
 \end{array} \right.
 \end{eqnarray*}
Then we could get
\begin{eqnarray*}
f_{2\bigtriangledown_{\gamma}}^\times(x)&=&\sup_{r>0, d_\mu(s, x)\leq r}\left|\int_{\RR}K_2(r, s, y)f(y)d\mu(y)/r\right|
\\&\leq& \sum_{k=0}^{+\infty}\sup_{r>0, d_\mu(s, x)\leq r}\left|\int_{\RR}K_2(r, s, y)\psi_{k,x}(y)f(y)d\mu(y)/r\right|
\\&\lesssim& \sum_{k=0}^{+\infty} (2^k)(1+2^k)^{-1-\gamma}f_{\gamma}^*(x)
\\&\lesssim_\lambda& f_{\gamma}^*(x).
\end{eqnarray*}
This proves the proposition.
\end{proof}

\begin{proposition}\label{sk}
For $f\in L^1(\RR, \mu)$,$1\geq\gamma>0$, $\infty>p>0$ we could obtain
$$  f_{S\gamma}^*(x)= f_{\gamma}^*(x)\ \ \ a.e.x\in\RR \ in\ \mu\ measure.$$
Further more, if
$ \int_\RR |f_{\gamma}^*(x)|^p d\mu (x)\leq \infty $ or $ \int_\RR |f_{S\gamma}^*(x)|^p d\mu(x)\leq \infty $,
we could obtain
 $$\int_\RR |f_{\gamma}^*(x)|^p d\mu (x) \sim \int_\RR |f_{S\gamma}^*(x)|^p d\mu (x)<\infty .$$
\end{proposition}
\begin{proof}
We will prove the following\,(\ref{exuu1})\,first:
\begin{eqnarray}\label{exuu1}
 f_{S\gamma}^*(x)= f_{\gamma}^*(x)\ \ \ a.e.x\in\RR \ in\ \mu\ measure .
\end{eqnarray}
By the definition of $ f_{S\gamma}^*(x)$ and $f_{\gamma}^*(x)$,
it is clear that $f_{S\gamma}^*(x)\leq f_{\gamma}^*(x)$.
If $\phi$ satisfies $ L(\phi, \gamma)\leq r^{-\gamma}$ and ${\rm
supp}\,\phi\subset B_\mu(x, r)$, then $\phi$ is a  continuous function in $\mu$ measure with compact support. Thus there exists sequence $ \{\psi_n\}_n \subset S(\RR, d_\mu x)$ with $\lim_{n\rightarrow\infty}\|\psi_n(t) -\phi(t)\|_\infty = 0$, $\|\psi_n(t) -\phi(t)\|_\infty\neq0$.
Denote $\delta_n(x)$ as
$$\delta_n(x)=\left|\int_{B_\mu(x, r)} f(y)\left(\phi(y)-\psi_n(y)\right)d\mu(y)/r\right|.$$
Then we could conclude:
$$\delta_n(x)\leq M_\mu f(x)\|\psi_n(y) -\phi(y)\|_\infty.$$
We use $i_n$ to denote as $i_n= \|\psi_n(y) -\phi(y)\|_\infty$,
thus we could obtain that:
$$\left\{x: \delta_n(x)>\alpha\right\}\subseteq \left\{x: M_\mu f(x) > \frac{\alpha}{i_n}\right\}.$$
Notice that  $M_\mu$ is weak-(1, 1) bounded, thus the following inequality holds for any $\alpha>0$:
$$\left|\left\{x: \delta_n(x)>\alpha\right\}\right|_{\mu}\leq \frac{1}{\alpha} \|f\|_{L^1(\RR, \mu)}\|\psi_n(y) -\phi(y)\|_\infty.$$
Thus $$\lim_{n\rightarrow+\infty} \left|\left\{x: \delta_n(x)>\alpha\right\}\right|_\mu=0 .$$
Then there exists a sequence $\{n_j\}\subseteq \{n\}$ such that $$\int_\RR f(y)\phi(y)d\mu(y)/r=\lim_{n_j\rightarrow\infty}\int_\RR f(y)\psi_{n_j}(y)d\mu(y)/r,\ \ \ a.e.x\in\RR \ in\ \mu\ measure$$ for $f\in L^1(\RR, \mu)$.
Thus we could obtain:$$\int_\RR f(y)\phi(y)d\mu(y)/r\leq  f_{S\gamma}^*(x)\ \ \ a.e.x\in\RR \ in\ \mu\ measure $$
for any $\phi$ satisfies $ L(\phi, \gamma)\leq r^{-\gamma}$ and ${\rm
supp}\,\phi\subset B_\mu(x, r)$.
We could then deduce $$\sup_{\phi,r>0}\left|\int_\RR f(y)\phi(y)d\mu(y)/r\right|\leq  f_{S\gamma}^*(x)\ \ \ a.e.x\in\RR \ in\ \mu\ measure .$$
Thus
$$  f_{S\gamma}^*(x)= f_{\gamma}^*(x)\ \ \ a.e.x\in\RR \ in\ \mu\ measure.$$
Let $E$ denote a set defined as $E=\left\{x: f_{S\gamma}^*(x)= f_{\gamma}^*(x)\right\}$. \textbf{Next we will prove that} for any $x_0\in\RR$, there is a point $\overline{x}_0 \in E$ such that
\begin{eqnarray}\label{exuu2}
f_{S\gamma}^*(x_0)\lesssim f_{S\gamma}^*(\overline{x}_0).
\end{eqnarray}
Notice that for $x_0\in\RR$, there exist $r_0>0$ and $\phi_0$ satisfying: $supp\,\phi_0\subset B_\mu(x_0, r_0)$, $\phi_0\in S(\RR, d_\mu x)$, $L(\phi_0, \gamma)\leq r_0^{-\gamma}$, $\|\phi_0\|_{L^{\infty}}\leq1$. Then the following inequality could be obtained:
$$\left|\frac{1}{r_0}  \int f(y)\phi_0 (y)d \mu(y)\right|\geq \frac{1}{2} f_{S \gamma}^*(x_0).$$
$|\mu(\RR \backslash E)|=|\mu( E^c)|=0$ implies  $E$ is dense in $\RR$. Then there exists a $\overline{x}_0\in E$ with $d_\mu (x_0, \overline{x}_0)\leq\frac{r_0}{4}$.
Thus $supp\,\phi_0\subset B_\mu(\overline{x}_0, 4r_0)$ holds. Thus we could obtain
$$\left|\frac{1}{r_0}  \int f(y)\phi_0 (y)d \mu(y)\right|\leq C f_{S\gamma}^*(\overline{x}_0),$$
where C is a constant independent on $f$, $\gamma$ and $r_0$. Then  Formula\,(\ref{exuu2})\,could be deduced. By Formula\,(\ref{exuu2})\,we could deduce that:
\begin{eqnarray}\label{uu3}
 \int_E |f_{S\gamma}^*(x)|^p d\mu (x)<\infty\,\Rightarrow  \,\int_\RR |f_{S\gamma}^*(x)|^p d\mu (x) \sim \int_E |f_{S\gamma}^*(x)|^p d\mu (x)<\infty .
\end{eqnarray}
In the same way, we could conclude that
\begin{eqnarray}\label{uu4}
\int_\RR |f_{\gamma}^*(x)|^p d\mu (x) \sim \int_E |f_{\gamma}^*(x)|^p d\mu (x).
\end{eqnarray}
From Formula\,(\ref{exuu1})\,we could deduce:
$$\int_E |f_{\gamma}^*(x)|^p d\mu (x) = \int_E |f_{S\gamma}^*(x)|^p d\mu (x)<\infty .$$
The above Formula together with\,(\ref{uu3})\,(\ref{uu4})\,lead to
$$\int_\RR |f_{\gamma}^*(x)|^p d\mu (x) \sim \int_\RR |f_{S\gamma}^*(x)|^p d\mu (x)<\infty $$
holds if
$ \int_\RR |f_{\gamma}^*(x)|^p d\mu (x)< \infty $ or $ \int_\RR |f_{S\gamma}^*(x)|^p d\mu(x)< \infty $.
This proves the proposition.
\end{proof}
\begin{definition}[$\mathbf{SS_{\beta}} $ ]\label{ssbeta}
We  use $SS_{\beta}$ ($\beta>0$) to denote as
\begin{eqnarray}\label{ss5}
 SS_{\beta}=\bigg\{\phi: \phi\in S(\RR, dx),\, supp\,\phi\subset[-1, 1], \|\phi\|_{L^{\infty}}\leq1, [\phi]_{\beta}\leq1 \bigg\}.
 \end{eqnarray}

\end{definition}
By Proposition\,\ref{sk6} and  Proposition\,\ref{holder},   we could also define  $f_{S\gamma}^*$ ($1\geq\gamma>0$) and $f_{S\beta}^*$ ($\beta>0$) for $f\in S'(\RR, d_{\mu}x)$ as following:
\begin{eqnarray}\label{mm1}
f_{S\gamma}^*(x)=\sup_{\psi,r>0}\left\{\left|\int_\RR f(y)\psi\left(\frac{\mu(x, y)}{r} \right)d\mu(y)\right|/r: r>0, \psi(t)\in S(\RR, dx), \right.\\
\nonumber  supp\,\psi(t)\subset [-1, 1], \|\psi\|_{L^{\infty}}\leq1, H^{\gamma}\psi\leq1  \bigg\}
\end{eqnarray}
\begin{eqnarray}\label{mm11}
f_{S\beta}^*(x)&=&\sup_{\psi,r>0}\left\{\left|\int_\RR f(y)\psi\left(\frac{\mu(x, y)}{r} \right)d\mu(y)\right|/r: r>0, \psi(t)\in SS_{\beta}  \right\}.
\end{eqnarray}

\begin{definition}[$\mathbf{M_{\phi\beta}f(x) } $ ]\label{mm2} For $f\in S'(\RR, d_{\mu}x)$, $M_{\phi\beta}f(x)$ is defined as
\begin{eqnarray*}
 M_{\phi\beta}f(x)  =\sup_{r>0}\left\{\left|\int_\RR f(y)\phi\left(\frac{\mu(x, y)}{r} \right)d\mu(y)\right|/r: r>0, \phi(t)\in SS_{\beta} \right\}.
\end{eqnarray*}
 \end{definition}
Thus it is easy to see that
\begin{eqnarray}\label{sk7}
f^*_{S\beta}(x)\sim\sup_{\phi(t) \in SS_{\beta}}M_{\phi\beta}f(x).
\end{eqnarray}
Let $M_{\phi\beta}^*f(x)$ be defined as
\begin{eqnarray}
 M_{\phi\beta}^*f(x)  =\sup_{d_{\mu}(x, y)<r}\left\{\left|\int_\RR f(u)\phi\left(\frac{\mu(y, u)}{r} \right)d\mu(u)\right|/r: r>0, \phi(t)\in SS_{\beta}  \right\}.
\end{eqnarray}

\begin{definition}[$\mathbf{M_{\phi\beta}^*f(x)}$ and $\mathbf{M_{\phi\beta a}^*f(x)}$ ]\label{mm3}Notice that $\mu(y, u)=\mu(x, u)-\mu(x, y)$. For $f\in S'(\RR, d_{\mu}x)$, let  $s=\mu(x, y)$, $M_{\phi\beta}^*f(x)$ and $M_{\phi\beta a}^*f(x)$ could be written as following:
\begin{eqnarray}
 M_{\phi\beta}^*f(x)  =\sup_{|s|<r}\left\{\left|\int_\RR f(u)\phi\left(\frac{\mu(x, u)-s}{r} \right)d\mu(u)\right|/r: r>0, \phi(t)\in SS_{\beta}  \right\}
\end{eqnarray}
 \end{definition}
\begin{eqnarray}
 M_{\phi\beta a}^*f(x)  =\sup_{|s|<ar}\left\{\left|\int_\RR f(u)\phi\left(\frac{\mu(x, u)-s}{r} \right)d\mu(u)\right|/r: r>0, \phi(t)\in SS_{\beta} \right\}.
\end{eqnarray}
\begin{definition}[$ \mathbf{M_{\phi\beta N}^{**}f(x)}$ ]\label{mm4} For $f\in S'(\RR, d_{\mu}x)$,   $ M_{\phi\beta N}^{**}f(x)$ is defined as:
\begin{eqnarray}
 M_{\phi\beta N}^{**}f(x)  =\sup_{s\in\RR, r>0}\left\{\left|\int_\RR f(u)\phi\left(\frac{\mu(x, u)-s}{r} \right)\left(1+\frac{|s|}{r} \right)^{-N}d\mu(u)\right|/r: r>0, \phi(t)\in SS_{\beta}.  \right\}
\end{eqnarray}
 \end{definition}
Thus it is clear that
\begin{eqnarray}\label{su1}
M_{\phi\beta}f(x)\lesssim  M_{\phi\beta}^*f(x) \lesssim M_{\phi\beta N}^{**}f(x) .
\end{eqnarray}

\begin{definition}[$\mathbf{H_{\mu\beta}^p(\RR)}$ and $\mathbf{\widetilde{H}_{\mu\beta}^p(\RR)}$ ]\label{2}
$\widetilde{H}_{\mu\beta}^p(\RR)$ and $H_{\mu\beta}^p(\RR)$ are defined as follows:
$$\widetilde{H}_{\mu\beta}^p(\RR)\triangleq\bigg\{g\in L^1(\RR, \mu): g_{S\beta}^*(x)\in L^p(\RR, \mu) \bigg\},$$
$$H_{\mu\beta}^p(\RR)\triangleq\bigg\{g\in S'(\RR, d_{\mu}x): g_{S\beta}^*(x)\in L^p(\RR, \mu) \bigg\}.$$
The norm  is defined as
$$ \|g\|_{H_{\mu\beta}^p(\RR)}^p =\int_\RR |g_{S\beta}^*(x)|^p d\mu(x).$$

\end{definition}
When $1<p<\infty$,  $H_{\mu\beta}^p(\RR)= L^p(\RR, \mu)$,  $\widetilde{H}_{\mu\beta}^p(\RR)$ is dense in $L^p(\RR, \mu)$.
\begin{proposition}\label{tan}For fixed numbers $a\geq b>0$, F(x, r) is  a function defined on $\RR_+^2$, its nontangential maximal function $F_a^*(x)$ is defined as
 $$F_a^*(x)=\sup_{d_\mu(x, y)<ar}|F(y, r)|.$$ If $F_a^*(x)\in L^1(\RR, \mu)$ or $F_b^*(x)\in L^1(\RR, \mu)$, then we could have
$$\int_\RR \chi\left\{x: F_a^*(x)>\alpha \right\} d\mu(x)\leq c\frac{a+b}{b} \int_\RR \chi\left\{x: F_b^*(x)>\alpha \right\} d\mu(x) .$$
c is a constant independent on $F$, a, b, and $\alpha$.
\end{proposition}
\begin{proof}
First we could see that $\left\{x: F_a^*(x)>\alpha \right\}$ is an open set. It is clear that $$\left\{x: F_b^*(x)>\alpha \right\}\subseteq\left\{x: F_a^*(x)>\alpha \right\},$$ when $a\geq b>0$.
For any $z$ with $z\in\left\{x: F_a^*(x)>\alpha \right\}$, there exists $x_0,\,r_0$ such that $|F(x_0, r_0)|> \alpha$ and $d_{\mu}(z, x_0)<ar_0$ hold. It is clear that $B_{\mu}(x_0, br_0)\subseteq \left\{x: F_b^*(x)>\alpha \right\}$ and $B_{\mu}(x_0, ar_0)\subseteq \left\{x: F_a^*(x)>\alpha \right\}$ hold. Thus we could deduce that the following Formula hold:
$$\frac{|B_{\mu}(z, (a+b)r_0)\bigcap\left\{x: F_b^*(x)>\alpha \right\}|_{\mu}}{|B_{\mu}(z, (a+b)r_0)|_{\mu}}\geq\frac{|B_{\mu}(x_0, br_0)|_{\mu}}{|B_{\mu}(x_0, (a+b)r_0)|_{\mu}}\geq\frac{b}{a+b}.$$
Thus we could obtain
$$\left\{x: F_a^*(x)>\alpha \right\} \subseteq \left\{x: M_{\mu}\chi\left\{x: F_b^*(x)>\alpha \right\} >\frac{b}{a+b} \right\},$$
where $M_{\mu}$ is the Hardy-Littlewood maximal operator. With the fact that $M_{\mu}$ is weak-(1, 1), we could deduce:
$$\int_\RR \chi\left\{x: F_a^*(x)>\alpha \right\} d\mu(x)\leq c\frac{a+b}{b} \int_\RR \chi\left\{x: F_b^*(x)>\alpha \right\} d\mu(x).$$
This proves the proposition.
\end{proof}

When $F_b^*(x)\in L^p(\RR, \mu)$, by Proposition\,\ref{tan}, we could obtain   the following inequality for $p>0$:
 \begin{eqnarray}\label{no1}
 \int_\RR |F_a^*(x)|^pd\mu(x)\leq c \left(\frac{a+b}{b}\right) \int_\RR |F_b^*(x)|^pd\mu(x) .
 \end{eqnarray}

\begin{proposition}\label{no} For $f\in S'(\RR, d_{\mu}x)$, if $\| M_{\phi\beta}^*f(x)\|_{L^p(\RR, \mu)}<\infty$, then
$$\|M_{\phi\beta N}^{**}f(x)\|_{L^p(\RR, \mu)}\leq c_1\| M_{\phi\beta}^*f(x)\|_{L^p(\RR, \mu)}\ \ \hbox{for}\ p> 0, N>1/p. $$
$c_1$ is independent on $\phi$ and $f$.
\end{proposition}
\begin{proof}
For $\phi(t)\in SS_{\beta}$,
\begin{eqnarray*}
 M_{\phi\beta N}^{**}f(x)  &=&\sup_{s\in\RR, r>0}\left|\int_\RR f(y)\phi\left(\frac{\mu(x, y)-s}{r} \right)\left(1+\frac{|s|}{r} \right)^{-N}d\mu(y)\right|/r
 \\&\lesssim& \left(\sup_{ 0 <s\leq r} + \sum_{k=1}^{\infty}\sup_{2^{k-1}r<s\leq2^kr}\right) 2^{-kN}  \left|\int_\RR f(y)\phi\left(\frac{\mu(x, y)-s}{r} \right)d\mu(y)\right|/r
 \\&\lesssim& \sum_{k=0}^{\infty}  2^{-kN} M_{\phi\beta 2^k}^*f(x).
\end{eqnarray*}
Thus together with Formula\,(\ref{no1}), we could deduce the following inequality for $N>1/p$:
\begin{eqnarray*}
 \int_\RR |M_{\phi\beta N}^{**}f(x)|^pd\mu(x)\leq c_1  \int_\RR |M_{\phi\beta}^*f(x)|^pd\mu(x) .
 \end{eqnarray*}
This proves our Proposition.
\end{proof}

It is clear that  the following Proposition holds from \cite{Stein}:
\begin{proposition}\cite{Stein} \label{no3}
Suppose $\phi, \psi \in SS_{\beta}$,  with $\int\psi(x) dx=1$. Then there is a sequence $\{\eta^{k}\},$ $\eta^{k}\in S(\RR, dx)$, so that
$$\phi\left(\frac{\mu(y, u)}{r}\right)=\sum_{k=0}^{\infty}\int_\RR \eta^{k}\left(\frac{s}{r}\right)\psi\left(\frac{\mu(y, u)-s}{2^{-k}r}\right)\frac{ds}{2^{-k}r}.$$
$\eta^{k}$ satisfies $$\|\eta^{k}\|_{a, b}\leq C(2^{-kM}),\ \ as\,k\rightarrow\infty.$$

\end{proposition}

Now we need to prove that the nontangential maximal operator $M_{\psi\beta }^*f(x)$  allows the control of maximal function $f_{S\beta}^*(x)$.
\begin{proposition}\label{no4}There exists $\beta>0$, such that for any  $ \psi\in SS_{\beta} $, with $\int\psi(x) dx=1$ and $p> 0$, the following holds for $f\in S'(\RR, d_{\mu}x)$ if $\| M_{\psi\beta}^*f(x)\|_{L^p(\RR, \mu)}<\infty$:
$$\|f_{S\beta}^*\|_{L^p(\RR, \mu)}\leq c\| M_{\psi\beta }^*f\|_{L^p(\RR, \mu)},$$
C is dependent on $\beta$.
\end{proposition}
\begin{proof}
For any  $\phi, \psi\in SS_{\beta}$, with $\int\psi(x) dx=1$ by Proposition\,\ref{no3}, we have
\begin{eqnarray*}
M_{\phi\beta}f(x)  =\sup_{r>0}\left|\int_\RR f(y)\phi\left(\frac{\mu(x, y)}{r} \right)d\mu(y)\right|/r\lesssim \sup_{r>0}\sum_{k=0}^{\infty}\left|\int_\RR\int_\RR f(y)\eta^{k}\left(\frac{s}{r}\right)\psi\left(\frac{\mu(x, y)-s}{2^{-k}r}\right)d\mu(y)\frac{ds}{2^{-k}r}\right|/r.
\end{eqnarray*}
Thus  we could obtain:
\begin{eqnarray*}
M_{\phi\beta}f(x)  &\lesssim& \sup_{r>0}\sum_{k=0}^{\infty}\left|\int_\RR\int_\RR f(y)\eta^{k}\left(\frac{s}{r}\right)\psi\left(\frac{\mu(x, y)-s}{2^{-k}r}\right)d\mu(y)\frac{ds}{2^{-k}r}\right|/r
\\&\lesssim& \sup_{ r>0}\sum_{k=0}^{\infty} \int_\RR    \left|\int_\RR f(y)\psi\left(\frac{\mu(x, y)-s}{2^{-k}r} \right)\left(1+\frac{|s|}{2^{-k}r} \right)^{-N}\frac{d\mu (y)}{2^{-k}r}\right| \left|\eta^{k}\left(\frac{s}{r}\right)\left(1+\frac{|s|}{2^{-k}r} \right)^{N}\right|ds/r
\\&\lesssim& M_{\psi\beta N}^{**}f(x)\sum_{k=0}^{\infty} \int_\RR\left|\eta^{k}\left(\frac{s}{r}\right)\left(1+\frac{|s|}{2^{-k}r} \right)^{N}\right|ds/r
\\&\lesssim& M_{\psi\beta N}^{**}f(x)\sum_{k=0}^{\infty} 2^{-k}
\\&\lesssim& M_{\psi\beta N}^{**}f(x),
\end{eqnarray*}
where $\|\eta^{k}\|_{a, b}=O(2^{-k(N+1)}) $ for a suitable collection of seminorms.
Thus $$f_{S\beta}^*(x)\sim\sup_{\phi\in SS_{\beta}}M_{\phi\beta}f(x)\lesssim M_{\psi\beta N}^{**}f(x).$$
For all $x\in \RR$, $N>1/p$, from Proposition\,\ref{no}, we could get
$$\|f_{S\beta}^*\|_{L^p(\RR, \mu)}\leq c\| M_{\psi\beta }^*f\|_{L^p(\RR, \mu)}.$$ This proves our proposition.
\end{proof}

\begin{proposition}\label{no5}There exists $\beta>0$, such that for $\ p> 0,  \phi\in SS_{\beta},\  with \ \int\phi(x) dx=1$, the following holds for $f\in S'(\RR, d_{\mu}x)$ if $\| M_{\phi\beta}^*f\|_{L^p(\RR, \mu)}<\infty$:
$$\|M_{\phi\beta }^*f\|_{L^p(\RR, \mu)}\leq c\|M_{\phi\beta }f \|_{L^p(\RR, \mu)}\ \ .$$
C is dependent on  $\beta$.
\end{proposition}
\begin{proof}
We assume  $\|M_{\phi\beta }^*f\|_{L^p(\RR, \mu)}<\infty$ first. Let $F$ be defined as $F=\{x: f_{S\beta}^*(x)\leq \sigma M_{\phi\beta }^*f(x)\}$.
By Proposition\,\ref{no4}, the following holds:
\begin{eqnarray}\label{g1}
\int_{F^c}|M_{\phi\beta }^*f(x)|^pd\mu(x)\leq \sigma^{-p} \int_{F^c}|f_{S\beta}^*(x)|^pd\mu(x)\leq C\sigma^{-p} \int_{\RR} |M_{\phi\beta }^*f(x)|^pd\mu(x) .
\end{eqnarray}
Choosing $\sigma^p\geq2C$, we could have
\begin{eqnarray}\label{exuu5}
 \int_{\RR}|M_{\phi\beta }^*f(x)|^p d\mu(x)\lesssim\int_{F} |M_{\phi\beta }^*f(x)|^pd\mu(x).
\end{eqnarray}
Next we will show that for any $q>0$
$$\left|M_{\phi\beta }^*f(x) \right|^q\leq c M_\mu(M_{\phi\beta}f)^q(x).$$
Let $f(x, r)$ be defined as
$$f(x, r)=\int_\RR f(u)\phi\left(\frac{\mu(x, u)}{r} \right)d\mu(u)/r.$$
Then for any $x\in\RR$, there exists (y, r), satisfying $d_\mu(x, y)<r$ and $|f(y, r)|\geq M_{\phi \beta}^*f(x) /2$. Choose $   0<\delta<1$  and  $x'$ satisfying $d_\mu(x', y)< \delta r$.
Then there exists $\, \xi\in [x', y] $ such that:
\begin{eqnarray*}
|f(x', r)-f(y, r)|&\leq& \delta r\sup_{x\in B_{\mu}(y, \delta r)}\left|\frac{d}{d_\mu x}f(x, r)\right|
\\&\leq&  \delta \sup_{\xi\in B_{\mu}(y, \delta r)}\left|  \int_\RR f(u) \phi^{(1)}\left(\frac{\mu(\xi, u)}{r} \right)d\mu(u)/r                 \right|
\\&\leq&\delta\sup_{\xi\in B_{\mu}(y, \delta r)}\left|\int_\RR f(u) \phi^{(1)}\left(\frac{\mu(x, u)-\mu(x, \xi)}{r} \right)d\mu(u)/r \right|
\\&\leq&\delta\sup_{|h|\leq1+\delta}\left| \int_\RR f(u) \phi^{(1)}\left(\frac{\mu(x, u)}{r}-h \right) d\mu(u)/r \right|.
\end{eqnarray*}
Notice that  $|h|\leq1+\delta<2$ with $\|H_{x}^{\beta}\phi^{(1)}(x-h)\|_{\infty}\leq C$, $\|\phi^{(1)}(x-h)\|_{\infty}\leq C$. By the definition of $f_{S\beta}^*(x)$,
\begin{eqnarray*}
|f(x', r)-f(y, r)|
\leq C_0\delta f_{S\beta}^*(x)\leq C_0\delta\sigma M_{\phi\beta }^*f(x) \ \ \hbox{for}\ x\in F.
\end{eqnarray*}
Taking $\delta$ small enough such that $ C_0\delta \sigma\leq1/4$, we obtain $$\left|f(x', r)\right|\geq \frac{1}{4}M_{\phi\beta }^*f(x).$$
Thus the following inequality holds:
\begin{eqnarray*}
 \left|M_{\phi\beta }^*f(x) \right|^q &\leq& \left|\frac{1}{B_\mu(y, \delta r )}\right| \int_{B_\mu(y, \delta r)}  4^q|f(x', r)|^q d\mu(x')
 \\&\leq&   \left|\frac{B_\mu(x, (1+\delta) r )}{B_\mu(y, \delta r )}\right|  \left|\frac{1}{B_\mu(x, (1+\delta) r )}\right|\int_{B_\mu(x, (1+\delta) r)}  4^q|f(x', r)|^qd\mu(x')
\\&\leq&\frac{1+\delta}{\delta}\left|\frac{1}{B_\mu(x, (1+\delta) r )}\right|\int_{B_\mu(x, (1+\delta) r)}  4^q|f(x', r)|^qd\mu(x')
\\&\leq& C M_\mu[(M_{\phi\beta }f)^q](x),
\end{eqnarray*}
where $M_\mu$ is the Hardy-Littlewood Maximal Operator. Thus for $p$ satisfying $p>q$, using the maximal theorem for $M_\mu$ leads to
\begin{eqnarray}\label{g2}
\int_{F}|M_{\phi\beta }^*f(x)|^pd\mu(x)\leq C  \int_{F} (M_\mu[(M_{\phi\beta }f)^q](x))^{p/q}\leq C \int_{\RR} |M_{\phi \beta}f(x)|^pd\mu(x).
\end{eqnarray}
Combining\,(\ref{exuu5})\,and\,(\ref{g2})\,together, we could prove the proposition.
\end{proof}
\begin{proposition}\cite{Stein}\textbf{Classical Hardy spaces $H^p(\RR)$ in Euclid space}\label{no12}

Let $F=\{\|\cdot\|_{a, b}\}$ be any finite collection of seminorms on $S(\RR, dx)$. We use $S_{F}$ to denote the subset of $S(\RR, dx)$ controlled by
this collection of seminorms:$$S_{F}=\left\{\phi\in S(\RR, dx): \|\phi\|_{a, b}\leq1\  \hbox{for}\  \hbox{any}\  \|\cdot\|_{a, b}\in F \right\}.$$
Let $M_Ff(x)$ be defined as $M_Ff(x)=\sup_{\phi\in S_{F}}\sup_{t>0}\left(f\ast\phi_t\right)(x).$
If $f\in H^p(\RR)$, then $ \|f\|_{H^p(\RR)}^p =\int_\RR |M_Ff(x)|^p dx.$
Thus every $f\in H^p(\RR)$ can be written as a sum of $H^p(\RR)$ atoms:
$f=\sum_k\lambda_ka_k$
in the sense of distribution.
An $H^p(\RR)$ atom is a function $a(x)$ so that:

{\rm(i)} $a(x)$ is supported in a ball B in Euclid space;

{\rm(ii)} $|a(x)|\leq|B|^{-1/p}$ almost everywhere;

{\rm(iii)}$\int_\RR x^{n}a(x)dx=0$ for all $n\in\ZZ$ with $|n|\leq p^{-1}-1$.
Further more

$$ \|f\|_{H^p(\RR)}^p =\int_\RR |M_Ff(x)|^p dx\sim \sum_k \lambda_k^p.$$
\end{proposition}

\begin{proposition}\label{H spa}
For $\alpha$ and $\beta$ satisfying $\beta\geq\alpha>p^{-1}-1$($0<p\leq1$), we could deduce that $\widetilde{H}_{\mu\beta}^p(\RR)$ is dense in $H_{\mu\beta}^p(\RR)$ and we could also deduce that $$H_{\mu\beta}^p(\RR)=H_{\mu\alpha}^p(\RR).$$ For any $f\in H_{\mu\beta}^p(\RR)$, we could also have $$C_2\|f\|_{H_{\mu\beta}^p(\RR)}^p\leq \|f\|_{H_{\mu\alpha}^p(\RR)}^p \leq C_1\|f\|_{H_{\mu\beta}^p(\RR)}^p,$$ where $C_1$ and $C_2$ are independent on $f$.
\end{proposition}
\begin{proof}
First, with the fact $SS_\beta\subseteq SS_\alpha$, it is easy to see that $$H_{\mu\beta}^p(\RR)\supseteq H_{\mu\alpha}^p(\RR),\ \ \|f\|_{H_{\mu\beta}^p(\RR)}^p\leq C \|f\|_{H_{\mu\alpha}^p(\RR)}^p $$
for $\beta\geq\alpha>p^{-1}-1$. Thus we could deduce that $f\in H_{\mu\beta}^p(\RR)$, if $f\in H_{\mu\alpha}^p(\RR)$.

Next we will prove that $f\in H_{\mu\alpha}^p(\RR)$, if $f\in H_{\mu\beta}^p(\RR)$.
 Notice that $P(x)=\mu(x, 0)$ is a bijection on $\RR$. Let $P^{-1}(x)$  be the reverse map of $P(x)$. Let $g(t)=f\circ P^{-1}(t)$.  From
 Definition\,\ref{mm2}, Definition\,\ref{mm3}, Definition\,\ref{mm4}, Definition\,\ref{1} and Definition\,\ref{2}, Proposition\,\ref{no}, Proposition\,\ref{no4}, Proposition\,\ref{no5} and Proposition\,\ref{no12}, we could deduce that $g(t)\in H^p(\RR)$, if $f\in H_{\mu\beta}^p(\RR)$. With the fact
that $H^p(\RR)\bigcap L^1(\RR)$ is dense in $H^p(\RR)$, we could deduce that
$\widetilde{H}_{\mu\beta}^p(\RR)$ is dense in $H_{\mu\beta}^p(\RR)$. We could also deduce the the following equation:
 $$\|f\|_{H_{\mu\beta}^p(\RR)}^p=\|g\|_{H^p(\RR)}^p.$$
By Proposition\,\ref{no12}, $g\in H^p(\RR)$ can be written as a sum of $H^p(\RR)$ atoms:
$$g=\sum_k\lambda_ka_k$$
in the sense of distribution. Let $b_{k}(x)=a_k(P(x)),$
then it is clear that the  functions $\{b_k(x)\}_k$ satisfy the following:

{\rm(i)} $b_k(x)$ is supported in a ball $B_\mu(x_k, r_k)$;

{\rm(ii)} $|b_k(x)|\leq|B_\mu(x_k, r_k)|^{-1/p}$ almost everywhere in $\mu$ measure;

{\rm(iii)}$\int \mu(x, 0)^{n}b_k(x)d\mu(x)=0$ for all $n\in\ZZ$ with $|n|\leq p^{-1}-1$.
Together with Proposition\,\ref{no12}, we could deduce that
$$\int_\RR f(x)\phi(x) d\mu (x)=\int_\RR \sum_k\lambda_kb_k(x)\phi(x) d\mu (x)=\sum_k\int_\RR \lambda_kb_k(x)\phi(x) d\mu (x)$$
holds for any $\phi(x)\in S(\RR, d_\mu x)$, and
$$ \|f\|_{H_{\mu\beta}^p(\RR)}^p=\|g\|_{H^p(\RR)}^p \sim \sum_k \lambda_k^p,$$
holds.
For any $\psi(x)\in SS_{\alpha}$ satisfying $\int\psi(x)dx=1$,  we have:
\begin{eqnarray}\label{1000}
 \int_{B_\mu(x_k, 4 r_k)} |b_{k\alpha}^*(x)|^pd\mu(x)&\leq& C \int_{B_\mu(x_k, 4 r_k)} |M_{\mu}b_{k}(x)|^pd\mu(x)
 \\&\leq& \nonumber C \left(\int_{B_\mu(x_k, 4 r_k)} |M_{\mu}b_{k}(x)|^2d\mu(x)\right)^{p/2}\left(\int_{B_\mu(x_k, 4 r_k)} 1d\mu(x)\right)^{1-(p/2)}
 \\&\leq& \nonumber C,
 \end{eqnarray}
where C is independent on $\psi$ and $b_{k}$.
For $s\in\ZZ$, $s\leq\alpha$, by Taylor Expansion, there exists $\xi\in B_{\mu}(x_k, t)$ such that the following holds:
\begin{eqnarray}\label{1002}
\psi\left(\frac{\mu(t, x)}{r}\right)
&=& \nonumber \sum_{s=0}^{[\alpha]-1}\frac{1}{s!}\psi^{(s)}\left(\frac{\mu(x_k, x)}{r}\right)\left(\frac{\mu(t, x_k)}{r}\right)^s
\\&+& \nonumber \frac{1}{[\alpha]!}\psi^{([\alpha])}\left(\frac{\mu(\xi, x)}{r}\right)\left(\frac{\mu(t, x_k)}{r}\right)^{[\alpha]}.
 \end{eqnarray}
Let $P(x, x_k)$ be defined as following:
\begin{eqnarray}\label{1111}
P(x, x_k)
&=& \nonumber \sum_{s=0}^{[\alpha]-1}\frac{1}{s!}\psi^{(s)}\left(\frac{\mu(x_k, x)}{r}\right)\left(\frac{\mu(t, x_k)}{r}\right)^s.
\end{eqnarray}
Thus we could obtain
\begin{eqnarray}\label{1112}
\left|P(x, x_k)-\psi\left(\frac{\mu(t, x)}{r}\right)\right|\leq  \frac{1}{[\alpha]!}\left|\left(\frac{\mu(t, x_k)}{r}\right)^\alpha\right|.
 \end{eqnarray}
Thus by Proposition\,\ref{holder2} and the  vanishing property of $b_k$ we could have:
\begin{eqnarray}\label{1001}
 & &\int_{B_\mu(x_k, 4 r_k)^c} \left|\int b_k(t)\psi\left(\frac{\mu(t, x)}{r}\right)\frac{d\mu(t)}{r}\right|^pd\mu(x)
 \\&=&\nonumber\int_{B_\mu(x_k, 4 r_k)^c} \left|\int b_k(t)\left(\psi\left(\frac{\mu(t, x)}{r}\right)-P(x, x_k)\right)\frac{d\mu(t)}{r}\right|^pd\mu(x)
 \\&\leq &\nonumber C \int_{B_\mu(x_k, 4 r_k)^c} \left|\frac{r_k^{\alpha+1-p^{-1}}}{r^{\alpha+1}}\right|^pd\mu(x).
 \end{eqnarray}
Notice that $r>|\mu(x, x_k)-r_k|$, $\alpha>p^{-1}-1$ and $0<p\leq1$, thus Formula\,(\ref{1001}) implies:
\begin{eqnarray}\label{1003}
  \int_{B_\mu(x_k, 4 r_k)^c} \left|\frac{r_k^{\alpha+1-p^{-1}}}{r^{\alpha+1}}\right|^pd\mu(x)\leq C.
 \end{eqnarray}
Formula\,(\ref{1000}) and Formula\,(\ref{1003}) imply:
\begin{eqnarray*}
\int_\RR |b_{k\alpha}^*(x)|^pd\mu(x)\leq C,
 \end{eqnarray*}
where C is independent on $\psi$ and $b_{k}$. Thus
$$ \|f\|_{H_{\mu\alpha}^p(\RR)}^p \leq  C\sum_k \lambda_k^p\|b_{k}\|_{H_{\mu\alpha}^p(\RR)}^p \leq C \sum_k \lambda_k^p\leq C\|f\|_{H_{\mu\beta}^p(\RR)}^p.$$
Thus $f\in H_{\mu\alpha}^p(\RR)$, if $f\in H_{\mu\beta}^p(\RR)$. Thus, we could deduce that
$$H_{\mu\alpha}^p(\RR)=H_{\mu\beta}^p(\RR).$$
This proves the Proposition.
\end{proof}

\begin{definition}
Let $\{b_k^{n, p}(x)\}$ be functions as follows:\\
{\rm(i)} $b_k^{n, p}(x)$ is supported in a ball $B_\mu(x_k, r_k)$;\\
{\rm(ii)} $|b_k^{n, p}(x)|\leq|B_\mu(x_k, r_k)|^{-1/p}$ almost everywhere in $\mu$ measure;\\
{\rm(iii)}$\int \mu(x, 0)^{m}b_k^{n, p}(x)d\mu(x)=0$ for all $m\in\NN$ with $ m\leq n$.\\
For $n\geq[p^{-1}-1]$,  $A^{n, p}(\RR)$ is defined  as
\begin{eqnarray*}A^{n, p}(\RR)&\triangleq&  \left\{f\in  S'(\RR, d_{\mu}x):\int_\RR f(x)\phi(x) d\mu (x)=\sum_k\int_\RR \lambda_kb_k^{n, p}(x)\phi(x) d\mu (x)
\ \right. \\  & &\left. \hbox{for}\  \hbox{any}\ \phi(x)\in S(\RR, d_\mu x), \hbox{where}\   \sum_{k}|\lambda_{k}|^{p}<+\infty .\right\}
\end{eqnarray*}
\end{definition}
The norm is defined by:
$$\|f\|_{A^{n, p}(\RR)}=\inf\left(\sum_{k}|\lambda_{k}|^{p}\right)^{1/p}.$$
Thus by Proposition\,\ref{H spa}, we could conclude that
$$A^{n, p}(\RR)= H_{\mu\alpha}^p(\RR)=H_{\mu\beta}^p(\RR)$$
for $\beta\geq\alpha>p^{-1}-1$ and $n\geq[p^{-1}-1]$($0<p\leq1$).
\begin{theorem}\label{H spa1}
For $\beta_1\geq\beta_2>p^{-1}-1$,  $n\geq[p^{-1}-1]$, $f\in A^{n, p}(\RR)$ ($0<p\leq1$), we could obtain $$A^{n, p}(\RR)= H_{\mu\beta_2}^p(\RR)=H_{\mu\beta_1}^p(\RR),$$
and
$$\|f\|_{A^{n, p}(\RR)}\sim\|f\|_{H_{\mu\beta_1}^p(\RR)}\sim\|f\|_{H_{\mu\beta_2}^p(\RR)}.$$
We could also deduce that $\widetilde{H}_{\mu\beta_1}^p(\RR)$ is dense in $H_{\mu\beta_1}^p(\RR)$ from Proposition\,\ref{H spa}.
\end{theorem}

\begin{proposition}\label{exx6}
For the kernel $K_1(r, x, y)$ as above,  there exists sequence $\{a_{x, r}^{\tau}(y): a_{x, r}^{\tau}(y)\in C_c(\RR, d_{\mu}y)\bigcap S(\RR, d_{\mu}y)\}_{\tau}$ satisfying the following:\\
{\rm(i)} \ $\displaystyle{a_{x, r}^{\tau}(y)=a_{y, r}^{\tau}(x)}$,\\
{\rm(ii)} $ \lim_{\tau\rightarrow\infty}\| K_1(r, x, y)-a_{x, r}^{\tau}(y)\|_{\infty} =0 $,\\
{\rm(iii)} $\  0\leq a_{x, r}^{\tau}(y)\le C,$\\
{\rm(iv)} \ For $r>0, x, y,z\in\mathbb R$,
 $$ |a_{x, r}^{\tau}(y)-a_{x, r}^{\tau}(z)|\le C \Big(\frac{d_\mu(y,z)}{r}\Big)^\gamma. $$\\
C is constant independent on $K_1(r, x, y)$ and $a_{x, r}^{\tau}(y)$.\\
{\rm(v)} For $\tau$ small enough
\begin{eqnarray*}
\left|a_{x, r}^{\tau}(y)-K_1(r, x, y)\right|\leq C \Big(\frac{\tau}{r}\Big)^\gamma.
\end{eqnarray*}
{\rm(vi)} \  $a_{x, r}^{\tau}(x) > C, \  \hbox{for}\ \ r>0, x\in\mathbb R$.
\end{proposition}
\begin{proof}
Let $\rho(x)$ to be a fixed function so that
\begin{eqnarray*}
\rho(x)=\left\{ \begin{array}{cc}
                             \vartheta\exp\left\{\frac{1}{|x|^2-1}\right\}, \ \ \hbox{for}\  |x|<1\\\\
                             0, \ \ \ \hbox{for}\  |x|\geq1.
                           \end{array}\right.
\end{eqnarray*}
where $\vartheta$ is a constant satisfying $\int\rho(x)dx=1$.
We use $a_{x, r}^{\tau}(y)$ to denote as
$$a_{x, r}^{\tau}(y)=\int_\RR \int_\RR K_1(r, t_1, t_2)\rho\left(\frac{\mu(x, t_1)}{\tau} \right)\rho\left(\frac{\mu(y, t_2)}{\tau} \right)\frac{d\mu(t_1)}{\tau}\frac{d\mu(t_2)}{\tau} .$$
It is clear that {\rm(i)} {\rm(ii)} and {\rm(iii)} hold.  We will prove {\rm(iv)} next.
Let
$
\frac{\mu(y, t_2)}{\tau}=\frac{\mu(z, t_3)}{\tau}.
$
Notice that
\begin{eqnarray}\label{ee5*0}
\rho\left(\frac{\mu(y, t_2)}{\tau} \right)=\rho\left(\frac{\mu(z, t_3)}{\tau} \right) \ \ and \ \  \ \frac{d\mu(t_2)}{\tau}=\frac{d\mu(t_3)}{\tau}
\end{eqnarray}
hold when
\begin{eqnarray}\label{ee5*}
\frac{\mu(y, t_2)}{\tau}=\frac{\mu(z, t_3)}{\tau}.
\end{eqnarray}
Thus by Formula\,(\ref{ee5*0}), we could deduce the following Formula when Formula\,(\ref{ee5*}) holds:
\begin{eqnarray}\label{ee5}
\left|a_{x, r}^{\tau}(y)-a_{x, r}^{\tau}(z)\right|&=&\left|\int_\RR \int_\RR K_1(r, t_1, t_2)\rho\left(\frac{\mu(x, t_1)}{\tau} \right)\rho\left(\frac{\mu(y, t_2)}{\tau} \right)\frac{d\mu(t_1)}{\tau}\frac{d\mu(t_2)}{\tau}\right.
\\ \nonumber  &-&\left.\int_\RR \int_\RR K_1(r, t_1, t_3)\rho\left(\frac{\mu(x, t_1)}{\tau} \right)\rho\left(\frac{\mu(z, t_3)}{\tau} \right)\frac{d\mu(t_1)}{\tau}\frac{d\mu(t_3)}{\tau}\right|
\\ \nonumber &=&\left|\int_\RR \int_\RR\left(K_1(r, t_1, t_2)-K_1(r, t_1, t_3)\right)\rho\left(\frac{\mu(x, t_1)}{\tau} \right)\rho\left(\frac{\mu(z, t_3)}{\tau} \right)\frac{d\mu(t_1)}{\tau}\frac{d\mu(t_3)}{\tau}\right|.
\end{eqnarray}
Notice  that $supp\,\rho(x)\subseteq\{x: |x|<1\}$. Thus we could deduce that $d_\mu(x, t_1)<\tau$, $d_\mu(y, t_2)<\tau$ and $d_\mu(z, t_3)<\tau$.
If we choose $\tau$ small enough such that $\frac{d_\mu(y,z)}{r}\sim \frac{d_\mu(t_2,t_3)}{r}$, then
\begin{eqnarray*}
\left| K_1(r, t_1, t_2)-K_1(r, t_1, t_3)\right|
\leq C \Big(\frac{d_\mu(t_2,t_3)}{r}\Big)^\gamma
\leq C\Big(\frac{d_\mu(y,z)}{r}\Big)^\gamma.
\end{eqnarray*}
Then together with Formula\,(\ref{ee5}), we could conclude
\begin{eqnarray*}
\left|a_{x, r}^{\tau}(y)-a_{x, r}^{\tau}(z)\right| \leq C\Big(\frac{d_\mu(y,z)}{r}\Big)^\gamma.
\end{eqnarray*}
Thus {\rm(iv)} holds. We will prove {\rm(v)} next. Similar to Formula\,(\ref{ee5}), we could obtain:
\begin{eqnarray}\label{exexu2}
\left|a_{x, r}^{\tau}(y)-K_1(r, x, y)\right|&=&\left|\int_\RR \int_\RR K_1(r, t_1, t_2)\rho\left(\frac{\mu(x, t_1)}{\tau} \right)\rho\left(\frac{\mu(y, t_2)}{\tau} \right)\frac{d\mu(t_1)}{\tau}\frac{d\mu(t_2)}{\tau}\right.
\\ \nonumber  &-&\left.\int_\RR \int_\RR K_1(r, x, y)\rho\left(\frac{\mu(x, t_1)}{\tau} \right)\rho\left(\frac{\mu(z, t_3)}{\tau} \right)\frac{d\mu(t_1)}{\tau}\frac{d\mu(t_3)}{\tau}\right|
\\ \nonumber &=&\left|\int_\RR \int_\RR\left(K_1(r, t_1, t_2)-K_1(r, x, y)\right)\rho\left(\frac{\mu(x, t_1)}{\tau} \right)\rho\left(\frac{\mu(z, t_3)}{\tau} \right)\frac{d\mu(t_1)}{\tau}\frac{d\mu(t_3)}{\tau}\right|.
\end{eqnarray}
Notice that
\begin{eqnarray*}
\left| K_1(r, t_1, t_2)-K_1(r, x, y)\right| &\leq&C\left| K_1(r, t_1, t_2)-K_1(r, t_1, y)\right|+C\left| K_1(r, t_1, y)-K_1(r, x, y)\right|
\\ \nonumber &\leq&C\Big(\frac{d_\mu(t_2,y)}{r}\Big)^\gamma+C\Big(\frac{d_\mu(t_1,x)}{r}\Big)^\gamma
\\ \nonumber &\leq&C\Big(\frac{\tau}{r}\Big)^\gamma.
\end{eqnarray*}
Together with Formula\,(\ref{exexu2}), we could conclude
\begin{eqnarray}\label{exexu3}
\left|a_{x, r}^{\tau}(y)-K_1(r, x, y)\right|\leq C \Big(\frac{\tau}{r}\Big)^\gamma,
\end{eqnarray}
for $\tau$ small enough.
This proves our proposition.
\end{proof}

\begin{proposition}\label{no6}For $\displaystyle{ p> \frac{1}{1+\gamma}}, \ \  f\in L^1(\RR, \mu)$,  $1\geq\gamma>0$, there exists some $\beta$ with $\beta>\gamma$ such that  the following inequality holds:
$$\| f_{S\beta}^*\|_{L^p(\RR, \mu)}\leq c\|f_{1\bigtriangledown_{\gamma}}^\times \|_{L^p(\RR, \mu)}\ \ .$$
\end{proposition}
\begin{proof}
Let  $\phi\in SS_\beta$ first. Notice that $C_c(\RR, dx)$ is dense in $C_0(\RR, dx)$,  by Proposition\,\ref{sk6} and Proposition\,\ref{sk66}, $C_c(\RR, d_\mu x)$ is dense in $C_0(\RR, d_\mu x)$. By the fact that $K_1(r, x, y)=K_1(r, y, x)$ and $\left|\int_\RR K_1(r, x, y) d_\mu(y)/r\right|\geq m>0$, together with Proposition\,\ref{exx6}, there exists sequence $\{\phi_{x, r}^\tau(y): \phi_{x, r}^\tau(y) \in S(\RR, d_\mu (y))\}_\tau$ satisfying the following conditions:
\begin{eqnarray*}
\left\{ \begin{array}{cc}
                             \displaystyle{\phi_{x, r}^\tau(y)=\phi_{y, r}^\tau(x)}, \ \phi^\tau_{x, r}(y) \in S(\RR, d_\mu y) , \\

                             supp\,\phi_{x, r}^\tau(y) \subseteq B_\mu(x, r), \ \ \displaystyle{\left|\int_\RR \phi_{x, r}^\tau(y) d_\mu(y)/r\right|\geq m/2>0}\\
                              L(\phi_{x, r}^\tau(y), \gamma)\leq r^{-\gamma}, 0\leq\phi_{x, r}^\tau(y)\leq C, \hbox{for}\ r>0\\
                              \displaystyle{ \lim_{\tau\rightarrow0} \phi_{x, r}^\tau(y) = K_1(r, x, y)} .
                           \end{array}\right.
\end{eqnarray*}
Thus by  Proposition\,\ref{sk6}, Proposition\,\ref{sk66} and Proposition\,\ref{exx6}, there exists sequence $\{\phi_{x}^\tau(y): \phi_{x}^\tau(y) \in S(\RR, d y)\}_\tau$ satisfying:
\begin{eqnarray*}
\left\{ \begin{array}{cc}
                            \phi_{x, r}^\tau(y)= \phi_{x}^\tau\left(\frac{\mu(x, y)}{r} \right)=\phi_{y}^\tau\left(\frac{\mu(y, x)}{r} \right), \ \|\phi_{x}^\tau(t)\|_{L^{\infty}}\leq1, H^{\gamma}\phi_{x}^{\tau}(t)\lesssim 1\\
                            \phi_{x}^\tau(y) \in S(\RR, d y),\ \ \displaystyle{ \lim_{\tau\rightarrow0} \phi_{x}^\tau\left(\frac{\mu(x, y)}{r} \right) = K_1(r, x, y)}\\
                            \displaystyle{\left|\int_\RR \phi_{x}^\tau(t) dt\right|\geq \frac{m}{2}>0},\ \ supp\,\phi_{x}^\tau(t) \subseteq   [-1, 1]        .
                           \end{array}\right.
\end{eqnarray*}

Notice that $\left|\int t^\beta \phi_y^\tau(t)dt\right|\lesssim C_{ \beta} $, thus we could deduce the following inequality:
\begin{eqnarray}\label{tano1}
\sup_{\xi\in\RR}\left|\frac{d^{\beta}}{d\xi^\beta}(\SF\phi_y^\tau)(\xi)\right|\leq C_{ \beta},
\end{eqnarray}
where $C_{ \beta}$ is a
constant independent on $\tau$. Notice that $(\SF\phi_y^\tau)(0)=1$, thus by Formual\,(\ref{tano1}), we could also deduce that there exists a $k_o$ independent on $\tau$, such that
$$\left|(\SF\phi_y^\tau)(2^{-k_o}\xi)\right|\geq 1/2 \ \ \hbox{for}\,\, |\xi|\leq2.$$
Fix a function $\varphi\in S(\RR, dx)$ so that
\begin{eqnarray*}
\left\{ \begin{array}{cc}
\varphi(\xi)=0 \ \ \hbox{for}\,|\xi|\geq1 \\
\\
\varphi(\xi)=1\ \ \hbox{for}\,|\xi|\leq1/2.
\end{array}\right.
\end{eqnarray*}
The function $\varphi^k\in S(\RR, dx)$ is definted as:
\begin{eqnarray*}
\left\{ \begin{array}{cc}
\varphi^k(\xi)=\varphi(\xi) \ \ \hbox{for}\,k=0, \\
\\
\varphi^k(\xi)=\varphi(2^{-k}\xi)-\varphi(2^{1-k}\xi)\ \ \hbox{for}\,k\geq1.
\end{array}\right.
\end{eqnarray*}

We use $\eta_\tau^{k}$ to denote as
$$(\SF\eta_\tau^{k})(\xi)=\frac{\varphi^k(\xi)(\SF\phi)(\xi)}{(\SF\phi_y^\tau)(2^{-k}2^{-k_o}\xi)},$$
where $\SF$ is the Fourier transform. By the fact that $\sup_{\xi\in\RR}|\frac{d^{\beta}}{d\xi^\beta}(\SF\phi_y^\tau)(2^{-k_o}\xi)|\leq C_{ \beta, k_o}$ and \begin{eqnarray}\label{ee7}
\sup_{\xi\in\RR}\left||\xi|^{\alpha}\frac{d^{\beta}}{d\xi^\beta}(\SF\phi)(\xi)\right|\lesssim_{\alpha, \beta} 1,
\end{eqnarray} where $C_{ \beta, k_o}$ is a
constant independent on $\tau$, we could deduce that for any $M>0$, the following inequality holds:
\begin{eqnarray}\label{ee2}
\sup_{\xi\in\RR}\left||\xi|^{\alpha}\frac{d^{\beta}}{d\xi^\beta}(\SF\eta_\tau^{k})(\xi)\right|\leq C_{\alpha, \beta, M, k_o}2^{-kM},
\end{eqnarray}
where $C_{\alpha, \beta, M, k_o}$ is a
constant independent on $\tau$ and $k$. Then by Proposition\,\ref{no3}, for any $\,\phi(t)\in SS_{\beta}$ with $\int_\RR \phi(t)dt=1$, we could deduce:
\begin{eqnarray}\label{ee1}
\phi\left(\frac{\mu(x, y)}{r}\right)=\sum_{k=0}^{\infty}\int_\RR \eta_\tau^{k}\left(\frac{s}{r}\right)\phi_y^\tau\left(\frac{\mu(x, y)-s}{2^{-k}r}\right)\frac{ds}{2^{-k}r}.
\end{eqnarray}
Then by Formula\,(\ref{ee1}) with the fact that  $f\in L^1(\RR, \mu)$ we have
\begin{eqnarray}\label{ee3}
M_{\phi\beta}f(x)&=& \sup_{r>0}\left|\int_\RR f(y)\phi\left(\frac{\mu(x, y)}{r} \right)d\mu(y)\right|/r
\\&=&\nonumber C\sup_{r>0}\sum_{k=0}^{+\infty}\left|\int_\RR\int_\RR f(y)\eta_\tau^{k}\left(\frac{s}{r}\right)\phi_y^\tau\left(\frac{\mu(x, y)-s}{2^{-k}2^{-k_o}r}\right)\frac{ds}{2^{-k}2^{-k_o}r}d\mu(y)\right|/r
\\&\leq&\nonumber C \sum_{k=0}^{+\infty} \left|\int_\RR\eta_\tau^{k}\left(\frac{s}{r}\right)\left(1+\frac{|s|}{2^{-k-k_o}r} \right)^{N}\frac{ds}{r}\right| \sup_{ r>0, s\in\RR}\left|\int_\RR f(y)\phi_y^\tau\left(\frac{\mu(x, y)-s}{r} \right)\left(1+\frac{|s|}{r} \right)^{-N}\frac{d\mu (y)}{r}\right|.
\end{eqnarray}
By Formula\,(\ref{ee2}), we could deduce that $$\sum_{k=0}^{+\infty} \left|\int_\RR\eta_\tau^{k}\left(\frac{s}{r}\right)\left(1+\frac{|s|}{2^{-k-k_o}r} \right)^{N}\frac{ds}{r}\right|\leq C_{N, k_o} \sum_{k=0}^{\infty} 2^{-k} ,$$
 where $C_{N, k_o}$ is a constant independent on $\tau$. Together with Formula\,(\ref{ee3}), we could obtain:
\begin{eqnarray}\label{ee6}
M_{\phi\beta}f(x)&\lesssim&  \sup_{ r>0, s\in\RR} \left|\int_\RR f(y)\phi_y^\tau\left(\frac{\mu(x, y)-s}{r} \right)\left(1+\frac{|s|}{r} \right)^{-N}\frac{d\mu (y)}{r}\right|
\\&\lesssim&\nonumber   \left(\sup_{0\leq s< r}+\sum_{k=1}^{\infty}\sup_{ 2^{k-1}r\leq s< 2^{k}r}\right)\left|\int_\RR f(y)\phi_y^    \tau\left(\frac{\mu(x, y)-s}{r} \right)\left(1+\frac{|s|}{r} \right)^{-N}\frac{d\mu (y)}{r}\right|
\\&\lesssim&\nonumber   \sum_{k=0}^{+\infty}2^{-(k-1)N}\sup_{ 0\leq s< 2^{k}r}\left|\int_\RR f(y)\phi_y^\tau\left(\frac{\mu(x, y)-s}{r} \right)\frac{d\mu (y)}{r}\right|.
\end{eqnarray}
Thus by Formula\,(\ref{ee6}) the following holds:
\begin{eqnarray}\label{ee4}
f_{S\beta}^*(x)&= &\sup_{\phi\in SS_\beta}M_{\phi\beta}f(x)
\\&\lesssim&\nonumber C \sum_{k=0}^{+\infty}2^{-(k-1)N} \sup_{0\leq s< 2^{k}r}\left|\int_\RR f(y)\phi_y^\tau\left(\frac{\mu(x, y)-s}{r} \right)\frac{d\mu (y)}{r}\right|.
\end{eqnarray}
For a positive measure $\mu$ where $\mu(x, u)$ is a bijection on $\RR$,  let $s=\mu(x, u)$ with $d_\mu(x, u)<2^{k}r$.
We use $T(x, k, \tau)$, $\left(F^\tau f\right)(u, r)$ and $\left(K_1f\right)(u, r)$ to denote as:
$$T(x, k, \tau)= \sup_{0\leq s< 2^{k}r}\left|\int_\RR f(y)\phi_y^\tau\left(\frac{\mu(x, y)-s}{r} \right)\frac{d\mu (y)}{r}\right|=\sup_{0\leq d_\mu(x, u)< 2^{k}r}\left|\int_\RR f(y)\phi_{u, r}^\tau(y)\frac{d\mu (y)}{r}\right|,$$
and
$$\left(F^\tau f\right)(u, r)=\int_\RR f(y)\phi_{u, r}^\tau(y)\frac{d\mu (y)}{r},\ \ \ \ \ \left(K_1f\right)(u, r)=\int_\RR f(y)K_1(r, u, y)\frac{d\mu (y)}{r}.$$
$\int_\RR |T(x, k, \tau)|^pd\mu(x)<\infty$ and Formula\,(\ref{no1}) lead to
 \begin{eqnarray}\label{no7}
 \int_\RR |T(x, k, \tau)|^pd\mu(x)\leq c \left(1+2^k\right) \int_\RR |T(x, 0, \tau)|^pd\mu(x).
 \end{eqnarray}
For $N>1/p$,  we could obtain
\begin{eqnarray}\label{no9}
 \int_\RR |f_{S\beta}^*(x)|^pd\mu(x)\leq C_{p,n,\beta} \int_\RR |T(x, 0, \tau)|^pd\mu(x),
 \end{eqnarray}
where $C_{p,n,\beta}$ is a constant independent on $\tau$. By Formula\,(\ref{exexu3}) it is clear that (taking $\tau=\frac{r}{n}$)
\begin{eqnarray}
\left|\left(F^\tau f\right)(u, r)-\left(K_1f\right)(u, r)\right|&\leq& \int_\RR \left|f(y)\right||\phi_{u, r}^\tau(y)-K_1(r, u, y)|\frac{d\mu (y)}{r}
\\&\leq&\nonumber C_{\gamma} |M_\mu f(u)|\left(\frac{1}{n}\right)^{\gamma},
\end{eqnarray}
where $C_\gamma$ is dependent on $\gamma$,  and $M_\mu$ is the Hardy-Littlewood Maximal Operator.
 Let us set  $$\delta_n(u)=\left|\left(F^\tau f\right)(u, r)-\left(K_1f\right)(u, r)\right|.$$
Thus we could deduce the following:
$$\left\{x: \delta_n(x)>\alpha\right\}\subseteq \left\{x: M_\mu f(x) > \frac{1}{C_\gamma}n^{\gamma}\alpha \right\}.$$
Notice that $M_\mu$ is weak-(1, 1) bounded. Then the following holds for any $\alpha>0$:
$$\left|\left\{x: \delta_n(x)>\alpha\right\}\right|_{\mu}\leq \frac{C_\gamma}{\alpha} \|f\|_{L^1(\RR, \mu)}\left(\frac{1}{n}\right)^{\gamma}.$$
Thus $$\lim_{n\rightarrow+\infty} \left|\left\{x: \delta_n(x)>\alpha\right\}\right|_\mu=0 .$$
Thus there exists a sequence $\{\tau_j\}\subseteq \{\tau\}$ such that the following holds: $$\lim_{\tau_j\rightarrow0}\left(F^{\tau_j}f\right)(u, r)=\left(K_1f\right)(u, r),\ \ \ a.e.u\in\RR \ in\ \mu\ measure$$ for $f\in L^1(\RR, \mu)$.
Denote $$E=\{u\in\RR: \lim_{\tau_j\rightarrow0}\left(F^{\tau_j}f\right)(u, r)=\left(K_1f\right)(u, r)\}.$$
That   E is  dense in $\RR$ could be deduced from the fact $\left|E^c\right|_{\mu}=0$. Notice that for any $x_0\in\RR$ and any $\tau_j\in \{\tau_j\}$, there exists a $(u_0, r_0)$ with $r_0>0$, $u_0\in\RR$, $d_\mu(u_0, x_0)<r_0$ such that the following holds:
$$\left|\left(F^{\tau_j}f\right)(u_0, r_0)\right| \geq \frac{1}{2} |T(x_0, 0, \tau_j)|.$$
Because $\left(F^{\tau_j}f\right)(u, r_0)$ is a continuous function in $u$ variable and E is  dense in $\RR$. There exists a $\widetilde{u}_0\in E$ with $d_\mu(\widetilde{u}_0, x_0)<r_0$ such that
$$\left|\left(F^{\tau_j}f\right)(\widetilde{u}_0, r_0)\right| \geq \frac{1}{4}  |T(x_0, 0, \tau_j)|.$$
Thus we could deduce that
\begin{eqnarray}\label{exexu5}
\sup_{\{u\in E: d_\mu(u, x)<r\}} \left|\left(F^{\tau_j}f\right)(u, r)\right|\sim \sup_{\{u\in \RR: d_\mu(u, x)<r\}} \left|\left(F^{\tau_j}f\right)(u, r)\right|.
\end{eqnarray}
Formula\,(\ref{exexu5}) together with the dominated convergence theorem  (Proposition\,\ref{exx6}(iii)), we could conclude:
\begin{eqnarray}
\overline{\lim}_{\tau_j\rightarrow0}\int_\RR |T(x, 0, \tau_j)|^pd\mu(x) &\sim&\nonumber \overline{\lim}_{\tau_j\rightarrow0}\int_\RR \sup_{\{u\in E: d_\mu(u, x)<r\}} \left|\left(F^{\tau_j}f\right)(u, r)\right|^pd\mu(x)
\\&\leq&\nonumber C \int_\RR \overline{\lim}_{\tau_j\rightarrow0} \sup_{\{u\in E: d_\mu(u, x)<r\}} \left|\left(F^{\tau_j}f\right)(u, r)\right|^pd\mu(x)
\\&\leq&\nonumber C  \int_\RR\sup_{\{u\in E: d_\mu(u, x)<r\}} \left|\left(K_1f\right)(u, r)\right|^pd\mu(x)
\\&\leq& C  \int_\RR\sup_{\{u\in \RR: d_\mu(u, x)<r\}} \left|\left(K_1f\right)(u, r)\right|^pd\mu(x).
\end{eqnarray}
That is
$$\|f_{S\beta}^*\|_{L^p(\RR, \mu)}\leq c\|f_{1\bigtriangledown_{\gamma}}^\times \|_{L^p(\RR, \mu)}\ \ .$$
This proves our proposition.
\end{proof}

\begin{proposition}\label{kernel1}
$K_2(r, x, y)$ is the kernel in Definition\,\ref{us5}.  Then for any fixed $\alpha$ with $ 0<\alpha<\gamma\leq1,$ the following holds:

$$0\le \left|K_2(r, a, y)-K_2(r, b, y)\right| \le
C \Big(\frac{d_\mu(a,b)}{r}\Big)^{\alpha}\Big(1+\frac{d_\mu(x,y)}{r}\Big)^{-(\gamma-\alpha)-1}, $$
and
\begin{eqnarray*}
& &\left|\left(K_2(r, a, y)-K_2(r, b, y)\right)-\left(K_2(r, a, z)-K_2(r, b, z)\right)\right|
\\&\leq & C \Big(\frac{d_\mu(a,b)}{r}\Big)^{\alpha}\Big(\frac{d_\mu(y,z)}{r}\Big)^{\gamma-\alpha}\Big(1+\frac{d_\mu(x,y)}{r}\Big)^{-2(\gamma-\alpha)-1},
\end{eqnarray*}
for $d_\mu(a,
b)\lesssim r $, $\frac{d_\mu(y,
z)}{r}\leq C_3\min\{1+\frac{d_\mu(a,y)}{r}, 1+\frac{d_\mu(a,z)}{r}\} $, $x\in B_{\mu}(a, 2r)\bigcap B_{\mu}(b, 2r)$.
\end{proposition}
\begin{proof}
First,  we consider the case when $$d_\mu(a, b)\leq d_\mu(y, z).$$
From the fact that $d_\mu(a,
b)\lesssim r $, $\frac{d_\mu(y,
z)}{r}\leq C_3\min\{1+\frac{d_\mu(a,y)}{r}, 1+\frac{d_\mu(a,z)}{r}\} $, the following relations could be obtained:
\begin{eqnarray}\label{1120}
1+\frac{d_\mu(a,y)}{r}\sim1+\frac{d_\mu(b,y)}{r}, 1+\frac{d_\mu(a,z)}{r}\sim1+\frac{d_\mu(b,z)}{r},\,and\, 1+\frac{d_\mu(a,z)}{r}\sim1+\frac{d_\mu(a,y)}{r}.
\end{eqnarray}
Notice that $$K_2(r, x, y)=K_2(r, y, x).$$
Then we could get
\begin{eqnarray}\label{1119}
| K_2(r, a, y)-K_2(r, b, y))|&\le&
C\Big(\frac{d_\mu(a,b)}{r}\Big)^{\gamma}\Big(1+\frac{d_\mu(a,y)}{r}\Big)^{-2\gamma-1}
\\&\le&\nonumber
C\Big(\frac{d_\mu(a,b)}{r}\Big)^{\gamma}\Big(1+\frac{d_\mu(a,y)}{r}\Big)^{-\gamma-\alpha}\Big(1+\frac{d_\mu(a,y)}{r}\Big)^{-(\gamma-\alpha)-1}
\\&\le&\nonumber
C \Big(\frac{d_\mu(a,b)}{r}\Big)^{\alpha}\Big(1+\frac{d_\mu(a,y)}{r}\Big)^{-(\gamma-\alpha)-1}.
\end{eqnarray}
Also we could obtain
$$| K_2(r, a, y)-K_2(r, b, y))|\le
C\Big(\frac{d_\mu(a,b)}{r}\Big)^{\gamma}\Big(1+\frac{d_\mu(a,y)}{r}\Big)^{-2\gamma-1},$$
and
$$| K_2(r, a, z)-K_2(r, b, z))|\le
C\Big(\frac{d_\mu(a,b)}{r}\Big)^{\gamma}\Big(1+\frac{d_\mu(a,z)}{r}\Big)^{-2\gamma-1}.$$
Together with Formula\,(\ref{1120}), we could conclude
\begin{eqnarray*}
& &\left|\left(K_2(r, a, y)-K_2(r, b, y)\right)-\left(K_2(r, a, z)-K_2(r, b, z)\right)\right|
\\&\leq & C \Big(\frac{d_\mu(a,b)}{r}\Big)^{\gamma}\Big(1+\frac{d_\mu(a,y)}{r}\Big)^{-2\gamma-1}.
\end{eqnarray*}
By the fact $d_\mu(a, b)\leq d_\mu(y, z)$ and $\displaystyle{1\lesssim 1+\frac{d_\mu(a,y)}{r}}$, we could obtain:
$$\Big(\frac{d_\mu(a,b)}{r}\Big)^{\gamma}\Big(1+\frac{d_\mu(a,y)}{r}\Big)^{-2\gamma-1}\lesssim\Big(\frac{d_\mu(a,b)}{r}\Big)^{\alpha}\Big(\frac{d_\mu(y,z)}{r}\Big)^{\gamma-\alpha}\Big(1+\frac{d_\mu(a,y)}{r}\Big)^{-2(\gamma-\alpha)-1}$$
Then for  $d_\mu(a, b)\leq d_\mu(y, z),$ the Formula
\begin{eqnarray}\label{exu1}
& &\left|\left(K_2(r, a, y)-K_2(r, b, y)\right)-\left(K_2(r, a, z)-K_2(r, b, z)\right)\right|\nonumber
\\&\leq & C \Big(\frac{d_\mu(a,b)}{r}\Big)^{\alpha}\Big(\frac{d_\mu(y,z)}{r}\Big)^{\gamma-\alpha}\Big(1+\frac{d_\mu(a,y)}{r}\Big)^{-2(\gamma-\alpha)-1}
\end{eqnarray}
holds.
In a similar way, we will obtain the Formula\,(\ref{exu1}) for the case when $d_\mu(a, b)\geq d_\mu(y, z).$
Notice that by Formula\,(\ref{1120}),
$$| K_2(r, a, y)-K_2(r, a, z))|\le
C\Big(\frac{d_\mu(y,z)}{r}\Big)^{\gamma}\Big(1+\frac{d_\mu(a,y)}{r}\Big)^{-2\gamma-1},$$
and
\begin{eqnarray*}
| K_2(r, b, y)-K_2(r, b, z))|&\le&
C\Big(\frac{d_\mu(y,z)}{r}\Big)^{\gamma}\Big(1+\frac{d_\mu(b,y)}{r}\Big)^{-2\gamma-1}
\\&\le&
C\Big(\frac{d_\mu(y,z)}{r}\Big)^{\gamma}\Big(1+\frac{d_\mu(a,y)}{r}\Big)^{-2\gamma-1}
\end{eqnarray*}
hold. Then we could obtain
\begin{eqnarray*}
& &\left|\left(K_2(r, a, y)-K_2(r, b, y)\right)-\left(K_2(r, a, z)-K_2(r, b, z)\right)\right|
\\&\leq & C \Big(\frac{d_\mu(y,z)}{r}\Big)^{\gamma}\Big(1+\frac{d_\mu(a,y)}{r}\Big)^{-2\gamma-1}.
\end{eqnarray*}
By the fact $d_\mu(a, b)\geq d_\mu(y, z)$ and $\displaystyle{1\lesssim 1+\frac{d_\mu(a,y)}{r}}$, the following holds:
$$\Big(\frac{d_\mu(y,z)}{r}\Big)^{\gamma}\Big(1+\frac{d_\mu(a,y)}{r}\Big)^{-2\gamma-1}\lesssim\Big(\frac{d_\mu(a,b)}{r}\Big)^{\alpha}\Big(\frac{d_\mu(y,z)}{r}\Big)^{\gamma-\alpha}\Big(1+\frac{d_\mu(a,y)}{r}\Big)^{-2(\gamma-\alpha)-1}.$$
Then for  $d_\mu(a, b)\geq d_\mu(y, z),$ we could get
\begin{eqnarray}\label{exu2}
& &\left|\left(K_2(r, a, y)-K_2(r, b, y)\right)-\left(K_2(r, a, z)-K_2(r, b, z)\right)\right|\nonumber
\\&\leq & C \Big(\frac{d_\mu(a,b)}{r}\Big)^{\alpha}\Big(\frac{d_\mu(y,z)}{r}\Big)^{\gamma-\alpha}\Big(1+\frac{d_\mu(a,y)}{r}\Big)^{-2(\gamma-\alpha)-1}.
\end{eqnarray}
By the fact that $x\in B_{\mu}(a, 2r)\bigcap B_{\mu}(b, 2r)$, we could deduce that:
\begin{eqnarray}\label{1121}
1+\frac{d_\mu(a,y)}{r}\sim1+\frac{d_\mu(x,y)}{r}.
\end{eqnarray}
Formulas\,(\ref{1119})\,(\ref{exu1})\,(\ref{exu2})\,(\ref{1121})\,yeald the Proposition.
\end{proof}

\begin{proposition}\label{no11}
For any $0<\gamma\leq1$, $f\in L^1(\RR, \mu)$,  if the following inequality holds
$$\|f_{i\bigtriangledown_\gamma}^\times\|_{L^p(\RR, \mu)}\sim\|f_{\gamma}^* \|_{L^p(\RR, \mu)}$$
then for $\displaystyle{\ 1\geq p> \frac{1}{1+\gamma}}, $ we could deduce  that:
$$\|f_{i\bigtriangledown_{\gamma}}^\times\|_{L^p(\RR, \mu)}\leq C\|f_{i}^\times \|_{L^p(\RR, \mu)},$$
where $C$ is dependent on $p$ and $\gamma$, and $i=1, 2$.
\end{proposition}
\begin{proof}We will only prove the proposition when $i=2$.
For any fixed $\alpha$ satisfying $ 0<\alpha<\gamma$ and $\displaystyle{\ p> \frac{1}{1+\gamma-\alpha}}$, Let $F$ denote as: $$F=\bigg\{x\in\RR: f_{\gamma-\alpha}^*(x)\leq\sigma f_{2\bigtriangledown_{\gamma}}^\times(x)\bigg\}.$$
By Proposition\,\ref{sk} and Proposition\,\ref{H spa}, we could deduce that the following holds for  $f\in L^1(\RR, \mu)$:
$$\|f_{\gamma-\alpha}^*\|_{L^p(\RR, \mu)}\sim_{\gamma, \alpha}\|f_{S(\gamma-\alpha)}^*\|_{L^p(\RR, \mu)}\sim_{\gamma, \alpha}\|f_{S\gamma}^*\|_{L^p(\RR, \mu)}\sim_{\gamma, \alpha}\|f_{\gamma}^*\|_{L^p(\RR, \mu)}.$$
Then it is clear that
\begin{eqnarray}\label{sv1}
\int_{F^c} |f_{2\bigtriangledown_{\gamma}}^\times(x)|^pd\mu(x)\leq \frac{C}{\sigma^p} \int_{F^c}| f_{\gamma-\alpha}^*(x)|^p d\mu(x)\leq \frac{C_{\gamma, \alpha}'}{\sigma^p} \int_{\RR}| f_{\gamma}^*(x)|^p d\mu(x)\leq \frac{C_{\gamma, \alpha}'}{\sigma^p} \int_{\RR}|f_{2\bigtriangledown_{\gamma}}^\times(x)|^p d\mu(x).
\end{eqnarray}\\
Choosing $\sigma^p\geq2C_{\gamma, \alpha}'$, we could have
\begin{eqnarray}\label{ex5}
 \int_{\RR}|f_{2\bigtriangledown_{\gamma}}^\times(x)|^p d\mu(x)\lesssim\int_{F} |f_{2\bigtriangledown_{\gamma}}^\times(x)|^pd\mu(x).
\end{eqnarray}
We use $D f(x)$ and $F(x, r)$ to denote  as: $$ D f(x)=\sup_{r>0}\left|\int_\RR f(t)K_2(r, x, t)\frac{d\mu(t)}{r}\right|,\ \ \ F(x, r)=\int_\RR f(t)K_2(r, x, t)\frac{d\mu(t)}{r}.$$
Next, we will show that for any $q>0$,
\begin{eqnarray}\label{ex6}
f_{2\bigtriangledown_{\gamma}}^\times(x)\leq C \left[M_\mu\left(D f\right)^q(x)\right]^{1/q} \ \ \hbox{for}\ x\in F,
\end{eqnarray}
where $M_\mu$ is the Hardy-Littlewood maximal operator. For any fixed $x_0\in F$,  there exists $(u_0, r_0)$ satisfying $d_\mu(u_0, x_0)<r_0$ such that  the following inequality holds:
\begin{eqnarray}\label{ex7}
 \left|F(u_0, r_0)\right|>\frac{1}{2} f_{2\bigtriangledown_{\gamma}}^\times(x_0).
\end{eqnarray}
Choosing $\delta<1$ small enough and  $u$ satisfying $d_\mu(u, u_0)< \delta r_0$,  we could deduce that
\begin{eqnarray*}
|F(u, r_0)-F(u_0, r_0)|&=&\left|\int_{\RR} f(y)K_2(r_0, u, y) d\mu(y)/r_0-\int_{\RR} f(y)K_2(r_0, u_0, y)d\mu(y)/r_0  \right|
\\&\leq& \left|  \int_{\RR} f(y) \left( K_2(r_0, u, y)-K_2(r_0, u_0, y)\right)d\mu(y)/r_0                 \right|.
\end{eqnarray*}
We could consider $\left( K_2(r_0, u, y)-K_2(r_0, u_0, y)\right)$ as a new kernel. By Proposition\,\ref{kernel1} and Proposition\,\ref{im1}, we could obtain:
\begin{eqnarray*}
|F(u, r_0)-F(u_0, r_0)|
\leq C\delta^{\alpha} f_{\gamma-\alpha}^*(x_0)\leq C\delta^{\alpha}\sigma f_{2\bigtriangledown_{\gamma}}^\times(x_0) \ \ \hbox{for}\ x_0\in F.
\end{eqnarray*}
Taking $\delta$ small enough such that $ C\delta^{\alpha}\sigma\leq1/4$, we obtain $$\left|F(u, r_0)\right|\geq \frac{1}{4}f_{2\bigtriangledown_{\gamma}}^\times(x_0) \ \ \hbox{for}\ u\in B_\mu(u_0, \delta r_0 ).$$
Thus the following inequality holds: for any $x_0\in F$,
\begin{eqnarray*}
 \left|f_{2\bigtriangledown_{\gamma}}^\times(x_0) \right|^q &\leq& \left|\frac{1}{B_\mu(u_0, \delta r_0 )}\right| \int_{B_\mu(u_0, \delta r_0 )}  4^q|F(u, r_0)|^q d\mu(u)
 \\&\leq&   \left|\frac{B_\mu(x_0, (1+\delta) r_0 )}{B_\mu(u_0, \delta r )}\right|  \left|\frac{1}{B_\mu(x_0, (1+\delta) r_0 )}\right|\int_{B_\mu(x_0, (1+\delta) r_0)}  4^q|F(u, r_0)|^qd\mu(u)
\\&\leq&\frac{1+\delta}{\delta}\left|\frac{1}{B_\mu(x_0, (1+\delta) r_0)}\right|\int_{B_\mu(x_0, (1+\delta) r_0)}  4^q|F(u, r_0)|^qd\mu(u)
\\&\leq& C M_\mu[(D f)^q](x_0)
\end{eqnarray*}
C is independent on $x_0$. Finally, using the maximal theorem for $M_\mu$ when $q<p$ leads to
\begin{eqnarray}\label{3}
\int_{F} \left|f_{2\bigtriangledown_{\gamma}}^\times(x)d\mu(x) \right|^p dx\leq C \int_{\RR}\left\{M_\mu[(D f)^{q}](x)\right\}^{p/q}d\mu(x) \leq C \int_{\RR}\left|f_{2}^\times(x)\right|^{p}d\mu(x).
\end{eqnarray}
Thus for any fixed $\alpha$ satisfying $ 0<\alpha<\gamma$ and $\displaystyle{\ p> \frac{1}{1+\gamma-\alpha}}$, the above Formula\,(\ref{3}) combined with Formula\,(\ref{ex5}) leads to
\begin{eqnarray}\label{4}
\|f_{2\bigtriangledown_{\gamma}}^\times\|_{L^p(\RR, \mu)}\leq C\|f_{2}^\times\|_{L^p(\RR, \mu)}\ \ ,
\end{eqnarray}
where $C$ is dependent on $p$ and $\alpha$. Next we will remove the number $\alpha$.  For any $\displaystyle{\ p> \frac{1}{1+\gamma}}$, let $p_0=\frac{1}{2}\left(p+\frac{1}{1+\gamma}\right)$ with $\displaystyle{\ p> p_0>\frac{1}{1+\gamma}}$ and let $\alpha= 1+\gamma -\frac{1}{p_0}$. Thus it is clear that $$p_0=\frac{1}{1+\gamma-\alpha},\ \ \ p>p_0.$$  Thus by Formula\,(\ref{4}),   we could obtain the following inequality holds for $\displaystyle{\ 1\geq p> \frac{1}{1+\gamma}}$
$$\|f_{2\bigtriangledown_{\gamma}}^\times(x)\|_{L^p(\RR, \mu)}\leq C\|f_{2}^\times(x) \|_{L^p(\RR, \mu)}\ \ $$
$C$ is dependent on $p$ and $\gamma$. This proves the Proposition.
\end{proof}
At last we will prove the following Proposition:
\begin{proposition}\label{important1} For  $\displaystyle{\frac{1}{1+\gamma}<p\leq1}$, $0<\gamma\leq1$, $f\in L^1(\RR, \mu)$, there exists $\beta>0$, such that the following conditions are equivalent:

{\rm(i)} \  $f_{S\beta}^*\in L^p(\RR, \mu) $.

{\rm(ii)} \ There is a $\phi(x)\in SS_{\beta} $ satisfying $\int \phi(x) dx \neq 0$ so that  $M_{\phi\beta}f(x) \in L^p(\RR, \mu) $.

{\rm(iii)} \  $f_{1\bigtriangledown_{\gamma}}^\times(x)=\sup_{d_\mu(x, y)<r}|F_{1}(r, y, f)| \in L^p(\RR, \mu)$.

{\rm(iv)}  $f_1^\times(x)=\sup_{r>0}|F_1(r, x, f)| \in L^p(\RR, \mu)$.

{\rm(v)} \  $f_{\gamma}^*\in L^p(\RR, \gamma) $.
\end{proposition}
\begin{proof}
${\rm(i)} \Rightarrow{\rm(ii)}$ is obvious.
${\rm(ii)} \Rightarrow{\rm(i)}$ is deduced from Proposition\,\ref{no4} and Proposition\,\ref{no5}.
${\rm(i)} \rightleftharpoons{\rm(v)}$ is deduced from Proposition\,\ref{H spa}.
$ {\rm(iii)}\Rightarrow{\rm(i)}$ is deduced from Proposition\,\ref{no6}.
$ {\rm(iv)}\Rightarrow{\rm(iii)}$ is deduced from Proposition\,\ref{no11}, Proposition\,\ref{H spa}, Proposition\,\ref{im1} and Proposition\,\ref{no6}.
$ {\rm(iii)}\Rightarrow{\rm(iv)}$ is obvious.
$ {\rm(v)}\Rightarrow{\rm(iii)}$ is deduced from Proposition\,\ref{im1}.
This proves the proposition.
\end{proof}
We define $H_{\mu}^p(\RR)$ and $\widetilde{H}_{\mu}^p(\RR)$  as:
\begin{definition}[$\mathbf{\widetilde{H}_{\mu}^p(\RR)}$ and $\mathbf{H_{\mu}^p(\RR)}$]\label{ha}
 $H_{\mu}^p(\RR)$ is defined as:
$$H_{\mu}^p(\RR)\triangleq\left\{g\in S'(\RR, d_{\mu}x): g_{S\beta}^*(x)\in L^p(\RR, \mu), \ \hbox{for}\ any\ \beta>p^{-1}-1 \right\}.$$
And its norm is is given by
$$ \|g\|_{H_{\mu}^p(\RR)}^p =\int_\RR |g_{S\beta}^*(x)|^p d\mu(x).$$
 $\widetilde{H}_{\mu}^p(\RR)$ is defined as:
$$\widetilde{H}_{\mu}^p(\RR)\triangleq\left\{g\in L^1(\RR, \mu): g_{S\beta}^*(x)\in L^p(\RR, \mu), \hbox{for\ any}\  \beta>p^{-1}-1 \right\}.$$
\end{definition}
From Theorem\,\ref{H spa1}, we could know that $H_{\mu}^p(\RR)$ space is the completion of  $\widetilde{H}_{\mu}^p(\RR)$  with  $\|\cdot\|_{H_{\mu}^p(\RR)}^p$ norm. Thus by Proposition\,\ref{important1} and Hahn-Banach Theorem, we could deduce the following:
\begin{theorem}\label{important} For  $\displaystyle{\frac{1}{1+\gamma}<p\leq1}$, $0<\gamma\leq1$, $f\in S'(\RR, d_{\mu}x)$, there exists $\beta>0$, such that the following conditions are equivalent:

{\rm(i)}  $f_{S\beta}^*\in L^p(\RR, \mu) $;

{\rm(ii)}  There is a $\phi(x)\in SS_{\beta} $ satisfying $\int \phi(x) dx \neq 0$ so that  $M_{\phi\beta}f(x) \in L^p(\RR, \mu) $;

{\rm(iii)} $f_{1\bigtriangledown_{\gamma}}^\times(x)=\sup_{d_\mu(x, y)<r}|F_{1}(r, y, f)| \in L^p(\RR, \mu)$;

{\rm(iv)}  $f_1^\times(x)=\sup_{r>0}|F_1(r, x, f)| \in L^p(\RR, \mu)$;

{\rm(v)}  $f_{\gamma}^*\in L^p(\RR, \mu) $;

{\rm(vi)} $H_{\mu}^p(\RR)$ space is the completion of  $\widetilde{H}_{\mu}^p(\RR)$  with  $\|\cdot\|_{H_{\mu}^p(\RR)}^p$ norm.
\end{theorem}

\section{ Hardy spaces associated with the Dunkl setting}\label{sss5}
In this Section we will discuss the Hardy spaces associated with the one dimensional  Dunkl setting. In section\.\S\ref{area}, we will give a real characterization of  $H^p_{\lambda}(\RR^2_+)$.  We will use another way different from Burkholder-Gundy-Silverstein in\,\cite{BGS}, in a very simple way. In section\.\S\ref{homo},  we will prove  that $H_{\lambda}^p(\RR)$ is a kind of Homogeneous Hardy spaces  for $\frac{1}{1+\gamma_\lambda}<p\leq1$, then we could obtain the real-variable method of $H_{\lambda}^p(\RR)$ by the theory of  Homogeneous Hardy spaces.
\subsection{Real Parts of  function in $H_{\lambda}^p(\RR_+^2)$ and  maximal function }\label{area}

\begin{definition}\cite{JianQuan Liao 1}\cite{ZhongKai Li 3}\label{Poisson-a}
For $f\in L_{\lambda}^1(\RR)\bigcap L_{\lambda}^{\infty}(\RR)$, $x\in\RR, \ y\in(0,\infty)$, we can define $\lambda$-Possion integral and conjugate $\lambda$-Poisson integral by
\begin{eqnarray*}
(Pf)(x,y)=(f\ast_{\lambda}P_y)(x)=c_{\lambda}\int_{\RR}f(t)(\tau_xP_y)(-t)|t|^{2\lambda}dt, \\
(Qf)(x,y)=(f\ast_{\lambda}Q_y)(x)=c_{\lambda}\int_{\RR}f(t)(\tau_xQ_y)(-t)|t|^{2\lambda}dt,
\end{eqnarray*}
where $\lambda$-Poisson kernel $(\tau_xP_y)(-t)$ has the  representation
\begin{eqnarray}\label{D-Poisson-ker-11}
(\tau_xP_y)(-t)=
\frac{\lambda\Gamma(\lambda+1/2)}{2^{-\lambda-1/2}\pi}\int_0^\pi\frac{y(1+{\rm
sgn}(xt)\cos\theta)
}{\big(y^2+x^2+t^2-2|xt|\cos\theta\big)^{\lambda+1}}\sin^{2\lambda-1}\theta
d\theta,
\end{eqnarray}

and $(\tau_xQ_y)(-t)$ is the conjugate $\lambda$-Poisson kernel, with the following representation:
\begin{eqnarray}\label{D-conjugate-Poisson-ker-1}
(\tau_xQ_y)(-t)=
\frac{\lambda\Gamma(\lambda+1/2)}{2^{-\lambda-1/2}\pi}\int_0^\pi\frac{(x-t)(1+{\rm
sgn}(xt)\cos\theta)
}{\big(y^2+x^2+t^2-2|xt|\cos\theta\big)^{\lambda+1}}\sin^{2\lambda-1}\theta
d\theta.
\end{eqnarray}

The  maximal functions are: $Q_{\nabla}^*f(x)=\sup_{|s-x|<y}|(Qf)(s,y)|$, $P_{\nabla}^*f(x)=\sup_{|s-x|<y}|(Pf)(s,y)|$, and $F_{\nabla}^*(x)=\sup_{|s-x|<y}|F(s, y)|$.
\end{definition}
\begin{proposition}\label{ss}\cite{ZhongKai Li 3}Let $F\in H_{\lambda}^p(\RR_+^2)$ and $f(x)\in L_{\lambda}^p(\RR)$, then the following hold:

{\rm(i)} For $1<p<\infty$,  $\|Q_{\nabla}^*f\|_{L^p_{\lambda}}\leq c^1_p \|f\|_{L^p_{\lambda}}$, $\|P_{\nabla}^*f\|_{L^p_{\lambda}}\leq c^2_p \|f\|_{L^p_{\lambda}}$.\

{\rm(ii)}For $\frac{2\lambda}{2\lambda+1}<p$, $F\in H_{\lambda}^p(\RR_+^2)$ if and only if $F_{\nabla}^*\in L_{\lambda}^p(\RR)$, and moreover $\|F\|_{H_{\lambda}^p}\geq\|F_{\nabla}^*\|_{L_{\lambda}^p}\geq c\|F\|_{H_{\lambda}^p} $.

{\rm(iii)} For $1\leq p<\infty$,  F(x, y) has boundary values, and let f(x) to be the real part of the boundary values of F(x, y) satisfying $F(x, y)=Pf(x, y)+iQf(x, y).$

{\rm(iv)}For  $1\leq p<\infty$,   $Pf(x, y)$ and  $Qf(x, y)$ satisfy the generalized Cauchy-Riemann system\,(\ref{a c r0})\,on $\RR_+^2$.

\end{proposition}

\begin{proposition}\cite{JianQuan Liao 1}\cite{ZhongKai Li 3}\label{k1}
Let $F(x, y)\in H_{\lambda}^p(\RR_+^2)$, $f(x)$ to be the boundary value of $F(x, y)$ for $p> p_0=\frac{2\lambda}{2\lambda+1}$,  then the following hold:

(i)For almost every $x\in\RR$, $\lim F(t,y)=f(x)$ exists as (t, y) approaches the point (x, 0) nontangentially.

(ii)  $\lim_{y\rightarrow0+}\|F(.,y)-f\|_{L_{\lambda}^p}=0$, for $\frac{2\lambda}{2\lambda+1}<p$.  $\|F\|_{H_{\lambda}^p}=\|f\|_{L_{\lambda}^p}$, for $1\leq p$.  $\|F\|_{H_{\lambda}^p}\geq\|f\|_{L_{\lambda}^p}\geq2^{1-2/p}\|F\|_{H_{\lambda}^p}$, for $\frac{2\lambda}{2\lambda+1}<p<1$, where $\|f\|_{L_{\lambda}^p}=(c_{\lambda}\int_{\RR}|f(x)|^p|x|^{2\lambda}dx)^{1/p}$.

(iii)Let $p>\frac{2\lambda}{2\lambda+1} $,  $p_1>\frac{2\lambda}{2\lambda+1} $ , $F(x , y)\in H_{\lambda}^p(\RR_+^2)$, and $f\in L_{\lambda}^{p_1}(\RR)$, then $F(x , y)\in H_{\lambda}^{p_1}(\RR_+^2).$

\end{proposition}

\begin{proposition}\cite{JianQuan Liao 1}\cite{ZhongKai Li 3}\label{sss1}
For simplicity, we write $\tau_tu(x, y)=\left[\tau_t\left(u(.,y)\right)\right](x)$.

(1) If u is twice continuously differentiable on $\RR_+^2$ and satisfies $\triangle_{\lambda}u=0$, then for $(x_0, y_0)\in\RR^2_+$,  $0<r<y_0$, we have
\begin{eqnarray*}
u(x_0, y_0)=\sigma_\lambda\int_{-\pi}^\pi (\tau_{r\cos\theta}u)(x_0, y_0+r\sin\theta)|\cos\theta|^{2\lambda}d\theta,
\end{eqnarray*}
where $\sigma_\lambda^{-1}=\int_{-\pi}^\pi |\cos\theta|^{2\lambda}d\theta =2\sqrt{\pi}\Gamma(\lambda+\frac{1}{2})/\Gamma(\lambda+1)$.

(2) For $f\in S(\RR, dx)$, for fixed $t\in \RR$, the function $x\rightarrow \tau_tf(x) \in S(\RR)$, and the following holds:
$$D_t(\tau_tf(x))=D_x(\tau_tf(x))=\left(\tau_t(Df)\right)(x).$$

(3) For $f\in L_\lambda^\infty(\RR)$, the following holds for $t\in\RR$ (We could use $\|.\|_{\infty}$ instead of $\|.\|_{L_\lambda^\infty(\RR)}$ for convenience):
$$\|\tau_t f\|_{L_\lambda^\infty(\RR)}\leq 4 \| f\|_{L_\lambda^\infty(\RR)} .$$

(4) For $1<p<\infty$, $u(x, y)$ is a $\lambda$-harmonic function on $\RR^2_+$. $u(x, y)$ is the $\lambda$-Poisson integral of some function $f(x)\in L^p_\lambda(\RR)$ if and only if $u(x, y)$ satisfies the following:
$$\sup_{t>0}c_\lambda\int_\RR |u(x, y)|^p |x|^{2\lambda}dx<\infty.$$
\end{proposition}

In \cite{T0}, the dual of intertwining operator are introduced as follows.
\begin{definition}[$\mathbf{Dual\  of\  intertwining\  operator}$]\cite{T0}
We use $V_\lambda^t$ to denote as the dual of intertwining operator:
$$V_\lambda^t(f)=\SF^{-1}\SF_\lambda(f),$$
$\left(V_\lambda^t\right)^{-1}$ to denote as:
$$\left(V_\lambda^t\right)^{-1}(f)=\SF_\lambda^{-1}\SF(f).$$
\end{definition}
The  properties of the dual of intertwining operator are as follows:
\begin{proposition}\label{u0}\cite{T0}
{\rm(i)}  \ $V_\lambda^t$ is a topological automorphism on $S(\RR, dx)$;

{\rm(ii)} \ If $\,suppf\subseteq B(0, a)$, then $\,supp V_\lambda^t(f)\subseteq B(0, a)$ and $\,supp \left(V_\lambda^t\right)^{-1}(f)\subseteq B(0, a)$;

{\rm(iii)}  \   $\displaystyle{V_\lambda^t(Df)(x)=\frac{d}{dx}V_\lambda^t(f)(x)}$ for any $f\in S(\RR, dx)$, where $D$ is the Dunkl operator.
\end{proposition}
By Proposition\,\ref{u0}, we could deduce the following Proposition\,\ref{uu5}:
\begin{proposition}\label{uu5}
For any $\phi\in S(\RR, dx)$, $$\sup_{x\in\RR}\left||x|^\alpha D^\beta\phi(x)\right|<\infty.$$
\end{proposition}

\begin{theorem}\label{u2}
Let $u(x, y)$ to be a $\lambda$-harmonic function satisfying  $u_{\nabla}^* \in L^p_{\lambda}(\RR)$. For $\frac{2\lambda}{2\lambda+1}<p <\infty$, there exists a $\lambda$-analytic function $F(z)\in H_{\lambda}^p(\RR_+^2)$ satisfying $u(x, y)=ReF(z)$ and
$$\|F\|_{H_{\lambda}^p(\RR_+^2)}\sim \|u_{\nabla}^*\|_{L_{\lambda}^p(\RR)}.$$
\end{theorem}
\begin{proof}

Case 1 $1<p<\infty$: It is clear that part\,(2) of this Theorem holds for $1<p<\infty$ by Proposition\,\ref{sss1}(4) and Proposition\,\ref{ss}(i)(ii)(iv).

Case 2  $\frac{2\lambda}{2\lambda+1}<p\leq 1$:
Notice that the following inequality holds for any $h\in \{h:|x-h|<t\}$:
$$|u(x, t)|\leq \sup_{|h-s|<l}|u(s, l)|.$$
We could also deduce  that $\int_{\{h: |x-h|< t\}}|h|^{2\lambda}dh\sim |x|^{2\lambda}|t|\gtrsim |t|^{2\lambda+1}$ for $0<t\leq |x|/2$, and  $\int_{\{h: |x-h|< t\}}|h|^{2\lambda}dh\sim\int_{\{h: |h|< t\}}|h|^{2\lambda}dh\sim |t|^{2\lambda+1}$ for $t\geq |x|/2$.
Then for $0<t$, we could have:
\begin{eqnarray*}
|u(x, t)|^p &\lesssim& \frac{1}{\int_{\{h: |x-h|< t\}}|h|^{2\lambda}dh}\int_{\{h: |x-h|< t\}}  \sup_{|h-s|<l}|u(s, l)|^p |h|^{2\lambda}dh
\\ &\lesssim& \frac{1}{t^{2\lambda+1}}\int_{\{h: |x-h|< t\}}  \sup_{|h-s|<l}|u(s, l)|^p |h|^{2\lambda}dh.
\end{eqnarray*}
Thus we could deduce the following Formula\,(\ref{sk4})  holds:
\begin{equation}\label{sk4}
|u(x, y)|\lesssim\|u_{\nabla}^*\|_{L^p_{\lambda}(\RR)} y^{-(2\lambda+1)/p}.
\end{equation}
 We define $v(x, y)$ as the conjugate $\lambda$-harmonic function of $u(x, y)$ as following:
\begin{equation}\label{sk8}
 v(x, y)=-\int_{y}^{+\infty} D_x u(x, r)dr.
\end{equation}
Next we will show that $v(x, y)$ is a well defined function. We use $\psi_{(\rho)}(\zeta,\xi)$ ($0<\rho<\infty$) to denote a radial  positive function on $\RR^2$ satisfying
$$supp\,\psi_{(\rho)}(\zeta,\xi)\subseteq \left\{(\zeta,\xi):\sqrt{\zeta^2+\xi^2}< \frac{\rho}{100}  \right\},\,\psi_{(\rho)}(\zeta,\xi)\in S(\RR^2, dx),$$
and
$$\int_{\RR^2}\psi_{(\rho)}(\zeta,\xi) |\zeta|^{2\lambda}d\zeta d\xi=1,\ \|\psi_{(\rho)}\|_{\infty}\sim \frac{1}{\rho^{2\lambda+2}}.$$
Thus it is clear that
\begin{equation}\label{sk2}
\|D_\zeta\psi_{(\rho)}(\zeta,\xi)\|_{\infty}\lesssim \frac{1}{\rho^{2\lambda+3}},\  \|(D_\zeta)^2\psi_{(\rho)}(\zeta,\xi)\|_{\infty}\lesssim \frac{1}{\rho^{2\lambda+4}},\ \|D_\zeta(\partial_\xi)\psi_{(\rho)}(\zeta,\xi)\|_{\infty}\lesssim \frac{1}{\rho^{2\lambda+4}}.
\end{equation}
By Proposition\,\ref{sss1}(1), we could write $u(x, r)$ as following:
\begin{eqnarray}\label{sk1}
u(x, r)=\nonumber \sigma_\lambda\int_0^{+\infty}\int_\RR (\tau_{x}u)(-\zeta, r-\xi)\psi_{(r)}(\zeta,\xi) |\zeta|^{2\lambda}d\zeta d\xi
\\=\sigma_\lambda\int_0^{+\infty}\int_\RR u(s, t) \tau_{-s}\psi_{(r)}(x, r-t)|s|^{2\lambda}ds dt,
\end{eqnarray}
where $\sigma_\lambda^{-1}=\int_{-\pi}^\pi |\cos\theta|^{2\lambda}d\theta =2\sqrt{\pi}\Gamma(\lambda+\frac{1}{2})/\Gamma(\lambda+1)$. Thus we could deduce that $$(s, t)\in \left\{(s, t):\sqrt{(x-s)^2+(r-t)^2}< \frac{r}{10}\right\}\bigcup \left\{(s, t):\sqrt{(x+s)^2+(r-t)^2}< \frac{r}{10}\right\}.$$
We use $A_{\mu, \nu}$ to denote as the set: $$A_{\mu, \nu}= \left\{(s, t):\sqrt{(\mu-s)^2+(\nu-t)^2}< \frac{\nu}{10}\right\}\bigcup \left\{(s, t):\sqrt{(\mu+s)^2+(\nu-t)^2}< \frac{\nu}{10}\right\}.$$
Thus by Proposition\,\ref{sss1}(2)(3), Formula\,(\ref{sk4}), Formula\,(\ref{sk1}), Formula\,(\ref{sk2}) we could deduce the following inequality:
\begin{eqnarray}\label{sk9}
|D_xu(x, r)|&=&\left|\sigma_\lambda\int_0^{+\infty}\int_\RR u(s, t) \tau_{-s}(D\psi_{(r)})(x, r-t)|s|^{2\lambda}ds dt\right|
 \\ \nonumber &\lesssim& \sup_{(s, t)\in A_{x, r}}|u(s, t)| \|D_\zeta\psi_{(r)}(\zeta,\xi)\|_{\infty} r^{2\lambda+2}
 \\ \nonumber &\lesssim& r^{-(2\lambda+1)/p} \frac{1}{r^{2\lambda+3}}  r^{2\lambda+2}
 \\ \nonumber &\lesssim& r^{-(2\lambda+1)/p}r^{-1}.
\end{eqnarray}
In a similar way, we could obtain the following inequality
\begin{eqnarray}\label{sk10}
|(D_x)^2u(x, r)|&=&\left|\sigma_\lambda\int_0^{+\infty}\int_\RR u(s, t) \tau_{-s}(D^2\psi_{(r)})(x, r-t)|s|^{2\lambda}ds dt\right|
\\ \nonumber &\lesssim& r^{-(2\lambda+1)/p}r^{-2}.
\end{eqnarray}

Thus from Formula\,(\ref{sk8}), Formula\,(\ref{sk9}) and Formula\,(\ref{sk10}),  we could know that
the integral of $D_xu(x, r)$ and $(D_x)^2u(x, r)$ are meaningful. Thus $v(x, y)$, $D_x v(x, y)$ and $\partial_y v(x, y)$ are well defined functions. Thus it is not difficult to check that $v(x, y)$ and $u(x, y)$ satisfy the  $\lambda$-Cauchy-Riemann equations:
\begin{eqnarray*}
\left\{\begin{array}{ll}
                                    D_xu(x, y)-\partial_y v(x, y)=0,&  \\
                                    \partial_y u(x, y) +D_xv(x, y)=0.&
                                 \end{array}\right.
\end{eqnarray*}
Thus the function $F(z)=u(x, y)+iv(x, y)$ is a $\lambda$-harmonic function and $u(x, y)=ReF(z)$. By Formula\,(\ref{sk8}), it is clear that the following inequality holds:
\begin{eqnarray}\label{sk12}
|v(x, y)|&=&\left|-\int_{y}^{+\infty} D_x u(x, r)dr\right|
\\  \nonumber &=&\sigma_\lambda\left|\int_{y}^{+\infty}\int_0^{+\infty}\int_\RR u(s, t) \left(\tau_{-s}(D\psi_{(r)})\right)(x, r-t)|s|^{2\lambda}ds dtdr\right|
\\ \nonumber &\lesssim& \left|\int_{y}^{+\infty}\left(\sup_{(s, t)\in A_{x, r}} |ru(s, t)| \right) \frac{1}{r^2} dr\right|
\\ \nonumber &\lesssim& \left|\left(\sup_{r\geq y>0}\sup_{(s, t)\in A_{x, r}} |tu(s, t)|\right) \int_y^{+\infty} \frac{1}{r^2} dr\right|.
\end{eqnarray}
By Formula\,(\ref{sk4}), we could know that
$$\sup_{r\geq y>0}\sup_{(s, t)\in A_{x, r}} |tu(s, t)|<\infty.$$
Notice  that the balls $\{(s, t):(s, t)\in A_{x, r}\}$ are in the cone $\{(s, t):|s-x|<|t-\frac{y}{2}|, t>\frac{y}{2}\}$, that is: $\{(s, t):(s, t)\in A_{x, r}\}\subset \{(s, t):|s-x|<|t-\frac{y}{2}|, t>\frac{y}{2}\}$. Thus we could deduce that
\begin{eqnarray}\label{sk13}
\sup_{r\geq y>0}\sup_{(s, t)\in A_{x, r}} |tu(s, t)|\lesssim |yu_\nabla^*(x, \frac{y}{2})|+|yu_\nabla^*(-x, \frac{y}{2})|,
\end{eqnarray}
where $u_\nabla^*(x, \frac{y}{2})$ denotes $ \sup_{|s|<t}|u(x+s, \frac{y}{2}+t)|$.

Thus by Formula\,(\ref{sk12}) and  Formula\,(\ref{sk13}), we could obtain that:
\begin{eqnarray}\label{sk14}
|v(x, y)|&\lesssim&|yu_\nabla^*(x, \frac{y}{2})| \frac{1}{y}+|yu_\nabla^*(-x, \frac{y}{2})|\frac{1}{y}
\\ \nonumber &\lesssim& u_\nabla^*(x)+u_\nabla^*(-x).
\end{eqnarray}
Thus by Formula\,(\ref{sk14}), we could deduce the following inequality for any $y>0$:
$$\int_{-\infty}^{+\infty}|v(x, y)|^p|x|^{2\lambda}dx\lesssim\int_{-\infty}^{+\infty}|u_\nabla^*(x)|^p|x|^{2\lambda}dx\ \ \hbox{for}\ \frac{2\lambda}{2\lambda+1}<p\leq 1.$$
Then for $\frac{2\lambda}{2\lambda+1}<p\leq 1$, we could deduce that:
\begin{eqnarray}\label{sk15}
\|F\|_{H_{\lambda}^p(\RR_+^2)}\leq c\|u_{\nabla}^*\|_{L^p_{\lambda}}.
\end{eqnarray}

By Formula\,(\ref{sk15}) and Proposition\,\ref{ss}, we deduce the following inequality for $\frac{2\lambda}{2\lambda+1}<p\leq 1$:
$$\|F\|_{H_{\lambda}^p(\RR_+^2)}\sim\|u_{\nabla}^*\|_{L^p_{\lambda}}.$$
This proves the Theorem.
\end{proof}

\begin{proposition}\label{p4}
  $H_{\lambda}^p(\RR_+^2)\bigcap H_{\lambda}^2(\RR_+^2)\bigcap H_{\lambda}^1(\RR_+^2)$ is dense in $H_{\lambda}^p(\RR_+^2)$, for $ \frac{2\lambda}{2\lambda+1}< p\leq 1$.
\end{proposition}
\begin{proof}
From\,\cite{ZhongKai Li 3}, we could know that for $F(x, y)\in H_{\lambda}^p(\RR_+^2)$  and  $s>0$
\begin{eqnarray*}
\left(\int_{\RR}|F(x,y+s)|^{2}|x|^{2\lambda}dx \right)^{\frac{1}{2}} \leq cs^{(1/2-1/p)(1+2\lambda)}\|F\|_{H_{\lambda}^p(\RR_+^2)},
\end{eqnarray*}
and
$$\left(\int_{\RR}|F(x,y+s)||x|^{2\lambda}dx \right)^{\frac{1}{1}} \leq cs^{-(1/p-1/1)(1+2\lambda)}\|F\|_{H_{\lambda}^p(\RR_+^2)},$$
hold for $ \frac{2\lambda}{2\lambda+1}< p\leq 1$.
Thus we could deduce that $F(x,y+s)\in H_{\lambda}^2(\RR_+^2)\bigcap H_{\lambda}^1(\RR_+^2)$.  By Proposition\,\ref{k1}(ii), we could see that  $\lim_{s\rightarrow0+}\|F(\cdot,y+s)-F(\cdot,y)\|_{L^p_{\lambda}}=0$.
Then we could see that $H_{\lambda}^p(\RR_+^2)\bigcap H_{\lambda}^2(\RR_+^2)\bigcap H_{\lambda}^1(\RR_+^2)$ is dense in $H_{\lambda}^p(\RR_+^2)$. This proves the proposition.
\end{proof}
\begin{definition}\label{o1}
By Proposition\,\ref{k1} and Theorem\,\ref{u2},
 $\widetilde{H}_{\lambda}^p(\RR)$ ($\frac{2\lambda}{2\lambda+1}<p<\infty$) could be defined as
\begin{eqnarray*}
\widetilde{H}_{\lambda}^p(\RR)&\triangleq&\left\{g(x):g(x)=\lim_{y\rightarrow0} ReF(t,y), F\in H_{\lambda}^p(\RR_+^2)\bigcap H_{\lambda}^1(\RR_+^2)\bigcap H_{\lambda}^2(\RR_+^2)\right.\\& & \,(t, y)\, approaches\, the\, point\, (x, 0)\, nontangentially \bigg\}.
\end{eqnarray*}
with the norm:
$$\|g\|^p_{H_{\lambda}^p(\RR)}=\|P_{\nabla}^*g\|^p_{L_{\lambda}^p(\RR)}.$$
Thus
\begin{eqnarray*}
\widetilde{H}_{\lambda}^p(\RR)&\triangleq&\left\{g(x)\in L_{\lambda}^1(\RR)\bigcap L_{\lambda}^2(\RR): \|P_{\nabla}^*g\|^p_{L_{\lambda}^p(\RR)}<\infty\right\}.
\end{eqnarray*}
Thus $\widetilde{H}_{\lambda}^p(\RR)$ is a linear space equipped with the norm: $\|\cdot\|^p_{H_{\lambda}^p(\RR)}$, which is not complete. The completion of $\widetilde{H}_{\lambda}^p(\RR)$ with the norm $\|\cdot\|^p_{H_{\lambda}^p(\RR)}$ is denoted as  $H_{\lambda}^p(\RR)$.
(We will also define  $H_{\lambda}^p(\RR)$  as Theorem\,\ref{tan12}.)
\end{definition}
Thus we could have the following conclusions:
\begin{proposition}\label{s5}
 $H_{\lambda}^p(\RR)\bigcap H_{\lambda}^2(\RR)\bigcap H_{\lambda}^1(\RR)$ is dense in $H_{\lambda}^p(\RR)$ for $\frac{2\lambda}{2\lambda+1}<p<\infty$.   $H_{\lambda}^p(\RR)=L^{p}_{\lambda}(\RR)$, for $1<p<\infty$.    $H_{\lambda}^1(\RR)\subset L^{1}_{\lambda}(\RR).$
\end{proposition}

\subsection{  Homogeneous type Hardy Spaces on Dunkl setting  }\label{homo}
In Definition\,\ref{o1}, we have introduced the real-variable Hardy spaces: $H_{\lambda}^p(\RR)$ which  is  associated with the Complex-Hardy spaces $H_{\lambda}^p(\RR_+^2)$. In this section, we  will prove that the $H_{\lambda}^p(\RR)$ is  Homogeneous Hardy spaces.

We use $d\mu_{\lambda}(x)$ $\mu_{\lambda}(x, y)$ and $d_{\lambda}(x, y)$ to denote as:
$d_{\lambda}(x, y)=(2\lambda+1)\left|\int_y^x|t|^{2\lambda}dt\right|$, $\mu_{\lambda}(x, y)=(2\lambda+1)\int_y^x|t|^{2\lambda}dt$, $d\mu_{\lambda}(x)=(2\lambda+1)|x|^{2\lambda}dx$. And the ball $B(x, r)$ is denoted as: $B(x, r)=B_{\lambda}(x, r)=\{y: d_\lambda(x, y)<r\}$.

We will introduce a new kernel $K(r, x, t)$ as following:
\begin{eqnarray}\label{kernel} K(r, x, t) =\left\{ \begin{array}{cc}
                             r(\tau_xP_{r|x|^{-2\lambda}})(-t)  & \hbox{for} \ \  r<|x|^{2\lambda+1}, \\
                             r(\tau_xP_{r^{1/(2\lambda+1)}})(-t)  & \hbox{for}\ \  r\ge
                             |x|^{2\lambda+1}.
                           \end{array}\right.
\end{eqnarray}
Thus $K(r, x, t)=r(\tau_xP_{y})(-t)$, where $y$ has the representation
\begin{equation}\label{y(x,r)} y  =\left\{ \begin{array}{cc}
                             r|x|^{-2\lambda} & \hbox{for} \ \  r<|x|^{2\lambda+1}, \\
                             r^{1/(2\lambda+1)} & \hbox{for}\ \  r\ge
                             |x|^{2\lambda+1}.
                           \end{array}\right.
\end{equation}
Then for any $ f(x) \in L_{\lambda}^2(\RR)\bigcap L_{\lambda}^1 (\RR)\bigcap H_{\lambda}^p(\RR)$,  $ \frac{2\lambda}{2\lambda+1}< p\leq 1$, the following holds:
\begin{equation}\sup\limits_{r>0}
\int_{\mathbb
R}K(r, x, t)f(t)\frac{|t|^{2\lambda}dt}{r}=\sup\limits_{y>0}\left(P_{y}\ast_{\lambda}f\right)(x).\end{equation}
From\cite{ZhongKai Li 3}, the following inequality holds:
\begin{eqnarray}\label{Poisson-ker-3}
(\tau_xP_y)(-t)\sim\frac{y[y^2+(|x|+|t|)^2]^{-\lambda}}{y^2+(x-t)^2}\ln\left(\frac{y^2+(x-t)^2}{y^2+(x+t)^2}+2\right).
\end{eqnarray}
Then  we  will prove the following Theorem\,\ref{kkk}.
\begin{theorem}\label{kkk}
  $K(r, x, t)=r(\tau_xP_{y})(-t)$ is a kernel  satisfying the following:

 {\rm(i)} \  $K(r, x, x) \gtrsim1, \  \hbox{for}\ \ r>0, x\in\mathbb R$;

{\rm(ii)}  \ $0\le K(r, x, t)\lesssim
\Big(1+\frac{d_{\lambda}(x,t)}{r}\Big)^{-1-\gamma_\lambda}, \ \hbox{for} \ \ r>0, x,
t\in\mathbb R$;

{\rm(iii)} \ For $r>0, x, t,z\in\mathbb R$, if $\frac{d_\lambda(t,
z)}{r}\leq C\min\{1+\frac{d_\lambda(x,t)}{r}, 1+\frac{d_\lambda(x,z)}{r}\} $
$$| K(r, x, t)-K(r, x, z)|\lesssim \Big(\frac{d_{\lambda}(t,z)}{r}\Big)^{\gamma_\lambda}\Big(1+\frac{d_{\lambda}(x,t)}{r}\Big)^{-1-2\gamma_\lambda};$$
{\rm(iv)} \ $$K(r, x, y)=K(r, y, x),$$
where $\gamma_\lambda=\frac{1}{2(2\lambda+1)}$.
\end{theorem}

\begin{proof}
$K(r, x, y)=K(r, y, x)$ can be deduced from the fact that $(\tau_xP_y)(-t)=(\tau_tP_y)(-x).$
Notice that for any $ s\neq0$, we have
\begin{eqnarray*}
K(|s|^{2\lambda+1}r, sx, st)=K(r, x, t),\ \
d_{\lambda}(sx, st)=|s|^{-2\lambda-1}d_{\lambda}(x, t).
\end{eqnarray*}
Thus we need to only prove the theorem  for the case when $x=0$ and $x=1$. First, we will prove $K(r, x, x)\geq c>0$ for some constant c.\\
Case 1 \ \ \ \ \ \ \   x=0. By Formula\,(\ref{y(x,r)}), we could deduce that $y=r^{\frac{1}{2\lambda+1}}$. Thus from Formula\,(\ref{Poisson-ker-3}), we could deduce that\\
$$K(r, 0, 0)=\frac{r\ast r^{\frac{1}{2\lambda+1}}}{(r^{\frac{2}{2\lambda+1}})^{\lambda+1}}\geq 1.$$
Case 2 \ \ \ \ \ \ \ \  $x\neq0$, we need only to consider the case when x=1. \\
When $r<1$, by Formula\,(\ref{kernel}) and Formula\,(\ref{y(x,r)}), we have $y=r<1$. Thus from Formula\,(\ref{D-Poisson-ker-11}), we could deduce that:
\begin{eqnarray*}
K(r, 1, 1)&=&\frac{\lambda\Gamma(\lambda+1/2)}{2^{-\lambda-1/2}\pi}\int_0^\pi\frac{ry(1+\cos\theta)
}{\big(y^2+2-2\cos\theta\big)^{\lambda+1}}\sin^{2\lambda-1}\theta
d\theta.\\
&\geq& c\int_0^{y/4}\frac{ry(1+\cos\theta)
}{\big(y^2+2-2\cos\theta\big)^{\lambda+1}}\sin^{2\lambda-1}\theta
d\theta\\
&\geq& c.
\end{eqnarray*}
When $r\geq1$, from Formula\,(\ref{kernel}) and Formula\,(\ref{y(x,r)}), we could deduce that $y=r^{\frac{1}{2\lambda+1}}\geq1$. Thus we could obtain the following from Formula\,(\ref{Poisson-ker-3}):
\begin{eqnarray*}
K(r, 1, 1)&\geq&\frac{r^{\frac{1}{2\lambda+1}}r}{\big(r^{\frac{2}{2\lambda+1}}+2\big)^{\lambda+1}}\\
&\geq& c.
\end{eqnarray*}
Second, we will prove that  $0\le K(r, x, t)\le
A\Big(1+\frac{d_{\lambda}(x,t)}{r}\Big)^{-1-\gamma_\lambda}, \  \hbox{for}\ \ r>0, x,
t\in\mathbb R.$\\
Case 1 \ \ \ \ \ \ \   When x=0, by Formula\,(\ref{y(x,r)}), we could deduce that $y=r^{\frac{1}{2\lambda+1}}$. Thus from Formula\,(\ref{Poisson-ker-3}) the following holds:
$$K(r, 0, t)\sim C\Big(1+\frac{t^2}{r^{2/(2\lambda+1)}}\Big)^{-\lambda-1} \sim
A\Big(1+\frac{|t|^{2\lambda+1}}{(2\lambda+1)r}\Big)^{-\frac{2(\lambda+1)}{2\lambda+1}}=A\Big(1+\frac{d_\lambda(0,t)}{r}\Big)^{-\frac{2(\lambda+1)}{2\lambda+1}}.$$
Case 2 \ \ \ \ \ \ \ \  When $x\neq0$, we need only to consider the case  for x=1.  Notice that $y=r^{\frac{1}{2\lambda+1}}\geq1$ for $r\geq1$,  and $y=r$, for $r<1$. By Formula\,(\ref{Poisson-ker-3}), we could have\\
\begin{equation}\label{ss1} when\ r\geq1\ \   \ \  K(r, 1, t) \sim\left\{ \begin{array}{cc}
                             \displaystyle{\frac{r^{\frac{2\lambda+2}{2\lambda+1}}}{\left(r^{\frac{2}{2\lambda+1}}+t^2+1\right)^{\lambda+1}}}\ln\left(\frac{r^2+t^2+1}{r^2+(t+1)^2}+1\right) & \hbox{for} \ \  t<0, \\ \\
                            \displaystyle{ \frac{r^{\frac{2\lambda+2}{2\lambda+1}}}{\left(r^{\frac{2}{2\lambda+1}}+t^2+1\right)^{\lambda}\left((r^{\frac{2}{2\lambda+1}}+\left(1-t\right)^2\right)}} & \hbox{for} \ \  t\geq0.
                           \end{array}\right.
\end{equation}
\begin{equation}\label{ss2} when\ r<1\ \  \ \  K(r, 1, t) \sim\left\{ \begin{array}{cc}
                             \displaystyle{\frac{r^2}{\left(r^2+t^2+1\right)^{\lambda+1}}} & \hbox{for} \ \  t<0, \\ \\
                             \displaystyle{\frac{r^2}{\left(r^2+t^2+1\right)^{\lambda}\left(r^2+\left(1-t\right)^2\right)}} & \hbox{for} \ \  t\geq0.
                           \end{array}\right.
\end{equation}

If $r<1$, $1/2\leq t\leq 3/2$, we have $d_\lambda(1, t)\sim |1-t|$. Then
$$K(r, 1, t)\lesssim\left(1+\left(\frac{|1-t|}{r}\right)\right)^{-2}\lesssim \left(1+\left(\frac{d_{\lambda}(1, t)}{r}\right)\right)^{-\frac{2(\lambda+1)}{2\lambda+1}}.$$

If $r\geq1$, $1/2\leq t\leq 3/2$, we have $d_\lambda(1, t)\sim |1-t|$. Then
$$K(r, 1, t)\lesssim r^{\frac{2\lambda+2}{2\lambda+1}}\left(r+|1-t|\right)^{-\frac{2\lambda+2}{2\lambda+1}}\lesssim \left(1+\left(\frac{d_{\lambda}(1, t)}{r}\right)\right)^{-\frac{2(\lambda+1)}{2\lambda+1}}.$$

If $r<1$, $t\geq 3/2$, we have $d_\lambda(1, t)\sim |1-t|^{2\lambda+1}$. Then
$$K(r, 1, t)\lesssim r^{2}\left(|1-t|\right)^{-2(\lambda+1)}\lesssim \left(1+\left(\frac{d_{\lambda}(1, t)}{r}\right)\right)^{-\frac{2(\lambda+1)}{2\lambda+1}}.$$

If $r\geq1$, $t\geq 3/2$, we have $d_\lambda(1, t)\sim |1-t|^{2\lambda+1}$. Then
$$K(r, 1, t)\lesssim r^{\frac{2\lambda+2}{2\lambda+1}}\left(r^{\frac{2}{2\lambda+1}}+|1-t|^2\right)^{-(\lambda+1)}\lesssim \left(1+\left(\frac{d_{\lambda}(1, t)}{r}\right)\right)^{-\frac{2(\lambda+1)}{2\lambda+1}}.$$

If $r<1$, $-2 \leq t\leq 1/2$, we have $d_\lambda(1, t)\sim 1$. Then

$$K(r, 1, t)\lesssim r^{2}\ln(r^{-1}+1)\lesssim \left(1+\left(\frac{d_{\lambda}(1, t)}{r}\right)\right)^{-\frac{2(\lambda+1)}{2\lambda+1}}.$$

If $r\geq1$, $-2 \leq t\leq 1/2$, we have $d_\lambda(1, t)\sim 1$. Then

$$K(r, 1, t)\lesssim C\lesssim \left(1+\left(\frac{d_{\lambda}(1, t)}{r}\right)\right)^{-\frac{2(\lambda+1)}{2\lambda+1}}.$$

If $r<1$, $t\leq -2$, we have $d_\lambda(1, t)\sim t^{2\lambda+1}$. Then
$$K(r, 1, t)\lesssim C\frac{r^2}{|t|^{2(\lambda+1)}}\lesssim \left(1+\left(\frac{d_{\lambda}(1, t)}{r}\right)\right)^{-\frac{2(\lambda+1)}{2\lambda+1}}.$$

If $r\geq1$, $t\leq -2$, we have $d_\lambda(1, t)\sim t^{2\lambda+1}$. Then

$$K(r, 1, t)\lesssim C\frac{r^{\frac{2\lambda+2}{2\lambda+1}}}{(r^{\frac{2}{2\lambda+1}}+t^2)^{\lambda+1}} C\lesssim \left(1+\left(\frac{d_{\lambda}(1, t)}{r}\right)\right)^{-\frac{2(\lambda+1)}{2\lambda+1}}.$$

Thus we have established
\begin{equation}\label{dp3oo}
0\le K(r, x, t)\lesssim
\Big(1+\frac{d_{\lambda}(x,t)}{r}\Big)^{-\frac{2(\lambda+1)}{2\lambda+1}}, \  \hbox{for} \ \ r>0, x,
t\in\mathbb R.
\end{equation}
From the above Formula\,(\ref{dp3oo}), we could deduce that
\begin{equation*}
0\le K(r, x, t)\lesssim
\Big(1+\frac{d_{\lambda}(x,t)}{r}\Big)^{-1-\gamma_\lambda}, \  \hbox{for} \ \ r>0, x,
t\in\mathbb R.
\end{equation*}

At last, if $\frac{d_\lambda(t,
z)}{r}\leq C\min\{1+\frac{d_\lambda(x,t)}{r}, 1+\frac{d_\lambda(x,z)}{r}\} $, we will prove the following inequality
$$| K(r, x, t)-K(r, x, z)|\lesssim\Big(\frac{d_{\lambda}(t,z)}{r}\Big)^{\gamma_\lambda}\Big(1+\frac{d_{\lambda}(x,t)}{r}\Big)^{-1-2\gamma_\lambda}$$
for $r>0, x, t,z\in\mathbb R$.
If $\frac{d_{\lambda}(t, z)}{r}\lesssim1+\frac{d_{\lambda}(x,t)}{r},$ then we could deduce  the following inequality:
\begin{equation*}
\frac{d_{\lambda}(x, z)}{r}\lesssim\left(\frac{d_{\lambda}(x, t)}{r}+\frac{d_{\lambda}(t, z)}{r}\right)\\
\lesssim\left(\frac{d_{\lambda}(x, t)}{r}+1+\frac{d_{\lambda}(x, t)}{r}\right)\\
\lesssim1+\frac{d_{\lambda}(x, t)}{r}.
\end{equation*}
Then
\begin{equation*}
1+\frac{d_{\lambda}(x, z)}{r}\lesssim 1+\frac{d_{\lambda}(x, t)}{r}.
\end{equation*}
Thus we could deduce:
\begin{equation}\label{equ1}
1+\frac{d_{\lambda}(x, z)}{r}\sim 1+\frac{d_{\lambda}(x, t)}{r}.
\end{equation}

For $u\in\RR $ satisfying $(u-t)(u-z)\leq0 $, we could obtain
$$ \frac{d_{\lambda}(u, t)}{r}\lesssim\frac{d_{\lambda}(t, z)}{r}\lesssim C\min\{1+\frac{d_\lambda(x,t)}{r}, 1+\frac{d_\lambda(x,z)}{r}\} .$$
Thus :
\begin{equation}\label{equ111}
1+\frac{d_{\lambda}(x, u)}{r}\sim 1+\frac{d_{\lambda}(x, t)}{r},  \ \ \hbox{when}\, (u-t)(u-z)\leq0.
\end{equation}

It is enough to prove that if $\frac{d_\lambda(t,
z)}{r}\leq C\min\{1+\frac{d_\lambda(x,t)}{r}, 1+\frac{d_\lambda(x,z)}{r}\} $, then
\begin{eqnarray}\label{estimate kernel 2}
\Big(1+\frac{d_{\lambda}(x,t)}{r}\Big)^{1+2\gamma_\lambda}|K(r, x, t)-K(r, x, z)|\lesssim \Big(\frac{d_{\lambda}(t,z)}{r}\Big)^{\gamma_\lambda}.
\end{eqnarray}

Let $t, z$ to be fixed first. We could see that
\begin{eqnarray}\label{dp8}
|t-z|\lesssim_\lambda \left(d_\lambda(t, z)\right)^{\frac{1}{2\lambda+1}}.
\end{eqnarray}

$\mathbf{Case 1}$ \ \ \  When  x=0 ($y=r^{\frac{1}{2\lambda+1}}$), we  suppose that $z>0$ first.
By Formula\,(\ref{equ111}), we could obtain   the following inequality for $(u-t)(u-z)\leq0$:
$$1+\frac{d_{\lambda}(0, u)}{r}\sim 1+\frac{d_{\lambda}(0, z)}{r}\sim 1+\frac{d_{\lambda}(0, t)}{r} \sim 1+\frac{u^{2\lambda+1}}{r}.$$
By the Mean value theorems for definite integrals, we could have:\\
\begin{eqnarray*}
& &\Big(1+\frac{d_{\lambda}(0,z)}{r}\Big)^{\frac{2\lambda+3}{2\lambda+1}}| K(r, 0, t)-K(r, 0, z)|
\\&=&c_{\lambda}\Big(1+\frac{d_{\lambda}(0,z)}{r}\Big)^{\frac{2\lambda+3}{2\lambda+1}}\int_0^{\pi} r\left(\frac{y}{\big(y^2+t^2\big)^{\lambda+1}}-\frac{y}{\big(y^2+z^2\big)^{\lambda+1}} \right) \sin^{2\lambda-1}\theta d\theta\nonumber\\
&\lesssim& \left|\Big(1+\frac{u^{2\lambda+1}}{r}\Big)^{\frac{2\lambda+3}{2\lambda+1}}\frac{ur^{\frac{2\lambda+2}{2\lambda+1}}}{\big(r^{\frac{2}{2\lambda+1}}+u^2\big)^{\lambda+2}}\right| |t-z|. \nonumber\\
\end{eqnarray*}
\begin{equation*} \left\{ \begin{array}{cc}
                      \displaystyle{\left|\Big(1+\frac{u^{2\lambda+1}}{r}\Big)^{\frac{2\lambda+3}{2\lambda+1}}\frac{ur^{\frac{2\lambda+2}{2\lambda+1}}}{\big(r^{\frac{2}{2\lambda+1}}+u^2\big)^{\lambda+2}}\right|}\leq\frac{r^{\frac{1}{2\lambda+1}}r^{\frac{2\lambda+2}{2\lambda+1}}}{\big(r^{\frac{2}{2\lambda+1}}\big)^{\lambda+2}}\leq \frac{1}{r^{\frac{1}{2\lambda+1}}}& \hbox{for} \ \  |u|<r^{\frac{1}{2\lambda+1}}, \\ \\ \\
                            \displaystyle{ \left|\Big(1+\frac{u^{2\lambda+1}}{r}\Big)^{\frac{2\lambda+3}{2\lambda+1}}\frac{ur^{\frac{2\lambda+2}{2\lambda+1}}}{\big(r^{\frac{2}{2\lambda+1}}+u^2\big)^{\lambda+2}}\right|}\leq \left|\frac{1}{r^{\frac{1}{2\lambda+1}}}\frac{u^{2\lambda+4}}{\big(r^{\frac{2}{2\lambda+1}}+u^2\big)^{\lambda+2}}\right|\leq \frac{1}{r^{\frac{1}{2\lambda+1}}} & \hbox{for} \ \  |u|\geq r^{\frac{1}{2\lambda+1}}.
                           \end{array}\right.
\end{equation*}
Thus  when $\frac{d_\lambda(t,
z)}{r}\leq C\min\{1+\frac{d_\lambda(0,t)}{r}, 1+\frac{d_\lambda(0,z)}{r}\} $, the following inequality holds:
\begin{equation}\label{dpo}
\Big(1+\frac{d_{\lambda}(0,t)}{r}\Big)^{\frac{2\lambda+3}{2\lambda+1}}| K(r, 0, t)-K(r, 0, z)|\lesssim \frac{|t-z|}{r^{\frac{1}{2\lambda+1}}}\lesssim \Big(\frac{d_{\lambda}(t,z)}{r}\Big)^{\frac{1}{2\lambda+1}}.
\end{equation}
$\mathbf{Case  2}$ \ \ \ \ When $x\neq0$, it will be enough to prove Formula\,(\ref{estimate kernel 2}) for the case when $x=1$. From Formula\,(\ref{D-Poisson-ker-11}), we could write $K(r, 1, t)=r(\tau_1P_y)(-t)$  as following:
\begin{eqnarray}\label{D-Poisson-2}
r(\tau_1P_y)(-t)=
\frac{\lambda\Gamma(\lambda+1/2)}{2^{-\lambda-1/2}\pi}\int_{-1}^1\frac{ry
}{\big(y^2+1+t^2-2ts\big)^{\lambda+1}}(1+s)(1-s^2)^{\lambda-1}ds.
\end{eqnarray}
By Formula\,(\ref{D-Poisson-2}) and Mean value theorems for definite integrals, we could obtain:
\begin{eqnarray}\label{dp5}
| K(r, 1, t)-K(r, 1, z)|\nonumber&\sim&\nonumber \left|\int_{-1}^{1}\left(\frac{ry(1-s^2)^{\lambda-1}(1+s)
}{\big(y^2+1+t^2-2ts\big)^{\lambda+1}}-\frac{ry(1-s^2)^{\lambda-1}(1+s)
}{\big(y^2+1+z^2-2zs\big)^{\lambda+1}}\right)
ds\right|\\
&\lesssim&\left|\int_{-1}^{1}\frac{ry|u-s|
}{\big(y^2+1+u^2-2us\big)^{\lambda+2}}(1-s^2)^{\lambda-1}(1+s)
ds\right||t-z|,
\end{eqnarray}
where $u$ satisfies $(u-t)(u-z)\leq0$. Then we will discuss the Formula\,(\ref{dp5}) for three conditions: $Condition A\ \ u\geq 0$,
$Condition B\ \ u\leq -3/2\,or\, -1/2\leq u\leq0$, and  $Condition C\ \ -3/2 \leq u\leq -1/2$.

$\mathbf{Condition A\ \ u\geq 0}$.

When $u\geq0$,  for $\frac{d_\lambda(t,
z)}{r}\leq C\min\{1+\frac{d_\lambda(1, t)}{r}, 1+\frac{d_\lambda(1, z)}{r}\} $,  we will prove the following inequality:
$$\Big(1+\frac{d_{\lambda}(1, t)}{r}\Big)^{\frac{2\lambda+3}{2\lambda+1}}| K(r, 1,  t)-K(r, 1, z)|\lesssim\Big(\frac{d_{\lambda}(t, z)}{r}\Big)^{\frac{1}{2\lambda+1}}.$$
By Formula\,(\ref{D-Poisson-2}), Formula\,(\ref{equ111}) and Mean value theorems for definite integrals, we could obtain:
\begin{eqnarray}\label{dp1}
& &\Big(1+\frac{d_{\lambda}(1, t)}{r}\Big)^{\frac{2\lambda+3}{2\lambda+1}}| K(r, 1, t)-K(r, 1, z)|\nonumber\\
&\lesssim&\left|\Big(1+\frac{d_{\lambda}(1, |u|)}{r}\Big)^{\frac{2\lambda+3}{2\lambda+1}}\int_{-1}^{1}\frac{ry|u-s|
}{\big(y^2+1+u^2-2us\big)^{\lambda+2}}(1-s^2)^{\lambda-1}(1+s)
ds\right|\left|\left(t-z\right)\right|,
\end{eqnarray}
where $u$ satisfies $(u-t)(u-z)\leq0$.

Notice that the following Formulas\,(\ref{dp2}),\,(\ref{dp3}),\,(\ref{dp4}) hold for $-1\leq s\leq 1$ and $u\geq0$:
\begin{eqnarray}\label{dp2}
\left|\frac{u-1}{(y^2+1+u^2-2us)}\right|<\left|\frac{u-1}{(y^2+1+u^2-2u)}\right|.
\end{eqnarray}
For $0\leq s\leq 1$, we have:
\begin{eqnarray}\label{dp3}
\left|\frac{1-s}{(y^2+1+u^2-2us)}\right|\lesssim\frac{1}{(y^2+1+u^2)}.
\end{eqnarray}
For $-1\leq s\leq 0$, we have:
\begin{eqnarray}\label{dp4}
\left|\frac{1}{(y^2+1+u^2-2us)}\right|\lesssim\frac{1}{(y^2+1+u^2)}.
\end{eqnarray}

From Formula\,(\ref{dp1}) Formula\,(\ref{dp2}) Formula\,(\ref{dp3}) Formula\,(\ref{dp4}) and Formula\,(\ref{Poisson-ker-3}), we could obtain
the following Formula\,(\ref{dp9}):
\begin{eqnarray}\label{dp9}
& &\left|\int_{-1}^{1}\frac{ry|u-s|
}{\big(y^2+1+u^2-2us\big)^{\lambda+2}}(1-s^2)^{\lambda-1}(1+s)
ds\left(t-z\right)\right| \\ \nonumber
&\leq&\left(\left| \int_{-1}^{1}\frac{ry|u-1|(1-s^2)^{\lambda-1}(1+s)
}{\big(y^2+1+u^2-2us\big)^{\lambda+2}}
ds\right| +\left| \int_{-1}^{1}\frac{ry(1-s^2)^{\lambda}
}{\big(y^2+1+u^2-2us\big)^{\lambda+2}}
ds\right|\right)\left|\left(t-z\right)\right| \\ \nonumber
&\leq&C\left|\frac{|u-1|
}{\big(y^2+1+u^2-2|u|\big)}\left|r(\tau_1P_y)(-u)\right|\left(t-z\right)\right|+C\left|\frac{1}{\big(y^2+1+u^2\big)}\left|r(\tau_1P_y)(-u)\right|\left(t-z\right)\right|\\ \nonumber &\leq& C\left|\left(t-z\right)yr\frac{(1-|u|)^2+y^2+(1+u^2+y^2)|1-|u||}{((1-|u|)^2+y^2)^2(1+u^2+y^2)^{\lambda+1}}\right|.
\end{eqnarray}

i: If $r<1$,  then $y=r$.

$\mathbf{Condition A_1}$. $\hbox{For}\ r<1 ,\  |1-|u||\geq\frac{1}{10C}$ (for some constant $C>1$), we could deduce that   $d_{\lambda}(1, |u|)\gtrsim \frac{1}{10C}$. Thus the following could be obtained by Formula\,(\ref{dp8}):
\begin{eqnarray*}
& &\left|\Big(1+\frac{d_{\lambda}(1,|u|)}{r}\Big)^{\frac{2\lambda+3}{2\lambda+1}}\left(|t-z|\right)yr\frac{(1-|u|)^2+y^2+(1+u^2+y^2)|1-|u||}{((1-|u|)^2+y^2)^2(1+u^2+y^2)^{\lambda+1}}\right|
\\&=&
\left|\Big(1+\frac{d_{\lambda}(1,|u|)}{r}\Big)^{\frac{2\lambda+3}{2\lambda+1}}\left(|t-z|\right)r^2\frac{(1-|u|)^2+r^2+(1+u^2+r^2)|1-|u||}{((1-|u|)^2+r^2)^2(1+u^2+r^2)^{\lambda+1}}\right|
\\&\lesssim&\frac{d_\lambda(1, |u|)^{\frac{2\lambda+3}{2\lambda+1}}}{r^{\frac{2\lambda+3}{2\lambda+1}}}r^2\left|\left(|t-z|\right)\right|\frac{(1+u^2+r^2)|1-|u||}{((1-|u|)^2+r^2)^2(1+u^2+r^2)^{\lambda+1}}
\lesssim  \Big(\frac{d_{\lambda}(t,z)}{r}\Big)^{\frac{1}{2\lambda+1}}.
\end{eqnarray*}

$\mathbf{Condition A_2}$. $\hbox{For}\ r<1 ,\   \frac{r}{20C}\leq|1-|u||\leq\frac{1}{10C}$, it is clear that $d_{\lambda}(1, |u|)\sim |1-|u||$, $d_{\lambda}(t, z)\lesssim r+ d_{\lambda}(1, |u|)\leq C_1 d_{\lambda}(1, |u|) \leq \frac{C_1}{10C}$. Let $C$ to be a constant satisfying $\frac{C_1}{C}\leq1$, thus we could deduce that $d_{\lambda}(t, z)\leq \frac{1}{10}$. Then we could obtain that  $d_{\lambda}(t, z)\sim |t-z|$. Thus
\begin{eqnarray*}
& &\left|\Big(1+\frac{d_{\lambda}(1,|u|)}{r}\Big)^{\frac{2\lambda+3}{2\lambda+1}}\left(|t-z|\right)yr\frac{(1-|u|)^2+y^2+(1+u^2+y^2)|1-|u||}{((1-|u|)^2+y^2)^2(1+u^2+y^2)^{\lambda+1}}\right|
\\&=&
\left|\Big(1+\frac{d_{\lambda}(1,|u|)}{r}\Big)^{\frac{2\lambda+3}{2\lambda+1}}\left(|t-z|\right)r^2\frac{(1-|u|)^2+r^2+(1+u^2+r^2)|1-|u||}{((1-|u|)^2+r^2)^2(1+u^2+r^2)^{\lambda+1}}\right|
\\&\lesssim& \frac{1}{r^{\frac{1}{2\lambda+1}}}\frac{d_{\lambda}(t, z)}{d_{\lambda}(1, |u|)^{\frac{2\lambda}{2\lambda+1}}}
\lesssim \Big(\frac{d_{\lambda}(t,z)}{r}\Big)^{\frac{1}{2\lambda+1}}.
\end{eqnarray*}

$\mathbf{Condition A_3}$. $\hbox{For}\ r<1 ,\   |1-|u||\leq \frac{r}{20C}$, we\  have\ $d_{\lambda}(1, |u|)\sim |1-|u||,  |t-z|\sim d_{\lambda}(t, z)\lesssim r+d_{\lambda}(1,|u|)\lesssim r$, then
\begin{eqnarray*}
& &\left|\Big(1+\frac{d_{\lambda}(1,|u|)}{r}\Big)^{\frac{2\lambda+3}{2\lambda+1}}\left(|t-z|\right)yr\frac{(1-|u|)^2+y^2+(1+u^2+y^2)|1-|u||}{((1-|u|)^2+y^2)^2(1+u^2+y^2)^{\lambda+1}}\right|
\\&=&
\left|\Big(1+\frac{d_{\lambda}(1,|u|)}{r}\Big)^{\frac{2\lambda+3}{2\lambda+1}}\left(|t-z|\right)r^2\frac{(1-|u|)^2+r^2+(1+u^2+r^2)|1-|u||}{((1-|u|)^2+r^2)^2(1+u^2+r^2)^{\lambda+1}}\right|
\\&\lesssim&\frac{\left|t-z\right|}{r}
\lesssim \Big(\frac{d_{\lambda}(t, z)}{r}\Big)^{\frac{1}{2\lambda+1}}.
\end{eqnarray*}

$\mathbf{Condition A_4}$. ii: If $r\geq1$, then $y=r^{\frac{1}{2\lambda+1}}$. Thus
\begin{eqnarray*}
& &\Big(1+\frac{d_{\lambda}(1,|u|)}{r}\Big)^{\frac{2\lambda+3}{2\lambda+1}}\left|\left(|t-z|\right)\right|yr\frac{(1-|u|)^2+y^2+(1+u^2+y^2)|1-|u||}{((1-|u|)^2+y^2)^2(1+u^2+y^2)^{\lambda+1}}
\\&=&
\Big(1+\frac{d_{\lambda}(1,|u|)}{r}\Big)^{\frac{2\lambda+3}{2\lambda+1}}\left|\left(|t-z|\right)\right|r^{\frac{2\lambda+2}{2\lambda+1}}\frac{(1-|u|)^2+r^{\frac{2}{2\lambda+1}}+(1+u^2+r^{\frac{2}{2\lambda+1}})|1-|u||}{((1-|u|)^2+r^{\frac{2}{2\lambda+1}})^2(1+u^2+r^{\frac{2}{2\lambda+1}})^{\lambda+1}}
\\&\lesssim&
\left\{ \begin{array}{cc}
\frac{|u|^{2\lambda+3}}{r^{\frac{2\lambda+3}{2\lambda+1}}}\left|\left(t-z\right)\right|r^{\frac{2\lambda+2}{2\lambda+1}} \frac{(1-|u|)(1+u^2)}{(1-|u|)(1+u^2)r^{\frac{2\lambda+3}{2\lambda+1}}}\lesssim\Big(\frac{d_{\lambda}(t,z)}{r}\Big)^{\frac{1}{2\lambda+1}}, \ \hbox{for}\  |1-|u||\geq 2r^{\frac{1}{2\lambda+1}}, \ d_{\lambda}(1, |u|)\sim u^{2\lambda+1} \\ \\
\left|\left(t-z\right)\right|r^{\frac{2\lambda+2}{2\lambda+1}} \frac{r^{\frac{1}{2\lambda+1}}r^{\frac{2}{2\lambda+1}}}{r^{\frac{4}{2\lambda+1}}r^{\frac{2\lambda+2}{2\lambda+1}}}\lesssim\Big(\frac{d_{\lambda}(t,z)}{r}\Big)^{\frac{1}{2\lambda+1}}\,\ \ \hbox{for}\  |1-|u||\leq 2r^{\frac{1}{2\lambda+1}}.\\
\end{array} \right.
\end{eqnarray*}
Thus we have proved the following inequality when $u\geq0$:
\begin{eqnarray}\label{dp10}
\Big(1+\frac{d_{\lambda}(1,t)}{r}\Big)^{\frac{2\lambda+3}{2\lambda+1}}| K(r, 1, t)-K(r, 1, z)|\lesssim\Big(\frac{d_{\lambda}(t,z)}{r}\Big)^{\frac{1}{2\lambda+1}}
\end{eqnarray}
for $\frac{d_\lambda(t,
z)}{r}\leq C\min\{1+\frac{d_\lambda(1,t)}{r}, 1+\frac{d_\lambda(1,z)}{r}\} $.

$\mathbf{Condition B\ \ u\leq -3/2\,or -1/2\leq u\leq 0}$.

When $u\leq -3/2\,or -1/2\leq u\leq 0$,  for $\frac{d_\lambda(t,
z)}{r}\leq C\min\{1+\frac{d_\lambda(1, t)}{r}, 1+\frac{d_\lambda(1, z)}{r}\} $,  we will prove the following inequality:
$$\Big(1+\frac{d_{\lambda}(1, t)}{r}\Big)^{\frac{2\lambda+3}{2\lambda+1}}| K(r, 1,  t)-K(r, 1, z)|\lesssim\Big(\frac{d_{\lambda}(t, z)}{r}\Big)^{\frac{1}{2\lambda+1}}.$$
Notice that $\Big(1+\frac{d_{\lambda}(1, |u|)}{r}\Big) \sim \Big(1+\frac{d_{\lambda}(1, -u)}{r}\Big)$ when $u\leq -3/2\,or -1/2\leq u\leq 0$.
Thus by Formula\,(\ref{D-Poisson-2}), Formula\,(\ref{equ111}) and Mean value theorems for definite integrals, we could obtain:
\begin{eqnarray}\label{dp1o}
& &\Big(1+\frac{d_{\lambda}(1, t)}{r}\Big)^{\frac{2\lambda+3}{2\lambda+1}}| K(r, 1, t)-K(r, 1, z)|\nonumber\\
&\lesssim&\left|\Big(1+\frac{d_{\lambda}(1, |u|)}{r}\Big)^{\frac{2\lambda+3}{2\lambda+1}}\int_{-1}^{1}\frac{ry|u-s|
}{\big(y^2+1+u^2-2us\big)^{\lambda+2}}(1-s^2)^{\lambda-1}(1+s)
ds\right|\left|\left(t-z\right)\right|.
\end{eqnarray}
Notice that the following inequality hold for $-1\leq s\leq 1$:
\begin{eqnarray}\label{dp2o}
\left|\frac{u+1}{(y^2+1+u^2-2us)}\right|<\left|\frac{u+1}{(y^2+1+u^2+2u)}\right|.
\end{eqnarray}
For $0\leq s\leq 1$, we have:
\begin{eqnarray}\label{dp3o}
\left|\frac{1}{(y^2+1+u^2-2us)}\right|\lesssim\frac{1}{(y^2+1+u^2)}.
\end{eqnarray}
For $-1\leq s\leq 0$, we have:
\begin{eqnarray}\label{dp4o}
\left|\frac{1+s}{(y^2+1+u^2-2us)}\right|\lesssim\frac{1}{(y^2+1+u^2)}.
\end{eqnarray}
From Formula\,(\ref{dp1o}) Formula\,(\ref{dp2o}) Formula\,(\ref{dp3o}) Formula\,(\ref{dp4o}) and Formula\,(\ref{Poisson-ker-3}), we could obtain
\begin{eqnarray}
& &\left|\int_{-1}^{1}\frac{ry|u-s|
}{\big(y^2+1+u^2-2us\big)^{\lambda+2}}(1-s^2)^{\lambda-1}(1+s)
ds\left(t-z\right)\right| \nonumber\\ \nonumber
&\leq&\left(\left| \int_{-1}^{1}\frac{ry|u+1|(1-s^2)^{\lambda-1}(1+s)
}{\big(y^2+1+u^2-2us\big)^{\lambda+2}}
ds\right| +\left| \int_{-1}^{1}\frac{ry(1-s^2)^{\lambda-1}(1+s)^2
}{\big(y^2+1+u^2-2us\big)^{\lambda+2}}
ds\right|\right)\left|\left(t-z\right)\right|\\
&\lesssim&\left|\frac{|u+1|
}{\big(y^2+1+u^2-2|u|\big)}\left|r(\tau_{1}P_y)(-u)\right|\left(t-z\right)\right|+\left|\frac{1}{\big(y^2+1+u^2\big)}\left|r(\tau_{1}P_y)(-u)\right|\left(t-z\right)\right| \label{so1}\\ &\leq& C\left|\left(t-z\right)yr\frac{(1-|u|)^2+y^2+(1+u^2+y^2)|1-|u||}{((1-|u|)^2+y^2)(1+u^2+y^2)^{\lambda+2}}\right| \label{dp9o}.
\end{eqnarray}
From Formula\,(\ref{dp9o}), similar to the case $\mathbf{Condition A\ \ u\geq 0}$, we could deduce the following inequality:
\begin{eqnarray}\label{dp10o}
& &\Big(1+\frac{d_{\lambda}(1, t)}{r}\Big)^{\frac{2\lambda+3}{2\lambda+1}}| K(r, 1, t)-K(r, 1, z)|\nonumber\\
&\lesssim&\left|\Big(1+\frac{d_{\lambda}(1, |u|)}{r}\Big)^{\frac{2\lambda+3}{2\lambda+1}}\int_{-1}^{1}\frac{ry|u-s|
}{\big(y^2+1+u^2-2us\big)^{\lambda+2}}(1-s^2)^{\lambda-1}(1-s)
ds\right|\left|\left(t-z\right)\right|
\nonumber\\ &\lesssim&\left|\Big(1+\frac{d_{\lambda}(1, |u|)}{r}\Big)^{\frac{2\lambda+3}{2\lambda+1}} \left(t-z\right)yr\frac{(1-|u|)^2+y^2+(1+u^2+y^2)|1-|u||}{((1-|u|)^2+y^2)^1(1+u^2+y^2)^{\lambda+2}}\right|
\nonumber \\ &\lesssim&\Big(\frac{d_{\lambda}(t,z)}{r}\Big)^{\frac{1}{2\lambda+1}}.
\end{eqnarray}
Thus we have proved the following inequality when $u\leq -3/2\,or -1/2\leq u\leq 0$:
\begin{eqnarray}\label{dp10o0}
\Big(1+\frac{d_{\lambda}(1,t)}{r}\Big)^{\frac{2\lambda+3}{2\lambda+1}}| K(r, 1, t)-K(r, 1, z)|\lesssim\Big(\frac{d_{\lambda}(t,z)}{r}\Big)^{\frac{1}{2\lambda+1}}
\end{eqnarray}
for $\frac{d_\lambda(t,
z)}{r}\leq C\min\{1+\frac{d_\lambda(1,t)}{r}, 1+\frac{d_\lambda(1,z)}{r}\} $.

$\mathbf{Condition C\ \ -3/2 \leq u\leq -1/2}$

Notice that $d_{\lambda}(1, u)\sim 1$ and $ 1+ \frac{d_\lambda(1, u)}{r} \sim 1+ \frac{d_\lambda(1, t)}{r}\sim1+\frac{d_\lambda(1, z)}{r}$ for $\frac{d_\lambda(t,
z)}{r}\leq C\min\{1+\frac{d_\lambda(1,t)}{r}, 1+\frac{d_\lambda(1,z)}{r}\} $. Thus by Formula\,(\ref{dp9o}), we could deduce that:
\begin{eqnarray}\label{dp1oo}
& &\Big(1+\frac{d_{\lambda}(1, t)}{r}\Big)^{\frac{2\lambda+3}{2\lambda+1}}| K(r, 1, t)-K(r, 1, z)|\nonumber\\
& \lesssim &\Big(1+\frac{d_{\lambda}(1, u)}{r}\Big)^{\frac{2\lambda+3}{2\lambda+1}}| K(r, 1, t)-K(r, 1, z)|\nonumber\\
&\lesssim&\left|\Big(1+\frac{1}{r}\Big)^{\frac{2\lambda+3}{2\lambda+1}}yr\frac{(1-|u|)^2+y^2+(1+u^2+y^2)|1-|u||}{((1-|u|)^2+y^2)(1+u^2+y^2)^{\lambda+2}}\right|\left|\left(t-z\right)\right|.
\end{eqnarray}

$\mathbf{Condition C_1}$:  When $r>1$, we could deduce that $y=r^{\frac{1}{2\lambda+1}}$. By Formula\,(\ref{dp8}) and Formula\,(\ref{dp1oo}), we could deduce that
\begin{eqnarray}\label{dp2oo}
& &\Big(1+\frac{d_{\lambda}(1, t)}{r}\Big)^{\frac{2\lambda+3}{2\lambda+1}}| K(r, 1, t)-K(r, 1, z)|\nonumber\\
&\lesssim&\left|\Big(1+\frac{1}{r}\Big)^{\frac{2\lambda+3}{2\lambda+1}}yr\frac{(1-|u|)^2+y^2+(1+u^2+y^2)|1-|u||}{((1-|u|)^2+y^2)(1+u^2+y^2)^{\lambda+2}}\right|\left|\left(t-z\right)\right|\nonumber\\
&\lesssim&\frac{|t-z|}{r^{\frac{2}{2\lambda+1}}}\nonumber\\
&\lesssim&\frac{\left(d_\lambda(t, z)\right)^{\frac{1}{2\lambda+1}}}{r^{\frac{1}{2\lambda+1}}}.
\end{eqnarray}

$\mathbf{Condition C_2}$: When $0<r\leq1$ and $|t-z|\geq1/4$, we could deduce that $y=r$ and $d_\lambda(t, z)\geq C$ for some constant. Also it is clear that $$\frac{\left(\frac{1}{r}\right)^{\frac{1}{2\lambda+1}}}{\left(1+\frac{1}{r}\right)^{\frac{1}{2\lambda+1}}}\sim 1.$$
Thus from the above Formula\,(\ref{dp3oo}), we could deduce that
\begin{eqnarray}\label{dp2ooo}
|K(r, 1, t)-K(r, 1, z)|
&\lesssim&\Big(1+\frac{d_{\lambda}(1,t)}{r}\Big)^{-\frac{2(\lambda+1)}{2\lambda+1}}\nonumber\\
&\lesssim&\frac{\left(\frac{1}{r}\right)^{\frac{1}{2\lambda+1}}}{\left(1+\frac{1}{r}\right)^{\frac{1}{2\lambda+1}}}\Big(1+\frac{d_{\lambda}(1,t)}{r}\Big)^{-\frac{2(\lambda+1)}{2\lambda+1}}d_\lambda(t,z)^{\frac{1}{2\lambda+1}}\nonumber\\
&\lesssim&\frac{\left(\frac{d_\lambda(t,z)}{r}\right)^{\frac{1}{2\lambda+1}}}{\left(1+\frac{d_\lambda(1,t)}{r}\right)^{\frac{2\lambda+3}{2\lambda+1}}}.
\end{eqnarray}

$\mathbf{Condition C_3}$: When $0<r\leq1$ and $r/4\leq|t-z|\leq1/4$, with the fact that $-3/2 \leq u\leq -1/2$ we could deduce that $y=r$ and $d_\lambda(t, z)\sim |t-z|$. Thus  it is clear that
$$1<\left(\frac{d_\lambda(t, z)}{r}\right)^{\gamma_\lambda}.$$
And we could also deduce that $d_\lambda(1, t)\sim d_\lambda(1, u)\sim d_\lambda(1, z)\sim1$. Thus from the above Formula\,(\ref{Poisson-ker-3}), we could obtain:
\begin{eqnarray}\label{dp4oo}
|K(r, 1, t)-K(r, 1, z)|
&\lesssim&\Big(1+\frac{d_{\lambda}(1,t)}{r}\Big)^{-\frac{2(\lambda+1)}{2\lambda+1}}\nonumber\\
&\lesssim&\left(\frac{d_\lambda(t, z)}{r}\right)^{\gamma_\lambda}\Big(1+\frac{d_{\lambda}(1,t)}{r}\Big)^{-\frac{2(\lambda+1)}{2\lambda+1}}\nonumber\\
&\lesssim&\left(\frac{d_\lambda(t, z)}{r}\right)^{\gamma_\lambda}\Big(1+\frac{d_{\lambda}(1,t)}{r}\Big)^{-1-2\gamma_\lambda}.
\end{eqnarray}

$\mathbf{Condition C_4}$: When $0<r\leq1$ and $|t-z|\leq r/4$, with the fact that $-3/2 \leq u\leq -1/2$ we could deduce that $y=r$ and $d_\lambda(t, z)\sim |t-z|\lesssim r$. It is clear that
$$\frac{d_\lambda(t, z)}{r}\leq\left(\frac{d_\lambda(t, z)}{r}\right)^{\gamma_\lambda}.$$
Thus by Formula\,(\ref{so1}) and Formula\,(\ref{Poisson-ker-3}), we could deduce that:
\begin{eqnarray*}
& &\Big(1+\frac{d_{\lambda}(1, t)}{r}\Big)^{\frac{2\lambda+2}{2\lambda+1}}| K(r, 1, t)-K(r, 1, z)|\nonumber\\
&\lesssim&\left|\left(t-z\right)\right|\frac{r\ln r}{r^{\frac{2\lambda+2}{2\lambda+1}}}\\
&\lesssim&\frac{d_\lambda(t, z)}{r}\leq\left(\frac{d_\lambda(t, z)}{r}\right)^{\gamma_\lambda}.
\end{eqnarray*}
Thus we could obtain that
\begin{eqnarray}\label{dp5oo}
|K(r, 1, t)-K(r, 1, z)|\lesssim\left(\frac{d_\lambda(t, z)}{r}\right)^{\gamma_\lambda}\Big(1+\frac{d_{\lambda}(1,t)}{r}\Big)^{-1-2\gamma_\lambda}.
\end{eqnarray}
Notice that
$$\Big(1+\frac{d_{\lambda}(1, t)}{r}\Big)^{-1}\frac{d_\lambda(t, z)}{r}\leq\left(\Big(1+\frac{d_{\lambda}(1, t)}{r}\Big)^{-1}\frac{d_\lambda(t, z)}{r}\right)^{\gamma_\lambda}.$$

Thus from Formula\,(\ref{dpo}),  Formula\,(\ref{dp10}),  Formula\,(\ref{dp2oo}),  Formula\,(\ref{dp2ooo}),  Formula\,(\ref{dp4oo}),  Formula\,(\ref{dp5oo})  and Formula\,(\ref{dp10o0}), we could deduce that   for $\frac{d_\lambda(t,
z)}{r}\leq C\min\{1+\frac{d_\lambda(1,t)}{r}, 1+\frac{d_\lambda(1,z)}{r}\} $, the following inequality holds:
$$| K(r, 1, t)-K(r, 1, z)|\lesssim \Big(\frac{d_{\lambda}(t,z)}{r}\Big)^{\gamma_\lambda}\Big(1+\frac{d_{\lambda}(1,t)}{r}\Big)^{-1-2\gamma_\lambda}.$$
This proves the Theorem.
\end{proof}

\begin{proposition}\label{kernel2}
For any $\phi\in S(\RR, dx)$, where $\phi$ is an even function,

{\rm(i)}  \ $ |r\tau_x\phi_y(-t)|\lesssim
\Big(1+\frac{d_{\lambda}(x,t)}{r}\Big)^{-1-\gamma_\lambda}, \ \hbox{for} \ \ r>0, x,
t\in\mathbb R$;

{\rm(ii)} \ For $r>0, x, t,z\in\mathbb R$, if $\frac{d_\lambda(t,
z)}{r}\leq C\min\{1+\frac{d_\lambda(x,t)}{r}, 1+\frac{d_\lambda(x,z)}{r}\} $
$$|r\tau_x\phi_y(-t)-r\tau_x\phi_y(-z)|\lesssim \Big(\frac{d_{\lambda}(t,z)}{r}\Big)^{\gamma_\lambda}\Big(1+\frac{d_{\lambda}(x,t)}{r}\Big)^{-1-2\gamma_\lambda};$$
{\rm(iii)} \ $$r\tau_x\phi_y(-z)=r\tau_z\phi_y(-x).$$
$y$ has the representation
\begin{equation*} y  =\left\{ \begin{array}{cc}
                             r|x|^{-2\lambda} & \hbox{for} \ \  0<r<|x|^{2\lambda+1}, \\
                             r^{1/(2\lambda+1)} & \hbox{for}\ \  r\ge
                             |x|^{2\lambda+1}.
                           \end{array}\right.
\end{equation*}
\end{proposition}
\begin{proof}
When $\phi$ is even, we could write $r\tau_x\phi_y(-t)$ as:
\begin{eqnarray*}
r\tau_x\phi_y(-t)&=&c'_\lambda\int_0^\pi \frac{r}{y^{2\lambda+1}}\phi\left(\frac{\sqrt{x^2+t^2-2|xt|\cos\theta}}{y}\right)\left(1+sgn(xt)\cos\theta\right)\sin^{2\lambda-1}\theta d\theta
\\&=&c'_\lambda\int_{-1}^{1}\frac{r}{y^{2\lambda+1}}\phi\left(\frac{\sqrt{x^2+t^2-2xts}}{y}\right)(1+s)^{\lambda}(1-s)^{\lambda-1}ds ,
\\ & &\hbox{where}\,c'_\lambda=\frac{\Gamma(\lambda+(1/2))}{\Gamma(\lambda)\Gamma(1/2)}.
\end{eqnarray*}
Thus it is clear that the following holds
$$|r\tau_x\phi_y(-t)|\lesssim |r(\tau_xP_y)(-t)|, \ \ r\tau_x\phi_y(-z)=r\tau_z\phi_y(-x),$$
then we could deduce {\rm(ii)} and {\rm(iii)} of the Proposition. Next we will prove  {\rm(i)} of the Proposition. Similar to Theorem\,\ref{kkk},
we will only consider the cases for $x=0$ and $x=1$. \\
Case 1 \ \ \  When  x=0, we suppose $t, z>0$ first. Notice that $\phi'$ is odd and by the mean value theorem we could deduce that:
\begin{eqnarray*}
|r\phi_y(-t)-r\phi_y(-z)|&=&\frac{r}{y^{2\lambda+1}}\left|\phi\left(\frac{t}{y}\right)-\phi\left(\frac{z}{y}\right)\right|
\\&=& \frac{r}{y^{2\lambda+2}}\left|\phi'\left(\frac{\xi}{y}\right)\right|\left|t-z\right|
\\&\lesssim&r\frac{y|\xi|}{(y^2+\xi^2)^{\lambda+2}}\left|t-z\right|.
\end{eqnarray*}
Then by Theorem\,\ref{kkk}, we could obtain:
$$|r\phi_y(-t)-r\phi_y(-z)|\lesssim\Big(\frac{d_{\lambda}(t,z)}{r}\Big)^{\gamma_\lambda}\Big(1+\frac{d_{\lambda}(0,t)}{r}\Big)^{-1-2\gamma_\lambda}.$$
Case 2 \ \ \  When  x=1,  by the mean value theorem we could deduce that:
\begin{eqnarray}\label{3s}
& &|r\tau_1\phi_y(-t)-r\tau_1\phi_y(-z)|\\&=&\nonumber\left|c'_\lambda \frac{r}{y^{2\lambda+1}}\int_{-1}^{1}\left(\phi\left(\frac{\sqrt{1+t^2-2ts}}{y}\right)-\phi\left(\frac{\sqrt{1+z^2-2zs}}{y}\right)\right)(1-s^2)^{\lambda-1}(1+s)
ds\right||t-z|
\\&=&\nonumber\left|c'_\lambda \frac{r}{y^{2\lambda+2}}\int_{-1}^{1}\phi^{(1)}\left(\frac{\sqrt{\xi^2+1-2\xi s}}{y}\right)\frac{\xi-s}{\sqrt{\xi^2+1-2\xi s}}(1-s^2)^{\lambda-1}(1+s)
ds\right||t-z|.
\end{eqnarray}
Notice that $\phi^{(1)}$ is an odd function and $\phi^{(1)}\in S(\RR, dx)$, thus we could deduce the following:
\begin{eqnarray}\label{4s}
\left|\frac{\phi^{(1)}\left(\frac{\sqrt{\xi^2+1-2\xi s}}{y}\right)}{\frac{\sqrt{\xi^2+1-2\xi s}}{y}}\left(\frac{y^2+1+\xi^2-2\xi s}{y^2}\right)^{\lambda+2}\right|\lesssim 1.
\end{eqnarray}
Thus Formula\,(\ref{3s}) and Formula\,(\ref{4s}) lead to:
\begin{eqnarray}\label{5s}
& &|r\tau_1\phi_y(-t)-r\tau_1\phi_y(-z)|
\\&\lesssim&\nonumber \left|\int_{-1}^{1}\frac{ry|\xi-s|
}{\big(y^2+1+\xi^2-2\xi s\big)^{\lambda+2}}(1-s^2)^{\lambda-1}(1+s)
ds\right||t-z|.
\end{eqnarray}
Thus we could obtain {\rm(ii)} of this Proposition by   Formula\,(\ref{5s}) and Theorem\,\ref{kkk} for the case  $x=1$.
This proves the Proposition.
\end{proof}

\begin{proposition}\label{111}Let $B(x_0,r_0)$ satisfying $x_0>0$ and  $r_0^{\frac{1}{2\lambda+1}}<|x_0/2|$ be the ball in the homogeneous type space: $B(x_0, r_0)=\{y: d_{\lambda}(y, x_0)<r_0\}$,    $I_0$ the Euclidean interval: $I_0=(x_0-\delta_2, x_0+\delta_1)=B(x_0,r_0)$. For any $t \in  B(x_0,r_0)$,
the following inequalities hold: $$\delta_1<r_0^{\frac{1}{2\lambda+1}}<|x_0/2|,\ \ \ \  \ \delta_2<r_0^{\frac{1}{2\lambda+1}}<|x_0/2|,$$
$$|x_0|\sim|s|\ \ \hbox{for\,any}\  s\in B(x_0,r_0), \ \ \ \ \ \delta_1\sim\delta_2\sim\frac{r_0}{x_0^{2\lambda}}.$$
\end{proposition}
\begin{proof}
When $r_0^{\frac{1}{2\lambda+1}}<|x_0/2|,$ it is easy to see that:   $$ |x_0|\sim|s|\ \ \hbox{for\,any}\ s\in B(x_0,r_0) .$$
We could see that in fact  $\delta_1$ and $\delta_2$ have the representation:
$$\delta_2=\left|\left(x_0^{2\lambda+1}- r_0 \right)^{\frac{1}{2\lambda+1}}-x_0\right|,\ \ \ \ \delta_1=\left|\left(x_0^{2\lambda+1}+ r_0 \right)^{\frac{1}{2\lambda+1}}-x_0\right|.$$

With the fact that
$$\left|y-x\right|^{2\lambda+1}<\left|y^{2\lambda+1}-x^{2\lambda+1}\right|$$
holds for $x, y>0$, it is easy to see that $\delta_1\leq r_0^{\frac{1}{2\lambda+1}}$ and $\delta_2\leq r_0^{\frac{1}{2\lambda+1}}$.  By Taylor expansion near the origin, for $r_0^{\frac{1}{2\lambda+1}}<|x_0/2|,$  we could obtain that
$$\left|\left(x_0^{2\lambda+1}\pm r_0 \right)^{\frac{1}{2\lambda+1}}-x_0\right|\sim x_0\left|\left(1\pm \frac{r_0}{x_0^{2\lambda+1}} \right)^{\frac{1}{2\lambda+1}}-1\right|\sim\frac{r_0}{x^{2\lambda}_0}.$$
Therefore:
$$\delta_1\sim\delta_2\sim\frac{r_0}{x^{2\lambda}_0}.$$
This proves the proposition.
\end{proof}

\begin{proposition}\label{ex2}
Let $B(x_0,r_0)$ satisfying $x_0>0$ and  $r_0^{\frac{1}{2\lambda+1}}<|x_0/2|$ be the ball in the homogeneous type space: $B(x_0, r_0)=\{y: d_{\lambda}(y, x_0)<r_0\}$,    $I(x_0, t)$  be the Euclid interval: $I(x_0, t)=(x_0-t, x_0+t)$.
There exists constants $c_1>0$ and $c_2>0$ independent on $x_0$ and $r_0$, such that the following  holds: $$ I(x_0, c_2\frac{r_0}{x_0^{2\lambda}})\subseteq B(x_0,r_0)\subseteq I(x_0, c_1\frac{r_0}{x_0^{2\lambda}}).$$
And the following holds:
$$B(x_0,r_0)\subseteq I(x_0, r_0^{\frac{1}{2\lambda+1}}). $$
\end{proposition}
\begin{proof}
Notice that the following inequality holds when $x>0$ and $y>0$:$$\left|y-x\right|<\left|y^{2\lambda+1}-x^{2\lambda+1}\right|^{\frac{1}{2\lambda+1}}.$$
Then we could obtain $B(x_0,r_0)\subseteq I(x_0, r_0^{\frac{1}{2\lambda+1}}).$
By Proposition\,\ref{111}, we could obtain that $$\max_{y, x \in B(x_0,r_0)}|y-x|\sim \frac{r_0}{x^{2\lambda}_0} .$$
Therefore there are  constants $c_1>0$ and $c_2>0$ independent on $x_0$ and $r_0$, such that $$ I(x_0, c_2\frac{r_0}{x_0^{2\lambda}})\subseteq B(x_0,r_0)\subseteq I(x_0, c_1\frac{r_0}{x_0^{2\lambda}}).$$
Hence the Proposition holds.
\end{proof}

\begin{proposition}\label{exexex1}
For any fixed $\phi\in S(\RR, dx)$, where $\phi$ is an even function with $\,supp\phi\subseteq[-1,1],$ $0\leq\phi\leq1$, $\phi(0)=1$, then we could obtain the following:

{\rm(i)}  \ $ 0<r\tau_x\phi_y(-t)\lesssim
\Big(1+\frac{d_{\lambda}(x,t)}{r}\Big)^{-1-\gamma_\lambda}, \ \hbox{for} \ \ r>0, x,
t\in\mathbb R$;

{\rm(ii)} \ For $r>0, x, t,z\in\mathbb R$, if $\frac{d_\lambda(t,
z)}{r}\leq C\min\{1+\frac{d_\lambda(x,t)}{r}, 1+\frac{d_\lambda(x,z)}{r}\} $
$$|r\tau_x\phi_y(-t)-r\tau_x\phi_y(-z)|\lesssim \Big(\frac{d_{\lambda}(t,z)}{r}\Big)^{\gamma_\lambda}\Big(1+\frac{d_{\lambda}(x,t)}{r}\Big)^{-1-2\gamma_\lambda};$$

{\rm(iii)} \ $$r\tau_x\phi_y(-z)=r\tau_z\phi_y(-x);$$

{\rm(iv)}\ \ $|r\tau_x\phi_y(-x)|\sim1$;

{\rm(v)}$\,supp\, r\tau_x\phi_y(-t)\subseteq B(x, cr)\bigcup B(-x, cr) $, where $c$ is constant independent on $r, x, y, t$. There exists
a constant $C_0<\frac{1}{2^{2\lambda+1}} $, such that $B(x, cr)\bigcap B(-x, cr)=\emptyset$ for $0<y<C_0 |x|$;

$y$ has the representation
\begin{equation*} y  =\left\{ \begin{array}{cc}
                             r|x|^{-2\lambda} & \hbox{for} \ \  0<r<|x|^{2\lambda+1}, \\
                             r^{1/(2\lambda+1)} & \hbox{for}\ \  r\ge
                             |x|^{2\lambda+1}.
                           \end{array}\right.
\end{equation*}

\end{proposition}

\begin{proof}
{\rm(i)}, {\rm(ii)}, {\rm(iii)} and {\rm(iv)} of the Proposition could be deduced from Proposition\,\ref{kernel2}.
We will prove ${\rm(iv)}$ next, then we need to consider the cases for $x=0$ and $x=1$. It is clear that $|r\tau_x\phi_y(-x)|_{x=0}|=\phi(0)\sim1$ for the case $x=0$. When $0<y<1=x=t$, $r=y$, we could deduce that for some fixed $0<\delta<1$,
 the following holds:
\begin{eqnarray*}
\left|r\tau_1\phi_y(-1)\right| &=& \left|\int_{-1}^{1} c'_\lambda\frac{r}{y^{2\lambda+1}}\phi\left(\frac{\sqrt{2-2s}}{y}\right)(1-s^2)^{\lambda-1}(1+s)
ds\right|
\\ &\geq&\nonumber \left|\int_{1-\frac{\delta y^2}{100}}^{1} c'_\lambda \frac{r}{y^{2\lambda+1}}\phi\left(\frac{\sqrt{2-2s}}{y}\right)(1-s^2)^{\lambda-1}(1+s)
ds\right|
\\ &\geq& C_{\delta}.
\end{eqnarray*}

When $y>1=x=t$, $r=y^{2\lambda+1}$, we could deduce the following inequality:
\begin{eqnarray*}
\left|r\tau_1\phi_y(-1)\right| &=& \left|\int_{-1}^{1} c'_\lambda\frac{r}{y^{2\lambda+1}}\phi\left(\frac{\sqrt{2-2s}}{y}\right)(1-s^2)^{\lambda-1}(1+s)
ds\right|
\\ &\geq&\nonumber \left|\int_{1/4}^{1} c'_\lambda \frac{r}{y^{2\lambda+1}}\phi\left(\frac{\sqrt{2-2s}}{y}\right)(1-s^2)^{\lambda-1}(1+s)
ds\right|
\\ &\geq& C.
\end{eqnarray*}
Thus {\rm(iv)} of this Proposition holds. We will prove {\rm(v)} of this Proposition at last.

For $x,t,z\in\RR$, we use $W_{\lambda}(x,t,z)$ to denote as:
$W_{\lambda}(x,t,z)=W_{\lambda}^0(x,t,z)(1-\sigma_{x,t,z}+\sigma_{z,x,t}+\sigma_{z,t,x})$,
where
$$
W_{\lambda}^0(x,t,z)=\frac{c''_{\lambda}|xtz|^{1-2\lambda}\chi_{(||x|-|t||,
|x|+|t|)}(|z|)} {[((|x|+|t|)^2-z^2)(z^2-(|x|-|t|)^2)]^{1-\lambda}},
$$
$c''_{\lambda}=2^{3/2-\lambda}\big(\Gamma(\lambda+1/2)\big)^2/[\sqrt{\pi}\,\Gamma(\lambda)]$.
And $\sigma_{x,t,z}=\frac{x^2+t^2-z^2}{2xt}$, for $x\neq 0$ and $t\neq0$. $\sigma_{x,t,z}=0$, for $x= 0$ or $t=0$.
For $t\neq0$, we could write $(\tau_x\phi)(-t)$
\begin{eqnarray}\label{tau-1}
(\tau_x\phi)(-t)=c_{\lambda}\int_{\RR}\phi(z)W_\lambda(-t,x,z)|z|^{2\lambda}dz.
\end{eqnarray}
It is clear that $\tau_x\phi_y(-t)=0$ when $\left|\frac{|x|-|t|}{y}\right|\geq1$.
Thus the function $t\rightarrow\tau_x\phi_y(-t)$ satisfies $supp\,\tau_x\phi_y(-t)\subseteq \left(|x|-|y|,|x|+|y|\right)\bigcup\left(-|x|-|y|, -|x|+|y|\right)$.

Case  1 \ \ \ \ When $y\geq \frac{|x|}{2^{2\lambda+1}}$, it is clear that $r\sim y^{2\lambda+1}\gtrsim|x|^{2\lambda+1}$. Notice that
$supp\,r\tau_x\phi_y(-t)\subseteq \left(|x|-|y|,|x|+|y|\right)\bigcup\left(-|x|-|y|, -|x|+|y|\right)$, thus we could
deduce that there exists constants $c$ and $c_1$ that is independent on $r, x, y, \lambda$ such that: $supp\,r\tau_x\phi_y(-t) \subseteq B(0, cr)\subseteq B(x, c_1r)$ and $supp\,r\tau_x\phi_y(-t) \subseteq B(0, cr)\subseteq B(-x, c_1r)$.

Case  2 \ \ \ \ When $0<y<\frac{|x|}{2^{2\lambda+1}}$, we could see that $r=y|x|^{2\lambda}<\frac{|x|^{2\lambda+1}}{2^{2\lambda+1}}$. Thus, by
Proposition\,\ref{ex2}, we could deduce that the function $t\rightarrow\tau_x\phi_y(-t)$ satisfies $supp\,r\tau_x\phi_y(-t)\subseteq \left(|x|-|y|,|x|+|y|\right)\bigcup\left(-|x|-|y|, -|x|+|y|\right)\subseteq B(x, cr)\bigcup B(-x, cr)$. Also, it is clear that by
Proposition\,\ref{ex2}, there exists a constant $C_0<\frac{1}{2^{2\lambda+1}} $ such that $B(x, cr)\bigcap B(-x, cr)=\emptyset$ when $0<y<C_0 |x|$.  This proves {\rm(v)} of this Proposition.
\end{proof}

\begin{proposition}\label{exexex2}
For any fixed $\phi\in S(\RR, dx)$, where $\phi$ is an even function with $\,supp\phi\subseteq[-1,1],$ $0\leq\phi\leq1$, $\phi(0)=1$,  we use $K_3(r , x, t)$ to denote as:
$$K_3(r , x,  t)= r\tau_x\phi_y(-t)-r\tau_x\phi_y(t),\ \ \hbox{for}\ x\neq0$$
where $y$ has the representation
\begin{equation*} y  =\left\{ \begin{array}{cc}
                              r|x|^{-2\lambda} \ & \hbox{for} \ \  0<r<|x|^{2\lambda+1}, \\
                              0<y<C_0 |x| \  & ( C_0\  \hbox{is\  the\  constant\  in\  Proposition}\,\ref{exexex1} ) \hbox{and}\  x\neq0.
                           \end{array}\right.
\end{equation*}

Then we could obtain the following:

{\rm(i)}  \ $ \left|K_3(r , x,  t)\right|\lesssim
\Big(1+\frac{d_{\lambda}(x,t)}{r}\Big)^{-1-\gamma_\lambda}, \ \hbox{for} \ \ r>0, x,
t\in\mathbb R$;

{\rm(ii)} \ For $r>0, x, t,z\in\mathbb R$, if $\frac{d_\lambda(t,
z)}{r}\leq C\min\{1+\frac{d_\lambda(x,t)}{r}, 1+\frac{d_\lambda(x,z)}{r}\} $
$$|K_3(r , x,  t)-K_3(r , x,  z)|\lesssim \Big(\frac{d_{\lambda}(t,z)}{r}\Big)^{\gamma_\lambda}\Big(1+\frac{d_{\lambda}(x,t)}{r}\Big)^{-1-2\gamma_\lambda};$$

{\rm(iii)} \ $$K_3(r , x,  t)=K_3(r , t,  x);$$

{\rm(iv)}\ \  $ K_3(r , x,  x) \sim 1$ and $K_3(r , x,  t)=-K_3(r , x,  -t)$;

{\rm(v)}$\,supp\, K_3(r , x,  t)\subseteq B(x, cr)\bigcup B(-x, cr)$ with $B(x, cr)\bigcap \{x=0\} =\emptyset$, where $c$ is a constant independent on $r, x, y, t$;

{\rm(vi)} $ 0<K_3(r , x,  t)\leq C$ when $x>0$, and   $ -C \leq K_3(r , x,  t)< 0$ when $x<0$ for some constant $C$ independent on $r, x, t$.
\end{proposition}

\begin{proof}
{\rm(i)} {\rm(ii)} and {\rm(v)} of this Proposition can be deduced from Proposition\,\ref{exexex1} directly. Notice that we could write $K_3(r , x,  t)$ as following:
\begin{eqnarray*}
K_3(r , x,  t)=\int_{-1}^1c'_\lambda\frac{r}{y^{2\lambda+1}}\phi\left(\frac{\sqrt{x^2+t^2-2|xt|s}}{y}\right)2sgn(xt)(1-s^2)^{\lambda-1}s
ds.
\end{eqnarray*}
Thus we could deduce {\rm(iii)} of this Proposition. We will prove {\rm(iv)} of this Proposition at last. From\,(v)\, we could deduce that $sgn(xt)>0$, thus we could write $K_3(r , x,  t)$ as:
\begin{eqnarray*}
K_3(r , x,  t)=\int_{0}^1c'_\lambda\frac{r}{y^{2\lambda+1}}\left(\phi\left(\frac{\sqrt{x^2+t^2-2|xt|s}}{y}\right)-\phi\left(\frac{\sqrt{x^2+t^2+2|xt|s}}{y}\right)\right)2(1-s^2)^{\lambda-1}s
ds.
\end{eqnarray*}
We will prove ${\rm(iv)}$ of this Proposition next, then we need to consider the cases for $x=1$. When $0<y<C_0<1=x=t$, $r=y$, we could deduce that for some fixed $0<\delta<1$, the following holds:
\begin{eqnarray*}
K_3(r , 1,  1) &\geq& \int_{1-\frac{\delta y^2}{100}}^1c'_\lambda\frac{r}{y^{2\lambda+1}}\left(\phi\left(\frac{\sqrt{2-2s}}{y}\right)-\phi\left(\frac{\sqrt{2+2s}}{y}\right)\right)2(1-s^2)^{\lambda-1}s
ds
\\ &\geq& C_{\delta}.
\end{eqnarray*}
 Also it  is clear that $K_3(r , x,  t)$ is an odd function in $t$, thus $K_3(r , -1,  -1)\sim -1$. Thus we obtain {\rm(iv)} of this Proposition. Thus we could also deduce {\rm(vi)} of this Proposition. This proves the Proposition.
\end{proof}
In a similar way, we could obtain the following Proposition:
\begin{proposition}\label{exexex5}
For any fixed $\phi\in S(\RR, dx)$, where $\phi$ is an even function with $\,supp\phi\subseteq[-1,1],$ $0\leq\phi\leq1$, $\phi(0)=1$,  we use $K_4(r , x, t)$ to denote as:
$$K_4(r , x,  t)= r\tau_x\phi_y(-t)+r\tau_x\phi_y(t),\ \ \hbox{for}\ x\neq0,$$
where $y$ has the representation
\begin{equation*} y  =\left\{ \begin{array}{cc}
                              r|x|^{-2\lambda} \ & \hbox{for} \ \  0<r<|x|^{2\lambda+1}, \\
                              0<y<C_0 |x| \  & ( C_0\  \hbox{is\  the\  constant\  in\  Proposition}\,\ref{exexex1} ) \hbox{and}\  x\neq0.
                           \end{array}\right.
\end{equation*}

Then the following holds:

{\rm(i)}  \ $ \left|K_4(r , x,  t)\right|\lesssim
\Big(1+\frac{d_{\lambda}(x,t)}{r}\Big)^{-1-\gamma_\lambda}, \ \hbox{for} \ \ r>0, x,
t\in\mathbb R$;

{\rm(ii)} \ For $r>0, x, t,z\in\mathbb R$, if $\frac{d_\lambda(t,
z)}{r}\leq C\min\{1+\frac{d_\lambda(x,t)}{r}, 1+\frac{d_\lambda(x,z)}{r}\} $
$$|K_4(r , x,  t)-K_4(r , x,  z)|\lesssim \Big(\frac{d_{\lambda}(t,z)}{r}\Big)^{\gamma_\lambda}\Big(1+\frac{d_{\lambda}(x,t)}{r}\Big)^{-1-2\gamma_\lambda};$$

{\rm(iii)} \ $$K_4(r , x,  t)=K_4(r , t,  x);$$

{\rm(iv)}\ \  $ K_4(r , x,  x) \sim 1$ and $K_4(r , x,  t)=K_4(r , x,  -t)$;

{\rm(v)}$\,supp\, K_4(r , x,  t)\subseteq B(x, cr)\bigcup B(-x, cr)$ with $B(x, cr)\bigcap \{x=0\} =\emptyset$, where $c$ is a constant independent on $r, x, y, t$;

{\rm(vi)} $ 0<K_4(r , x,  t)\leq C$.
\end{proposition}

\begin{proposition}\label{exexex4}
We use $F_{\nabla}(x)$ to denote as $F_{\nabla}(x)=\sup_{|x-u|<y}\left|F(u, y)\right|$, $F_{\nabla\lambda}(x)$ to denote as $F_{\nabla\lambda}(x)=\sup_{d_{\lambda}(x, u)<r}\left|F(u, y)\right|$, where $y$ has the representation
\begin{equation}\label{utut1} y  =\left\{ \begin{array}{cc}
                             r|x|^{-2\lambda} & \hbox{for} \ \  0<r<|x|^{2\lambda+1}, \\
                             r^{1/(2\lambda+1)} & \hbox{for}\ \  r\ge
                             |x|^{2\lambda+1}.
                           \end{array}\right.
\end{equation}  Then we could have:
\begin{eqnarray}\label{tan1}
\|F_{\nabla\lambda}\|_{L^p_\lambda(\RR)}\sim_{\lambda, p}\|F_{\nabla}\|_{L^p_\lambda(\RR)}.
\end{eqnarray}
\end{proposition}
We also use $F_{+}(x)$ to denote as $F_{+}(x)=\sup_{y>0}\left|F(x, y)\right|$, $F_{+\lambda}(x)$ to denote as $F_{+\lambda}(x)=\sup_{r>0}\left|F(x, y)\right|$. Thus it is clear that $F_{+}(x)=F_{+\lambda}(x)$.

\begin{proof}

Case 1: When $0<y<\frac{|x|}{2^{2\lambda+1}}$, by Proposition\,\ref{ex2}  we could deduce that for some constants $c_1$ and $c_2$  $$ I(x, c_2\frac{r}{x^{2\lambda}})\subseteq B(x,r)\subseteq I(x, c_1\frac{r}{x^{2\lambda}}).$$  Thus we could deduce that
\begin{eqnarray}\label{tanoo}
I(x, c_2y)\subseteq B(x,r)\subseteq I(x, c_1y).
\end{eqnarray}

Case 2: When $y\geq\frac{|x|}{2^{2\lambda+1}}$, it is clear that $r\sim y^{2\lambda+1}$. Then we could see that there exists $c_1$ and $c_2$ independent on $x, r, y$, such that
\begin{eqnarray}\label{tanoo0}
I(x, c_2y)\subseteq B(x,r)\subseteq I(x, c_1y).
\end{eqnarray}
 Then by Formulas\,(\ref{tanoo})\,and\,(\ref{tanoo0}), together with Proposition\,\ref{tan}, we could deduce that Formula\,(\ref{tan1}) holds.
This proves the Proposition.
\end{proof}
We use $(f*_{\lambda}\phi)_{\nabla\lambda}(x)$, $(f*_{\lambda}\phi)_{\nabla}(x)$ and $(f*_{\lambda}\phi)_{+}(x)$ to denote as following:
\begin{eqnarray*}
(f*_{\lambda}\phi)_{\nabla\lambda}(x)= \sup_{ d_\lambda(u, x)< r} \left|f*_\lambda\phi_y(u)\right|,\ \ \ (f*_{\lambda}\phi)_{\nabla}(x)=  \sup_{ |x-u|< y}  \left|f*_\lambda\phi_y(u)\right|,
\end{eqnarray*}
\begin{eqnarray*}
(f*_{\lambda}\phi)_{+}(x)=  \sup_{ y>0} \left| f*_\lambda\phi_y(x)\right|,
\end{eqnarray*}
where $y$ has the representation as Formula\,(\ref{utut1}) and $\displaystyle{\phi_y(x)=\frac{1}{y^{2\lambda+1}}\phi\left(\frac{x}{y}\right)}.$

\begin{theorem}\label{exexex3}
For any fixed $\phi\in S(\RR, dx)$, where $\phi$ is an even function with $supp\,\phi\subseteq[-1,1],$ $0\leq\phi\leq1$, $\phi(0)=1$, we could deduce that for $f\in L_{\lambda}^1(\RR)$:
\begin{eqnarray}\label{tan2}
\|f_{S\beta}^*\|_{L^p_\lambda(\RR)}\sim_{\lambda, p, \beta, \phi}\|(f*_{\lambda}\phi)_{\nabla}\|_{L^p_\lambda(\RR)}\sim_{\lambda, p, \beta, \phi} \|(f*_{\lambda}\phi)_{+}\|_{L^p_\lambda(\RR)},
\end{eqnarray}
for $p>\frac{1}{1+\gamma_\lambda}$, for some\,$\beta>0$.
\end{theorem}
\begin{proof}
We use $f_o$ and $f_e$ to denote as: $$f_o(x)=\frac{f(x)-f(-x)}{2},\ \ \ f_e(x)=\frac{f(x)+f(-x)}{2}.$$
We use $ \widetilde{K}(r, x, t), \widetilde{K}_o(r, x, t), \widetilde{K}_e(r, x, t)$ to denote as:
\begin{eqnarray*}
& &\widetilde{K}(r, x, t)=r\tau_x\phi_y(-t),\\
& &2\widetilde{K}_o(r, x, t)=r\tau_x\phi_y(-t)-r\tau_x\phi_y(t),\\
& &2\widetilde{K}_e(r, x, t)=r\tau_x\phi_y(-t)+r\tau_x\phi_y(t),
\end{eqnarray*}
where $y$ has the representation
\begin{equation*} y  =\left\{ \begin{array}{cc}
                             r|x|^{-2\lambda} & \hbox{for} \ \  0<r<|x|^{2\lambda+1}, \\
                             r^{1/(2\lambda+1)} & \hbox{for}\ \  r\ge
                             |x|^{2\lambda+1}.
                           \end{array}\right.
\end{equation*}
One obvious fact is that the following two Formulas hold:
\begin{eqnarray}\label{cc1}
\|f_{S\beta}^*\|_{L^p_\lambda(\RR)}\lesssim\|(f_o)_{S\beta}^*\|_{L^p_\lambda(\RR)}+\|(f_e)_{S\beta}^*\|_{L^p_\lambda(\RR)}\lesssim\|f_{S\beta}^*\|_{L^p_\lambda(\RR)},
\end{eqnarray}
\begin{eqnarray}\label{cc2}
\|(f*_{\lambda}\phi)_{\nabla\lambda}\|_{L^p_\lambda(\RR)}\lesssim\|((f_o)*_{\lambda}\phi)_{\nabla\lambda}\|_{L^p_\lambda(\RR)}+\|((f_e)*_{\lambda}\phi)_{\nabla\lambda}\|_{L^p_\lambda(\RR)}\lesssim\|(f*_{\lambda}\phi)_{\nabla\lambda}\|_{L^p_\lambda(\RR)}.
\end{eqnarray}
Next, we will define new kernels as follows\,($C_0$ is the constant  in  Proposition\,\ref{exexex1}):

Case1:  $x>0$
\begin{equation*} K_o(r, x, t)  =\left\{ \begin{array}{cc}
                              \widetilde{K}(r, x, t) \ & \hbox{for} \ \  y\geq C_0 |x|, \\
                              \widetilde{K}_o(r, x, t)\chi_{(0, +\infty)}(t)\  & \hbox{for}\ \  0<y<C_0 |x|,
                           \end{array}\right.
\end{equation*}

\begin{equation*} K_e(r, x, t)  =\left\{ \begin{array}{cc}
                              \widetilde{K}(r, x, t) \ & \hbox{for} \ \  y\geq C_0 |x|, \\
                              \widetilde{K}_e(r, x, t)\chi_{(0, +\infty)}(t)\  & \hbox{for}\ \  0<y<C_0 |x|,
                           \end{array}\right.
\end{equation*}

Case2:  $x<0$
\begin{equation*} K_o(r, x, t)  =\left\{ \begin{array}{cc}
                              \widetilde{K}(r, x, t) \ & \hbox{for} \ \  y\geq C_0 |x|, \\
                              \widetilde{K}_o(r, x, t)\chi_{(-\infty, 0)}(t)\  & \hbox{for}\ \  0<y<C_0 |x|,
                           \end{array}\right.
\end{equation*}

\begin{equation*} K_e(r, x, t)  =\left\{ \begin{array}{cc}
                              \widetilde{K}(r, x, t) \ & \hbox{for} \ \  y\geq C_0 |x|, \\
                              \widetilde{K}_e(r, x, t)\chi_{(-\infty, 0)}(t)\  & \hbox{for}\ \  0<y<C_0 |x|,
                           \end{array}\right.
\end{equation*}

Case3:  $x=0$
\begin{equation*}
K_o(r, x, t)=K_e(r, x, t)=\widetilde{K}(r, x, t).
\end{equation*}
Thus we could see that the following two Formulas hold:
\begin{eqnarray}\label{exee1}
((f_o)*_{\lambda}\phi)_{\nabla\lambda}(x)\sim  \sup_{ d_\lambda(u, x)< r} \left|\int_\RR K_o(r, u, t)f_o(t)|t|^{2\lambda}dt/r\right|,
\end{eqnarray}

\begin{eqnarray}\label{exee2}
((f_e)*_{\lambda}\phi)_{\nabla\lambda}(x)\sim  \sup_{ d_\lambda(u, x)< r} \left|\int_\RR K_e(r, u, t)f_e(t)|t|^{2\lambda}dt/r\right|.
\end{eqnarray}

By Proposition\,\ref{exexex1}, Proposition\,\ref{exexex2}, Proposition\,\ref{exexex5}, we could deduce that $K_o(r, x, t)$ and $K_e(r, x, t)$ are just the kind of kernel $K_1(r, x, t)$ with compact support in Section\,1:\, Theorem\,\ref{important}. Thus by Formula\,(\ref{exee1}),  Formula\,(\ref{exee2}), and Theorem\,\ref{important},   we could deduce the following:
\begin{eqnarray}\label{exee3}
\|(f_o)_{S\beta}^*\|_{L^p_\lambda(\RR)}\sim \|((f_o)*_{\lambda}\phi)_{\nabla\lambda}\|_{L^p_\lambda(\RR)},
\end{eqnarray}

\begin{eqnarray}\label{exee4}
\|(f_e)_{S\beta}^*\|_{L^p_\lambda(\RR)}\sim \|((f_e)*_{\lambda}\phi)_{\nabla\lambda}\|_{L^p_\lambda(\RR)}.
\end{eqnarray}
Thus from Formula\,(\ref{cc1}) Formula\,(\ref{cc2}) Formula\,(\ref{exee3}) Formula\,(\ref{exee4}) and Proposition\,\ref{exexex4}, we could prove the theorem.
\end{proof}

\begin{proposition}\label{exexex6}
For $p>\frac{1}{1+\gamma_\lambda}$, $\phi$ is an even function with $supp\,\phi\subseteq[-1,1]$, $0\leq\phi\leq1$, $\phi(0)=1$, $\psi$ is an even function, $\int_\RR\psi(t)|t|^{2\lambda}dt\sim1$ with $\phi, \psi\,\in S(\RR, dx)$, then we could deduce the following for $f\in L^1_{\lambda}(\RR)$:
\begin{eqnarray}\label{tan3}
\|(f*_{\lambda}\phi)_{+}\|_{L^p_\lambda(\RR)}\sim_{\lambda, p,  \phi, \psi}\|(f*_{\lambda}\psi)_{\nabla}\|_{L^p_\lambda(\RR)} \sim_{\lambda, p, \phi, \psi} \|(f*_{\lambda}\psi)_{+}\|_{L^p_\lambda(\RR)}.
\end{eqnarray}
\end{proposition}

\begin{proof}
Fix a function $\varphi\in S(\RR, dx)$ so that:
\begin{eqnarray*}
\left\{ \begin{array}{cc}
\varphi(\xi)=0 \ \ \hbox{for}\,|\xi|\geq1 \\
\\
\varphi(\xi)=1\ \ \hbox{for}\,|\xi|\leq1/2,
\end{array}\right.
\end{eqnarray*}
where $\varphi$ is an even function. Then  $\varphi^k\in S(\RR, dx)$ can be defined as:
\begin{eqnarray*}
\left\{ \begin{array}{cc}
\varphi^k(\xi)=\varphi(\xi) \ \ \hbox{for}\,k=0, \\
\\
\varphi^k(\xi)=\varphi(2^{-k}\xi)-\varphi(2^{1-k}\xi)\ \ \hbox{for}\,k\geq1.
\end{array}\right.
\end{eqnarray*}
By Proposition\,\ref{uu5}\,and\,\ref{u0}, we could deduce that $\sup_{\xi\in\RR}\left| |\xi|^\beta \partial_\xi^\alpha(\SF_\lambda\psi)(\xi)\right|\lesssim C_{ \beta, \alpha}$, when $\psi(t)\in S(\RR, dx)$. Thus together with the fact that
$(\SF_\lambda\psi)(0)\sim 1$, we could deduce that there exists a $k_o$, such that
$$\left|(\SF_\lambda\psi)(2^{-k_o}\xi)\right|\gtrsim 1/2 \ \ \hbox{for}\,\, |\xi|\leq2.$$

We use $\eta^{k,\lambda}$ to denote as
$$(\SF_\lambda\eta^{k,\lambda})(\xi)=\frac{\varphi^k(\xi)(\SF_\lambda\phi)(\xi)}{(\SF_\lambda\psi)(2^{-k}2^{-k_o}\xi)},$$
where $\SF_\lambda$ denotes the Dunkl transform.

Then
\begin{eqnarray}\label{uue5}
\phi(x)=\sum_{k=0}^{+\infty}\eta^{k,\lambda}*_{\lambda}\psi_{2^{-k-k_o}}(x).
\end{eqnarray}

By the fact that $\displaystyle{\sup_{\xi\in\RR}\left|D^{\beta}(\SF_\lambda\psi)(\xi)\right|\lesssim_{ \beta} 1}$ and $\displaystyle{\sup_{\xi\in\RR}\left|\xi^{\alpha}D^{\beta}(\SF_\lambda\phi)(\xi)\right|\lesssim_{\alpha, \beta} 1}$, where $D$ is the Dunkl
operator, we could deduce that for any $M>0$
\begin{eqnarray}\label{uue1}
\sup_{\xi\in\RR}\left|\xi^{\alpha}D^{\beta}(\SF_\lambda\eta^{k,\lambda})(\xi)\right|\lesssim_{\alpha, \beta, M, k_o} 2^{-M}.
\end{eqnarray}
Thus we could deduce that
\begin{eqnarray}\label{uue3}
 \left|\int_\RR\eta^{k,\lambda}(x)\left(1+2^{k+k_o}|x| \right)^{N}|x|^{2\lambda}dx\right|\leq C2^{-k}.
\end{eqnarray}

By Formula\,(\ref{uue3}), we could deduce that
\begin{eqnarray}\label{uue4}
\sum_{k=0}^{+\infty} \left|\int_\RR\eta^{k,\lambda}\left(\frac{x}{t}\right)\left(1+\frac{|x|}{2^{-k-k_o}t} \right)^{N}|x|^{2\lambda}\frac{dx}{t^{2\lambda+1}}\right|\leq C_{k_o, N} \sum_{k=0}^{\infty} 2^{-k}.
\end{eqnarray}

Then by Formula\,(\ref{uue5}) and Formula\,(\ref{uue3}), we could deduce the following:
\begin{eqnarray}\label{uue2}
\sup_{t>0}\left| f*_\lambda\phi_t(x)\right| &=& \sup_{t>0}\left|\sum_{k=0}^{+\infty} f*_\lambda\eta_t^{k,\lambda}*_{\lambda}\psi_{2^{-k-k_o}t}(x)\right|
\\&\leq&\nonumber\sup_{t>0}\left|\sum_{k=0}^{+\infty}\int \tau_{-u}\left(f*_{\lambda}\psi_{2^{-k-k_o}t}\right)(x)\eta_t^{k,\lambda}(u)|u|^{2\lambda}du\right|
\\&\leq&\nonumber \sup_{t>0, u\in\RR}\left|\tau_{-u}\left(f*_{\lambda}\psi_{t}\right)(x)\left(1+\frac{|u|}{t} \right)^{-N}\right| \sum_{k=0}^{+\infty}\left|\int \eta^{k,\lambda}\left(\frac{u}{t}\right)\left(1+\frac{|u|}{2^{-k-k_o}t} \right)^{N}|u|^{2\lambda}\frac{du}{t^{2\lambda+1}}\right|
\\&\lesssim&\nonumber       \sum_{m=0}^{+\infty} \sup_{t>0, 2^{m-1}t<|u|\leq2^mt}2^{-mN}\left|\tau_{-u}\left(f*_{\lambda}\psi_{t}\right)(x)\right|+\sup_{t>0, |u|\leq t}\left|\tau_{-u}\left(f*_{\lambda}\psi_{t}\right)(x)\right|
\\&\lesssim&\nonumber       \sum_{m=0}^{+\infty} \sup_{t>0, |u|\leq2^mt}2^{-mN}\left|\tau_{-u}\left(f*_{\lambda}\psi_{t}\right)(x)\right|.
\end{eqnarray}

For $x\neq0$, we could write $\tau_{-u}\left(f*_{\lambda}\psi_{t}\right)(x)$ as
\begin{eqnarray}\label{tau-2}
\tau_{-u}\left(f*_{\lambda}\psi_{t}\right)(x)=c_{\lambda}\int_{\RR}\left(f*_{\lambda}\psi_{t}\right)(z)W_\lambda(x, -u, z)|z|^{2\lambda}dz.
\end{eqnarray}

For $x=0$, we could write $\tau_{u}\left(f*_{\lambda}\psi_{t}\right)(0)$ as
\begin{eqnarray}\label{tau-3}
\tau_{u}\left(f*_{\lambda}\psi_{t}\right)(0)= \left(f*_{\lambda}\psi_{t}\right)(u)
\end{eqnarray}

Notice that $ \left||x|-|u|\right|\leq|z|\leq|x|+|u|$, thus by Formula\,(\ref{uue2}) Formula\,(\ref{tau-2}) and Formula\,(\ref{tau-3}) with the fact that $\int_{\RR}\left|W_\lambda(x, -u, z)\right||z|^{2\lambda}dz\leq 4$, we could deduce that:

\begin{eqnarray}\label{uue6}
\sup_{t>0}\left| f*_\lambda\phi_t(x)\right|
\lesssim     \sum_{m=0}^{+\infty} \left|\sup_{|z-x|\leq2^mt}2^{-mN}\left(f*_{\lambda}\psi_{t}\right)(z)\right|+\sum_{m=0}^{+\infty} \left|\sup_{|z+x|\leq2^mt}2^{-mN}\left(f*_{\lambda}\psi_{t}\right)(z)\right|.
\end{eqnarray}

Thus Proposition\,\ref{tan} and Formula\,(\ref{uue6}) lead to the following inequality for $N>\frac{1}{p}$:

\begin{eqnarray}\label{tan4}
\|(f*_{\lambda}\phi)_{+}\|_{L^p_\lambda(\RR)}\lesssim\|(f*_{\lambda}\psi)_{\nabla}\|_{L^p_\lambda(\RR)}.
\end{eqnarray}

Proposition\,\ref{exexex4}   Proposition\,\ref{kernel2}  and  Proposition\,\ref{im1} lead to

\begin{eqnarray}\label{tan5}
\|(f*_{\lambda}\psi)_{\nabla}\|_{L^p_\lambda(\RR)}\lesssim\|f_{\gamma_\lambda}^*\|_{L^p_\lambda(\RR)}.
\end{eqnarray}
Formula\,(\ref{tan4}) Proposition\,\ref{exexex4} Proposition\,\ref{important1} and Theorem\,\ref{exexex3} lead to the following:

\begin{eqnarray}\label{tan6}
\|f_{\gamma_\lambda}^*\|_{L^p_\lambda(\RR)}\lesssim \|(f*_{\lambda}\phi)_{\nabla}\|_{L^p_\lambda(\RR)}.
\end{eqnarray}

 Formula\,(\ref{tan4}) Formula\,(\ref{tan5})  Formula\,(\ref{tan6}) Proposition\,\ref{no11}  and Theorem\,\ref{exexex3} lead to Formula\,(\ref{tan3}). This proves the Proposition.
\end{proof}

\begin{theorem}[$H_{\lambda}^p(\RR)$, $\widetilde{H}_{\lambda}^p(\RR)$  for $ p>\frac{1}{1+\gamma_\lambda}$.]\label{tan12}
For $ p>\frac1{1+\gamma_\lambda}$, $f(x)\in H_{\mu_\lambda}^p(\RR)$. Let $\gamma_\lambda=\frac{1}{2(2\lambda+1)}$, then we could obtain:
\begin{eqnarray}\label{tan7}
\|f_{\gamma_\lambda}^*\|_{L^p_\lambda(\RR)}\sim \|P_{\nabla}^*f\|_{L^p_\lambda(\RR)}.
\end{eqnarray}
Thus $\widetilde{H}_{\lambda}^p(\RR)$ and $H_{\lambda}^p(\RR)$ can be defined as follows:
$$\widetilde{H}_{\lambda}^p(\RR)=\widetilde{H}_{\mu_{\lambda}}^p(\RR)=\left\{g\in L_{\lambda}^2(\RR)\bigcap L_{\lambda}^1(\RR) :g_{S\gamma_\lambda}^*(x)\in L_{\lambda}^{p}(\RR) \right\}$$
$$H_{\lambda}^p(\RR)=H_{\mu_\lambda}^p(\RR)=\left\{g\in  S'(\RR, |x|^{2\lambda}dx): g_{S\gamma_\lambda}^*(x)\in L_{\lambda}^{p}(\RR) \right\}.$$

\end{theorem}
(remark: $H_{\mu_\lambda}^p(\RR)$ with the $\mu_\lambda$ measure is not $H_{\mu\,\lambda}^p(\RR)$, as in Definition\,\ref{ha}.)

\begin{proof}
Let  $f\in L^1_{\lambda}(\RR)$ first. By Proposition\,\ref{exexex4} Theorem\,\ref{kkk} and Proposition\,\ref{im1}, we could deduce that:
\begin{eqnarray}\label{tan8}
\|P_{\nabla}^*f\|_{L^p_\lambda(\RR)}\lesssim\|f_{\gamma_\lambda}^*\|_{L^p_\lambda(\RR)}.
\end{eqnarray}
Next we will prove
\begin{eqnarray}\label{tan9}
\|f_{\gamma_\lambda}^*\|_{L^p_\lambda(\RR)}\lesssim\|P_{\nabla}^*f\|_{L^p_\lambda(\RR)}.
\end{eqnarray}
Notice the $\lambda$-Poisson kernel is $\tau_xP_y(-t)$ with $P_y(x)=a_\lambda y\left(y^2+x^2\right)^{-\lambda-1}$, where $a_\lambda=2^{\lambda+1}\Gamma(\lambda+1)/\sqrt{\pi}$. We use similar idea in \cite{Stein}. There exists a function $\eta$ defined on
$[1, \infty)$  that is rapidly decreasing at $\infty$ and satisfies the moment conditions:
\begin{eqnarray}\label{tan10}
\int_1^\infty \eta(s)ds=1,\ \ \hbox{and}\ \ \int_1^\infty s^k\eta(s)ds=0,\ \hbox{for}\,k=1,2,\ldots.
\end{eqnarray}
Then  we could check that the function $\Phi(x)$  $$\Phi(x)=\int_1^{\infty}\eta(s)P_s(x)ds,$$
is rapidly decreasing and is an even function: $\Phi(x) \in S(\RR, dx)$ is even. Also it is clear that
$$\int \Phi(x)|x|^{2\lambda}dx=C \int_{1}^{\infty}\eta(s)ds\sim 1.$$
Thus we could deduce that:
\begin{eqnarray}\label{tan11}
\left(f*_{\lambda}\Phi_{y}\right)_{+}(x)&=&\sup_{y>0}\left|\int f(t)\tau_{-t}\Phi_{y}(x)|t|^{2\lambda}dt\right|
\\&=&\nonumber\sup_{y>0}\left|\int \tau_{-t}f(x)\Phi_{y}(t)|t|^{2\lambda}dt\right|
\\&\lesssim&\nonumber\sup_{y>0}\left|\int \tau_{-t}f(x)\int_1^{\infty}\eta(s)P_{ys}(t)ds|t|^{2\lambda}dt\right|
\\&\lesssim&\nonumber\sup_{y>0}\left|\int \tau_{-t}f(x)\int_1^{\infty}\eta(s)P_{ys}(t)ds|t|^{2\lambda}dt\right|
\\&\lesssim&\nonumber P_{+}^*f(x)
\\&\lesssim&\nonumber P_{\nabla}^*f(x).
\end{eqnarray}
Thus the above Formula\,(\ref{tan11}),  Proposition\,\ref{important1}, Theorem\,\ref{exexex3}, and Proposition\,\ref{exexex6}, we could deduce Formula\,(\ref{tan9}). Thus Formula\,(\ref{tan7}) holds for $f\in L^1_{\lambda}(\RR)$. Notice that $\widetilde{H}_{\mu_{\lambda}}^p(\RR)$ is dense
in $H_{\mu_\lambda}^p(\RR)$. Thus by the Hahn-Banach theorem, we could deduce that Formula\,(\ref{tan7}) holds for $f\in H_{\mu_\lambda}^p(\RR)$.
Thus together with Theorem\,\ref{important}, $\widetilde{H}_{\lambda}^p(\RR)$ and $H_{\lambda}^p(\RR)$ can be defined as follows:
$$\widetilde{H}_{\lambda}^p(\RR)=\widetilde{H}_{\mu_{\lambda}}^p(\RR)=\left\{g\in L_{\lambda}^2(\RR)\bigcap L_{\lambda}^1(\RR) :g_{S\gamma_\lambda}^*(x)\in L_{\lambda}^{p}(\RR) \right\}$$
$$H_{\lambda}^p(\RR)=H_{\mu_\lambda}^p(\RR)=\left\{g\in  S'(\RR, |x|^{2\lambda}dx): g_{S\gamma_\lambda}^*(x)\in L_{\lambda}^{p}(\RR) \right\},$$
where $\displaystyle{\gamma_\lambda=\frac{1}{2(2\lambda+1)}}$.
This proves the Theorem.
\end{proof}
Thus we could obtain the following Proposition:
\begin{proposition}\label{s3}
$u(x, y)$ is a $\lambda$-harmonic function, for $1 \geq p>\frac{1}{1+\gamma_\lambda}$\\
case1, $u_{\nabla}^\ast(x) \in L_{\lambda}^p(\RR)\bigcap L_{\lambda}^2(\RR)\bigcap L_{\lambda}^1(\RR)$, then there exists $f\in\widetilde{H}_{\lambda}^p(\RR)$, such that
\begin{eqnarray}\label{tan13}
u(x, y)= f*_\lambda P_y(x).
\end{eqnarray}
case2, $u_{\nabla}^\ast(x) \in L_{\lambda}^p(\RR)$, then there exists $f\in H_{\lambda}^p(\RR)$, such that
\begin{eqnarray}\label{tan14}
\displaystyle{\int \sup_{|x-s|<y}\bigg|u(s,y)-f*_\lambda P_y(s)\bigg|^p|x|^{2\lambda}dx=0,}
\end{eqnarray}
moreover, $$\|u_{\nabla}^\ast\|_{L_{\lambda}^p(\RR)}\sim \|f\|_{H_{\lambda}^p(\RR)}.$$

\end{proposition}
\begin{proof}
By Proposition\,\ref{sss1}(4), we could deduce Formula\,(\ref{tan13}). By Theorem\,\ref{u2}(2), Proposition\,\ref{p4}, Formula\,(\ref{tan13}), together with the fact that $\widetilde{H}_{\lambda}^p(\RR)$ is dense in $H_{\lambda}^p(\RR)$, we could deduce that Formula\,(\ref{tan14}) holds.
This proves the Proposition.
\end{proof}

\end{document}